\documentclass{article}

\usepackage{textcomp}

\usepackage[utf8]{inputenc} 
\usepackage[T1]{fontenc}    
\usepackage{hyperref}       
\usepackage{url}            
\usepackage{booktabs}       
\usepackage{amsfonts}       
\usepackage{nicefrac}       
\usepackage{lipsum}

\usepackage{mathtools,tensor}
\usepackage{latexsym,amssymb,amsthm,amsmath,enumerate,epsfig,setspace,bbding,color,graphicx,caption}
\usepackage{stmaryrd}
\usepackage[all]{xy}
\usepackage{hyperref}
\usepackage{subfiles}
\usepackage[margin=10pt,font=small,labelfont=bf]{caption}
\usepackage{tikz-cd}
\usepackage{tkz-graph}
\usetikzlibrary{arrows}
\usepackage{verbatim}
\usepackage{amsmath}
\usepackage{amssymb}
\usepackage{amsfonts}
\usepackage{todonotes}
\usepackage{mathtools}
\usepackage{tikz}
\usepackage{cancel}

\usepackage{enumitem}
\setlist[itemize]{itemsep=0cm}
\setlist[enumerate]{itemsep=0cm}

\setenumerate[0]{label=(\roman*)}

\usepackage[backend=bibtex, style=numeric-comp,maxbibnames=4]{biblatex} 
\usepackage{textcomp}

\newcommand{\G}{\Sigma}
\newcommand{\CC}{\mathcal{C}}
\newcommand{\F}{\mathcal{F}}
\newcommand{\Gaug}{\mathrm{Gauge}}
\newcommand{\g}{\mathfrak{g}}
\newcommand{\h}{\mathfrak{h}}
\renewcommand{\k}{\mathfrak{k}}

\newcommand{\z}{\mathfrak{z}}
\newcommand{\RR}{\mathbb{R}}
\newcommand{\NN}{\mathbb{N}}
\newcommand{\tto}{\rightrightarrows}

\newcommand{\BB}{\mathbf{X}}
\newcommand{\then}{\Rightarrow}
\newcommand{\gl}{\mathfrak{gl}}

\newtheorem{definition}{Definition}[section]
\newtheorem{theorem}[definition]{Theorem}
\newtheorem{proposition}[definition]{Proposition}
\newtheorem{lemma}[definition]{Lemma}
\newtheorem{corollary}[definition]{Corollary}

\newtheorem*{introtheorem}{Theorem}

\theoremstyle{definition}

\newtheorem{remarkx}[definition]{Remark}            
\newenvironment{remark}
  {\pushQED{\qed}\remarkx}
  {\popQED\endremarkx}
  
\newtheorem{examplex}[definition]{Example}          
\newenvironment{example}
  {\pushQED{\qed}\examplex}
  {\popQED\endexamplex}

\newtheorem{manualtheoreminner}{Axiom}      
\newenvironment{axiom}[1]{%
  \renewcommand\themanualtheoreminner{#1}%
  \manualtheoreminner
}{\endmanualtheoreminner}

\newcommand{\Hei}{\mathbf{Hei}}
\newcommand{\dom}{\text{Dom}}
\newcommand{\Ima}{\text{Im}}

\DeclareMathOperator{\Diff}{Diff}

\DeclareMathOperator{\germ}{germ}

\DeclareMathOperator{\Hom}{Hom}

\DeclareMathOperator{\Germ}{Germ}

\DeclareMathOperator{\Lie}{Lie}

\DeclareMathOperator{\Rep}{Rep}

\DeclareMathOperator{\pr}{pr}
\DeclareMathOperator{\id}{id}
\DeclareMathOperator{\Fr}{Fr}
\DeclareMathOperator{\Ad}{Ad}
\DeclareMathOperator{\GL}{GL}
\DeclareMathOperator{\Sp}{Sp}

\begin{document}

\title{A groupoid approach to transitive differential geometry}

\author{Luca Accornero
\footnote{Departement Wiskunde, KU Leuven, Belgium, \url{luca.accornero@kuleuven.be} }
\and
Francesco Cattafi \footnote{Institut f\"ur Mathematik, Julius-Maximilians-Universit\"at W\"urzburg, Germany, \url{francesco.cattafi@mathematik.uni-wuerzburg.de} }}

\maketitle

\begin{abstract}


This work is a spin-off of an on-going programme \cite{INPROGRESS} which aims at revisiting the original studies of Lie and Cartan on pseudogroups and geometric structures from a modern perspective.

We encode geometric structures induced by transitive Lie pseudogroups into principal $G$-bundles equipped with a transversally parallelisable foliation generated by a subalgebra of $\g$, called {\it Cartan bundles}. 
Our approach is complementary to~\cite{FRANCESCOPAP} and is based on Morita equivalence of Lie groupoids.

After identifying the main examples and properties, we develop a notion of {\it flatness with respect to a Lie algebra}, which encompasses the classical integrability of $G$-structures, the flatness of Cartan geometries, as well as the integrability of contact structures.
\end{abstract}

\begin{center}
 \textbf{MSC2020}: 58H05, 58A15, 53C10
\end{center}




\begin{center}
 \textbf{Keywords}: Lie pseudogroups, Cartan geometries, $G$-structures, contact structures, transitivity, integrability
\end{center}

\tableofcontents

\section*{Introduction}

\addcontentsline{toc}{section}{Introduction}


The aim of this paper is to provide a model for {\it transitive differential geometry}, in the sense discussed by Guillemin and Sternberg~\cite{GUILLEMINSTERNBERG}. 
The main subject of our study is the notion of {\it Cartan bundle}, which has been recently introduced by the second author in \cite{FRANCESCOPAP}. Here we rediscover it with different methods, making use of the Morita theory for a class of Lie groupoids equipped with special vector-valued 1-forms. We have discussed such theory, in the transitive case, in \cite{AccorneroCattafi1}, and we apply here its main results.

Both the points of view and the tools are part of an on-going programme conceived by Marius Crainic, which aims at revisiting the original studies of Lie \cite{LieEngel88TrGroupAll} and Cartan \cite{CARTANINFINITEGROUPS} on pseudogroups and geometric structures from a modern perspective. The first output of such long-term programme was the PhD thesis of Mar\'ia Amelia Salazar in 2013 \cite{MARIA}, followed by the PhD theses of Ori Yudilevich \cite{ORI} and of the two authors \cites{FRANCESCO, LUCA}. All the main outcomes will be collected in the forthcoming monograph \cite{INPROGRESS}, where also this paper will be embedded.


\subsection*{Motivation and background}

One of our motivations for this paper was to better understand and further develop the (not very well-known) concept of {\it integrability} of a $G$-structure {\it with respect to a Lie algebra}, introduced in the 80's by Albert and Molino \cite{MOLINOINTEGR}. In order to tackle this problem, we had to zoom out and adopt a different approach to transitive geometries involving Lie groupoids.

Accordingly, our starting point goes back to the works by Lie~\cite{LieEngel88TrGroupAll} and Cartan~\cite{CARTANINFINITEGROUPS} on geometric structures with a ``homogeneity'' property: any two points can be connected by a local symmetry of the structure. In modern language, this ``homogeneity'' assumption -- called {\it transitivity} by Guillemin and Sternberg~\cite{GUILLEMINSTERNBERG} -- allows one to encode geometric structures into Lie groups and principal bundles~\cite{CHERN,STERNBERG,GUILLEMINSTERNBERG,SHARPE,ALEKSEEVSKYMICHOR}, or, infinitesimally, in terms of Lie algebras~\cite{GUILLEMINSTERNBERG,SingerSternberg65}.

On the other hand, the above-mentioned modern reformulation \cite{INPROGRESS} of Lie's and Cartan's work provides the suitable framework to describe geometric structures using Lie {\it groupoids} and principal groupoid bundles. The idea of using groupoids to encode symmetries can be traced back to Ehresmann~\cite{Ehresmann53} and Haefliger~\cite{HAEFLIGERGAMMASTRUCTURES}. However, the key novel insight from~\cite{MARIA} is the remark that structure which controls the PDEs underlying the symmetries is compatible with the groupoid structure. This led in~\cite{MARIA} to the introduction of {\it Pfaffian groupoids} -- i.e.\ Lie groupoids with a compatible ``PDE-structure''. In turn, compatible principal actions of these groupoids were successfully used in~\cite{FRANCESCO} for studying geometric structures on manifolds, by means of {\it principal Pfaffian bundles}.

A question that naturally arises is how to recover the classical group-theoretical approaches from the Pfaffian one. It is not so hard to guess an answer: given a {\it transitive} groupoid, one can look at its isotropy group at any point -- due to transitivity, the choice of point is irrelevant. All groupoid-theoretical constructions, including principal actions, can be rephrased in terms of the isotropy group. The more conceptual way of describing this ``passage to the isotropy group'' involves the classical notion of {\it Morita equivalence} -- see e.g.~\cite{MATIAS}. 


To make this reasoning complete, however, one needs to carefully keep track of the extra ``Pfaffian structure'' mentioned above. For transitive groupoids, this was done in~\cite{AccorneroCattafi1}. A more complete and thorough treatment will be available in the monograph~\cite{INPROGRESS}, but for our purposes the transitive case that we dealt with in~\cite{AccorneroCattafi1} is sufficient. 


As a result, in this paper we are able to construct models for transitive differential geometries in terms of principal Pfaffian actions of Pfaffian groups: these models are precisely the {\it Cartan bundles} introduced in~\cite{FRANCESCOPAP} and re-discovered in~\cite{AccorneroCattafi1}. Here, by ``model'' for transitive differential geometry we mean the same as the authors of~\cite{GUILLEMINSTERNBERG}: that is, an abstract machinery suited to describe the possible transitive structures up to a certain notion of equivalence. 

The main advantage of our approach is the very close relationship with the groupoid description, which not only is more general but offers interesting insights in the transitive setting as well.

\subsection*{Structure of the paper and main contributions}

In the first section we introduce the classes of geometric structures that we aim to study. In doing so, we present the modern version of Lie's ``continuous transformation groups'' of symmetries, that is {\it Lie pseudogroups} $\Gamma$, see Definition~\ref{def:Lie_pseudo}. We then describe the rich structure of the jet space $J^k\Gamma$, and encode geometric structures in terms of principal $J^k \Gamma$-bundles, following~\cite{MARIA, ORI,FRANCESCO}. 

In the short second section we provide the basics on Pfaffian groupoids~\ref{def_Pfaffian_groupoid-form}, that we need in the rest of the paper. The guiding example is given by jet groupoids of Lie pseudogroups together with their Cartan forms. We then discuss principal Pfaffian bundles~\ref{def:princ-pfaff-bundles} to encode geometric structures. We warn the reader that the material presented here is by no means complete: we refer to the PhD theses~\cite{MARIA,ORI,FRANCESCO}, and especially to the upcoming monograph~\cite{INPROGRESS}, for a broader discussion. 

We then recall some of the results presented in~\cite{AccorneroCattafi1}, namely the description of Pfaffian groups and the concept of Pfaffian isotropy, as well as the properties of Pfaffian Morita equivalence for transitive groupoids. The outcome is that studying geometric structures where the underlying pseudogroup is transitive is equivalent to studying principal Pfaffian actions of Pfaffian groups (Corollary \ref{corollary_combination_props}).

Such principal actions are called {\it Cartan bundles} (Definition \ref{def:Cartan-bundle}) and are discussed in the third section, where the novel material starts. 
In Section \ref{exm:G-struct-as-C-bundles} we start by exploring the relevant examples: $G$-structures of first and higher order, as well as Cartan geometries, which provide known approaches to transitive differential geometry~\cite{CHERN,STERNBERG,SHARPE}. We also treat fibres of transitive Pfaffian groupoids; this example recovers the original definition of Cartan bundle presented in \cite{FRANCESCOPAP}.

In Section \ref{section_geometry_of_Cartan_bundles} we prove a few fundamental properties. In particular, we show that any Cartan $G$-bundle $(P,\theta)$ over $M$ induces a tower of the kind
\[
\begin{tikzcd}
P_0=P \arrow["G_0=G"', loop, distance=2em, in=125, out=55] \arrow[r, two heads] \arrow[rrd, two heads] & P_1=P_0/K_0 \arrow[r, two heads] \arrow["G_1=G_0/K_0"', loop, distance=2em, in=125, out=55] \arrow[rd, two heads] & P_2=P_1/K_1 \arrow[r, two heads] \arrow["G_2=G_1/K_1"', loop, distance=2em, in=125, out=55] \arrow[d, two heads] & \ldots \arrow[r, two heads] & P_{i+1}=P_i/K_i \arrow["G_{i+1}=G_i/K_i"', loop, distance=2em, in=125, out=55] \arrow[lld, two heads] \\
                                                                                                       &                                                                                                                   & M                                                                                                                &                             &                                                                                                              
\end{tikzcd}
\]
We distinguish the class of {\it finite-order} Cartan bundles as those for which the sequence stabilises after a finite number of steps. As expected, the Cartan bundles arising from $k$-th order $G$-structures have order $k$.

In Section \ref{section_integrability_Cartan_bundles} we propose a new notion of {\it flatness} (Definition \ref{def_a_flatness}) with respect to the choice of a suitable {\it model} $\z$, that we call Cartan-type extension (Definition \ref{def_infinitesimal_model}), in order to discuss integrability of the underlying geometry. In giving such definition we were of course heavily inspired by the classical definition of flatness for Cartan geometries~\cite{SHARPE,ALEKSEEVSKYMICHOR} (where a model, although not strictly needed, is often taken as part of the definition). However, its applications to $G$-structures provides interesting insights. For instance, we are able to describe a certain class of non-flat $G$-structures as flat with respect to the choice of a non-abelian model: see Example~\ref{ex:h-int-GS}, which reformulates the $\h$-integrability proposed by Albert and Molino \cite{MOLINOINTEGR}. This has applications for instance to contact structures (Example~\ref{ex_integrability_contact_structures}), where the relevant model $\z$ is the Heisenberg algebra.

Strictly related to the notion of flatness is a more general notion of connection for Cartan bundle (Definition \ref{def:h-conn}), which recovers the standard principal connections adapted to first-order $G$-structures, and which plays a key role in a fundamental structure result, see Theorem~\ref{properties_lifts}.

We procede by discussing a special class of models, namely the {\it reductive} ones (Definition \ref{def:reductive}), which include all the main examples discussed so far and are completely characterised by Theorem \ref{characterisation_reductive_extensions}:

\begin{introtheorem} 
 A model is reductive if and only if it is of the form $\g \ltimes \k$.
\end{introtheorem}

As an application, using the Cartan structure equations, we arrive to a deeper understanding of flatness, e.g.\ Theorem \ref{flatness_iff_integrability}:
\begin{introtheorem} 
A $G$-structure $P$ is $\g \ltimes \k$-flat if and only if there exists a principal connection $\tau$ such that
\begin{itemize}
 \item the curvature of $\tau$ vanishes;
 \item the torsion of $\tau$ is encoded by the Lie bracket on $\k$.
\end{itemize}
\end{introtheorem}

With similar techniques (but using the more general notion of connection for Cartan bundles), we prove an analogous result for higher-order $G$-structures (Proposition \ref{flatness_iff_integrability_higher_order}) and we explain that the same could be done for any Cartan bundle of finite order.

Last, in section 4 we return briefly to transitive pseudogroups. Our philosophy is to propose that there are phenomena within the theory of pseudogroup structures that the standard description in terms of principal {\it group} bundles is not able to capture.

We consider the general setting of a Lie group $K$ that acts on $\BB$ freely and transitively, and a transitive pseudogroup $\Gamma$ on $\BB$ which contains global translations by $K$. Then we discuss the existence of natural integrability models (Proposition \ref{transitive_pseudogroup_reductive_extension}) and enlighten the structure of the Pfaffian groupoid $J^1 \Gamma$ (Proposition \ref{prop:isomorphism_source_first_jet} and \ref{prop_Pfaffian_isomorphism_first_jet}).

\subsection*{Notations and conventions}

The reader is assumed to be familiar with the basics of jet bundles and Lie groupoids; these notions are only briefly recalled in section 1 in order to introduce new objects and fix notations.

\

Throughout the paper, we use the notation $\BB$ for the spaces of objects of (Lie) groupoids and for the spaces/manifolds over which pseudogroups are defined. Base spaces of principal bundles are denoted by $M$, and principal bundles themselves are denoted by $P$.

We use the notation $\Sigma\tto \BB$ for arbitrary Lie groupoids, and $\mathcal{G}\tto \BB$ for {\it étale} groupoids (this distinction is relevant only in the first section). The Lie algebroid of a Lie groupoid $\Sigma$ is denoted by ${\rm Lie}(\Sigma)\to \BB$ or, if there is no risk of confusion, simply by $A\to \BB$.

The unit of a (Lie) group $G$ is denoted by $e$, and its Maurer-Cartan form by $\omega_{\rm MC} \in \Omega^1 (G,\g)$. Given a Lie algebra action $a: \g \to \mathfrak{X}(P)$, we denote by $\alpha^\dagger := a(\alpha)$ the fundamental vector field associated to $\alpha \in \g$.

Gothic letters $\z$, $\k$, etc. are used to denote both Lie algebras and {\it almost} Lie algebras (this distinction is relevant only starting from Section \ref{section_integrability_Cartan_bundles}).


The canonical Cartan form on the jet prolongation $J^kY$ of a submersion $Y\to X$ is denoted by $\omega^k$; the same notation is used for the restriction of $\omega^k$ to a submanifold of $J^kY$. We use $\pi^k_h$ to denote the canonical projection from $J^kY$ to $J^hY$, whenever $h<k$, or just $\pr$ when there is no danger of confusion. 

All group(oid) actions are considered from the left and all manifolds and maps are smooth, unless explicitly stated otherwise.

\subsection*{Acknowledgements}
Both the authors would like to thank Marius Crainic for starting the long-term project of which this paper is part, for several insightful discussions, and for the support during their PhD years.

The first author would like to thank \'Alvaro del Pino for offering precious insights on contact structures and homogeneous spaces. The second author would like to thank Andreas \v{C}ap for several clarifications on Cartan geometries and Marco Zambon for suggestions on the foliation in Proposition \ref{properties_ideal_k}.

The first author was partially supported by the NWO through the Utrecht Geometry Center graduate programme (The Netherlands), and by the FWO-FNRS under EOS project G0I2222N (Belgium).

The second author was partially supported by the FWO-FNRS under EOS project G0H4518N (Belgium), by the ESI under the Junior Research Fellowship ``Cartan geometries via Pfaffian groupoids'' (Austria), by the FWF under Mozart Grant I 5015-N (Austria) and by the DFG under Walter Benjamin project 460397678 (Germany), and is a member of the GNSAGA - INdAM (Italy).

\section{Preliminaries and motivation: transitive geometries}\label{sec:first-section}


\subsection{Geometry on manifolds}
In broad terms, by ``geometric structure'' on a manifold $M$ one means some differential geometric data on $M$ (other than its smooth structure). A more precise definition depends, obviously, on the examples one has in mind/theoretical properties one wants to capture. Here, we will look at manifolds that locally look like $\mathbb{R}^n$ {\it equipped with additional structure}. In other words, we are interested in structures that
\begin{itemize}
\item can be ``transported'' to $M$ from analogous structures on $\mathbb{R}^n$ using that $M$ is locally modelled on $\mathbb{R}^n$;
\item are defined by systems of (very) well behaved differential equations.
\end{itemize}

\subsubsection{Local models}\label{sss:local-models}
Let $M$ be an $n$-dimensional manifold. The smooth structure on $M$ is defined by means of equivalence classes of atlases $(\{U_i\}_{i\in I}, \{\phi_i\}_{i\in I})$ with transition functions $\{\varphi_{ij}\}_{i, j\in I}$. We look at geometric structures defined by requiring the $\varphi_{ij}$'s to belong to a {\it set of symmetries} of some geometry on $\mathbb{R}^n$.
\begin{definition}
A {\bf pseudogroup} $\Gamma$ on a manifold $\BB$ is a subset of the set of smooth embeddings of open sets of $\BB$ into $\BB$ which is 
\begin{itemize}
\item closed under the group-like operations, i.e.\ composition (when defined) and inversion, and containing the identity $\id_\BB:\BB\to \BB$;
\item local, i.e.\ such that restrictions of elements in $\Gamma$ belong to $\Gamma$;
\item closed under gluings, i.e.\ if $\{V_k\}_{k\in K}$ covers an open $U$ and $\varphi_k\in \Gamma$ are the restrictions of an embedding $\varphi:U\hookrightarrow \BB$ to the opens $V_k$'s, then $\varphi\in \Gamma$.
\end{itemize}

\end{definition}
\begin{definition}\label{def:transitive-orbits}
Let $\Gamma$ be a pseudogroup on a manifold $\BB$. The orbit $\mathcal{O}_x$ of $x\in \BB$ is the set of $y\in \BB$ for which there exists $\varphi\in \Gamma$ such that $\varphi(x)=y$.

A pseudogroup on $\BB$ is called {\bf transitive}  if  for all $x,y\in \BB$ there is $\varphi\in \Gamma$ such that $\varphi(x)=y$, i.e.\ the only orbit is $\BB$ itself. 
\end{definition}

If $\Gamma$ is a pseudogroup on $\mathbb{R}^n$, one calls {\bf $\Gamma$-atlas} on $M$ an atlas whose transition functions are elements of $\Gamma$. A {\bf $\Gamma$-structure} is an equivalence class of $\Gamma$-atlases/a maximal $\Gamma$-atlas. The same definitions make sense when $\Gamma$ is defined on {\it any} manifold $\BB$. {\it If $\Gamma$ is transitive}, by restricting to elements of $\Gamma$ defined in a sufficently small open $U\subset \BB$ and with image in $U$ itself, one sees that there is no loss of generality in considering $\Gamma$-structures for pseudogroups on $\mathbb{R}^n$.

\begin{example}\label{ex:Gamma-n}
The simplest example of pseudogroup on $\mathbb{R}^n$ is $\Gamma^n:={\rm Diff}_{\rm loc}(\mathbb{R}^n)$, the set of all locally defined diffeomorphisms of $\mathbb{R}^n$. It can be thought of as the set of symmetries of the standard smooth structure of $\mathbb{R}^n$. A $\Gamma^n$-structure on a manifold $M$ is precisely a smooth structure on $M$; in fact, a $\Gamma^n$-atlas on $M$ is simply an ordinary atlas on $M$.
\end{example}

In general, the notion of pseudogroup is meant to capture the idea of set of symmetries, and concrete examples of pseudogroups arise by considering symmetries of known geometric objects.

\begin{example}\label{ex:G-struct}
Any Lie subgroup of linear transformations $G\subset \GL(n, \mathbb{R})$ generates a pseudogroup $\Gamma_G$ on $\mathbb{R}^n$, given by those embeddings $\varphi:U\hookrightarrow \mathbb{R}^n$ whose differential at any point $x\in U$ lies in $G$. Notice that $\Gamma_G$ contains all translations; consequently, it is transitive. 

On the other hand, $\Gamma_G$ can (and should!) be interpreted as the set of (locally defined) {\bf symmetries of a linear geometric structure} on $\mathbb{R}^n$.  Consider the set $\mathcal{S}_G$ of frames in the vector space $\mathbb{R}^n$ that can be obtained from the standard basis via a transformation by $G$.
For each $x\in \mathbb{R}^n$, let us denote by $T_x\mathbb{R}^n\cong \mathbb{R}^n$ the isomorphism induced by the translation
\[
y\in \mathbb{R}^n \mapsto y+x \in \mathbb{R}^n.
\]
Any map $\varphi\in \Gamma_G$ sends frames of $T_x\mathbb{R}^n\cong \mathbb{R}^n$, $x\in \dom(\varphi)$, lying in $\mathcal{S}_G$ to frames of $T_{\varphi(x)}\mathbb{R}^n\cong \mathbb{R}^n$ lying in $\mathcal{S}_G$.

By specifying $G$, one has more explicit examples; we list some of them below.
\begin{enumerate}
\item If $G=O(n)$, then we are considering orthogonal frames of $\mathbb{R}^n$. The pseudogroup $\Gamma_{O(n)}$ is the set $\Gamma_{\rm eucl}$ of locally defined isometries of the standard Euclidean metric on $\mathbb{R}^n$.
\item If $n=2k$, $G=\GL(k, \mathbb{C})$, then we are considering complex frames of $\mathbb{C}^k\cong \mathbb{R}^{2k}$. The pseudogroup $\Gamma_{\GL(k, \mathbb{C})}$ is the set $\Gamma_{\rm cpx}$ of locally defined biholomorphisms of the standard complex structure on $\mathbb{C}^k$.
\item If $n=2k$, $G=\Sp(k)$, then we are considering symplectic frames of $\mathbb{R}^{2k}$. The pseudogroup $\Gamma_{\Sp(k)}$ is the set $\Gamma_{\rm sp}$ of locally defined symplectomorphisms of the standard symplectic structure on $\mathbb{R}^{2k}$. 
\item If $G=\GL(n-k, k)$ is the group of invertible matrices whose lower left $k \times (n-k)$ block is the null matrix, then we are considering frames of $\mathbb{R}^n$ whose first $(n-k)$ vectors are tangent to the subspace $\mathbb{R}^{n-k}\hookrightarrow \mathbb{R}^n$ given by the first $(n-k)$-coordinates. The pseudogroup $\Gamma_{\GL(n-k, k)}$ is the set of locally defined symmetries of the standard codimension $k$ foliation on $\mathbb{R}^n$ by copies of $\mathbb{R}^k$.
\item If $n=2k+1$, $G=\Sp(k, 1)$ is the group of invertible matrices whose lower left $1 \times 2k$ block is the null matrix and whose top right $2k\times 2k$ block lies in $\Sp(k)$, then we are considering frames of $\mathbb{R}^n$ whose first $2k$ vectors are tangent to the subspace $\mathbb{R}^{2k}\hookrightarrow \mathbb{R}^n$ given by the first $2k$-coordinates and form a symplectic frame on $\mathbb{R}^{2k}$. The pseudogroup $\Gamma_{\Sp(k, 1)}$ is the set of locally defined symmetries of the standard codimension $1$ symplectic foliation whose leaves are copies of $\mathbb{R}^{2k}$ with the canonical symplectic structure.
\end{enumerate} 
The $\Gamma$-structures on a manifold $M$ associated to the above pseudogroups are, in order:
\begin{enumerate}
\item flat Riemannian structures;
\item complex structures;
\item symplectic structures;
\item codimension $k$ foliations;
\item codimension $1$ symplectic foliations.
\end{enumerate}
In all these cases, the structures on $M$ are obtained by pulling back the corresponding ``linear'' structures on $\mathbb{R}^n$ via charts in a $\Gamma$-atlas; one then uses the transition functions in $\Gamma$ to show that these local structures glue correctly. This construction yields a global structure on $M$ that is locally isomorphic to the corresponding ``linear'' one on $\mathbb{R}^n$.
\end{example}

\begin{example}\label{ex:cont-struct}

If a structure on $\mathbb{R}^n$ is not ``linear'' (in the sense of Example~\ref{ex:G-struct}), then its pseudogroup of local symmetries is not of the form $\Gamma_G$, where $G\subset \GL_n(\mathbb{R})$ is a subgroup of linear transformations. For instance, let $n=2k+1$, $k\in \mathbb{N}$ and let us consider coordinates $({\bf x}, {\bf y}, z)$ on $\mathbb{R}^{2k+1}$ where ${\bf x}=(x_1, \dots x_k)\in \mathbb{R}^k$ and ${\bf y}=(y_1, \dots y_k)\in \mathbb{R}^k$. Let 
\[
\alpha_{\rm std}:=\ker(dz-{\bf y}d {\bf x})
\]
be the standard contact form on $\mathbb{R}^{2k+1}$, so that $\xi_{\rm std}=\ker(\alpha_{\rm std})$ is the standard {\bf contact structure} on $\mathbb{R}^{2k+1}$.

Let $\Gamma_{\rm cont}$ be the pseudogroup of locally defined diffeomorphisms of $\mathbb{R}^{2k+1}$ preserving $\alpha_{\rm std}$; that is, a diffeomorphism onto its image $\phi:U\to \mathbb{R}^{2k+1}$, defined on an open subset $U\subset \mathbb{R}^{2k+1}$, belongs to $\Gamma_{\rm cont}$ if and only if $\phi^*(\alpha_{\rm std}|_{\phi(U)})=\alpha_{\rm std}|_U$. The pseudogroup $\Gamma_{\rm cont}$ is transitive, but it is not of the form $\Gamma_G$ for any $G\subset \GL_n(\mathbb{R})$. In fact, $\Gamma_{\rm cont}$ does not contain translations. 

We can still describe explicitly some of the {\it globally defined} elements of $\Gamma_{\rm cont}$ in terms of a Lie group. Recall that the Heisenberg group $\Hei$ is the manifold $\mathbb{R}^{2k+1}$ with Lie group structure given by the multiplication
\[
({\bf x}, {\bf y}, z)\cdot ({\bf x}', {\bf y}', z')=({\bf x}+{\bf x}', {\bf y}+{\bf y}', z+z'+\langle{\bf x}, {\bf y}'\rangle),
\]
where $\langle{\bf x}, {\bf y}'\rangle$ denotes the standard scalar product. The group $\Hei$ acts on 
itself by right multiplication; the diffeomorphisms of $\mathbb{R}^{2k+1}$ induced 
by this action belong to $\Gamma_{\rm cont}$.

Finally, with the same construction outlined for $\Gamma_G$, a $\Gamma_{\rm cont}$-structure on a manifold $M$ corresponds to a {\it coorientable} contact structure on $M$.
\end{example}

Examples~\ref{ex:G-struct} and~\ref{ex:cont-struct} deal with transitive pseudogroups on $\mathbb{R}^n$. In both cases -- pseudogroups of the form $\Gamma_G$,  for some subgroup $G\subset \GL(n, \mathbb{R})$, and $\Gamma_{\rm cont}$ -- a way to realise that the pseudogroup under consideration is transitive is to observe that it contains elements that come from a {\it free and transitive action} of some Lie group on $\mathbb{R}^n$ -- either the group of translations or the Heisenberg group.

\begin{definition}\label{def:generated-pseudogroup}
Let $K$ be a Lie group acting freely and transitively on a manifold $\BB$ and let $\Gamma$ be a pseudogroup on $\BB$. We say that $\Gamma$ {\bf contains $K$} if, for all $k\in K$, the diffeomorphism
\[
T_k: \BB\to \BB, \quad x \mapsto k\cdot x
\] 
belongs to $\Gamma$.

\end{definition}
As an immediate consequence of the definition, a pseudogroup containing a Lie group is transitive.

\begin{example}\label{ex_gamma_G_contains_Rn}
The Lie group of translations of $\mathbb{R}^n$ can be identified with $\mathbb{R}^n$ equipped with the sum of vectors. It acts freely and transitively on $\mathbb{R}^n$. For any Lie subgroup $G\subset \GL(n, \mathbb{R})$, we see that $\mathbb{R}^n$ is contained in $\Gamma_G$.
\end{example}

\begin{example}\label{ex_gamma_cont_contains_heis}
As already mentioned in Example~\ref{ex:cont-struct}, the Heisenberg group $\Hei$ acts freely and transitively on $\mathbb{R}^{2k+1}$ (it is, in fact, a Lie group structure on $\mathbb{R}^{2k+1}$). The pseudogroup $\Gamma_{\rm cont}$ contains $\Hei$.
\end{example}

\begin{example}\label{ex:not_Lie}
Given any (not necessarily regular) foliation $\mathcal{F}$ on $\mathbb{R}^n$, 
the pseudogroup $\Gamma_{\mathcal{F}}$ of {\it leafwise} diffeomorphisms (i.e.\ embeddings sending opens into leaves to opens into {\it the same} leaf) is very rarely transitive. In fact, the orbits of such a pseudogroup are precisely the leaves of $\mathcal{F}$. By choosing $\mathcal{F}$ to be a singular foliation, one can exhibit pseudogroups whose orbits have a very complicated behaviour. 
\end{example}


\subsubsection{Differential equations}
In Examples~\ref{ex:G-struct} and~\ref{ex:cont-struct} above, our pseudogroups are defined by considering symmetries of global geometric objects on $\mathbb{R}^n$; a map belongs to the pseudogroup if it pulls back the given object to itself.  On each specific example, this request translates into {\it differential conditions}/{\it differential equations} imposed to the elements of the pseudogroup. Let us be precise about the framework for differential equations which we adopt (see e.g.\ \cite{SAUNDERS, Vinogradovetal99} for details).

Recall that, given two manifolds $M$ and $N$, the space $J^k(M, N)$ of $k$-jets consists of equivalence classes $j^k_xf$ of maps $f:M\to N$ up to contact of order $k$ at $x\in M$; i.e., $j^k_xf=j^k_xg$, $f, g:M\to N$ if and only if $f$ and $g$ have the same Taylor expansion of order $k$ at $x$. The space $J^k(M, N)$ is always a smooth manifold -- the coordinates are simply given by the coordinates on $M$ and $N$ together with the coefficients of Taylor expansions of order $k$-- and it surjects onto $J^{k-1}(M, N)$ and onto $M$.
Recall also that a local section
\[
\sigma:U\to J^k(M, N), \quad x \mapsto \sigma(x)
\]
is called {\bf holonomic} when $\sigma (x) = j^k_xf$ for every $x \in U$ and for a fixed $f:U\subset M\to N$\footnote{The reader unfamiliar with jet spaces might want to notice that not all sections are holonomic. In general, when $\sigma:U\to J^k(M,N)$ is a local section, then for all $x\in U$, one has $\sigma(x) = j^k_x f_x$, where the $f_x:U_x\to N$ are some smooth functions defined on neighbourhoods $U_x$ of $x$. Holonomicity correspond to the request that the $f_x$'s can be chosen to be equal to a given function $f$ whose domain is $U$.}.
Holonomic sections are therefore in bijection with locally defined maps from $M$ to $N$.

\begin{definition}\label{def:diff_eq}
A geometric {\bf differential equation of order $k$} is an embedded submanifold $R \subset J^k(M, N)$. A {\bf local solution} is a map $f:U\subset M\to N$ such that the holonomic section 
\[
j^kf:U\to J^k(M, N), \quad x \mapsto j^k_x f
\]
 takes values in $R$.
\end{definition}

Now, let $\Gamma$ be a pseudogroup on $\BB$; for any $k\in \mathbb{N}$ we can form the $k$-jet space
\[
J^k\Gamma:=\{j^k_x\varphi:\ \varphi\in \Gamma\}\subset J^k(\BB, \BB),
\]
obtaining a tower of surjective continuous maps
\[
\dots \to J^k\Gamma\to J^{k-1}\Gamma\to \dots \to J^0\Gamma\to \BB.
\]
Here, $J^0\Gamma\subset \BB\times \BB$ encodes the orbits of $\Gamma$; that is, $(y, x)\in J^0\Gamma$ if and only if there is $\varphi\in \Gamma$ such that $\varphi(x)=y$. As Example~\ref{ex:not_Lie} shows for $k=0$, there is no reason for $J^k\Gamma$ to be a submanifold of $J^k(\BB, \BB)$.

\begin{definition}\label{def:Lie_pseudo}
A pseudogroup $\Gamma$ is called {\bf Lie pseudogroup} when 
\begin{itemize}
\item $J^k\Gamma$ is an embedded submanifold of $J^k(\BB, \BB)$, for all $k\in \mathbb{N}$;
\item all the projections $J^k\Gamma\to J^{k-1}\Gamma$ are surjective submersions.
\end{itemize}
We call $\Gamma$ a Lie pseudogroup {\bf of order $k$} if, additionally, $\varphi \in \Gamma$ if and only if $j^k\varphi$ takes values in $J^k\Gamma$, i.e.\ if and only if $\varphi$ is a solution of the PDE $J^k\Gamma$.
\end{definition}
The intuition behind the above definition is simply that ``Lie pseudogroups of order $k$ are pseudogroups arising as solutions of PDEs of order $k$''.
\begin{example}
All the pseudogroups from Examples~\ref{ex:G-struct} and~\ref{ex:cont-struct} are Lie pseudogroups. Furthermore, they are of order 1, since the differential conditions defining them are first order conditions.
On the other hand, as anticipated, the pseudogroup of Example~\ref{ex:not_Lie} is in general not Lie.
\end{example}

\subsubsection{Lie groupoids}
There is some more structure on $J^k\Gamma$ that needs to be uncovered. 
A first piece of structure, that arises already at the level of jets, is a distribution detecting holonomic sections; we will discuss it in section \ref{subs:cart-dist}. A second piece of structure is directly inherited from the group-like properties of $\Gamma$. Namely, one has
\begin{itemize}
\item a pair of projections $J^k\Gamma\rightrightarrows \BB$, called {\bf source} and {\bf target} map,
\[
s(j^k_x f)=x, \quad \quad t(j^k_x f)=f(x); 
\]
\item a ``partial {\bf multiplication}''  
\[
J^k\Gamma \tensor[_s]{\times}{_t} J^k\Gamma\to J^k\Gamma,\quad (j^k_{\varphi(x)}\varphi', j^k_x\varphi) \mapsto j^k_{\varphi(x)}\varphi'\cdot j^k_x\varphi=j^k_x(\varphi'\circ \varphi)
\]
\end{itemize}
such that
\begin{itemize}
\item there is a canonical global {\bf unit} section $u: \BB\to J^k\Gamma, x \mapsto j^k_x (\mathrm{id}_\BB)$ which is a bijection between points in $\BB$ and units for the multiplication;
\item each element has an inverse;
\item the multiplication is associative.
\end{itemize}

A {\bf groupoid} $\G$ over $\BB$, denoted by $\G\tto \BB$, is a set $\Sigma$ possessing a partial multiplication (and source and target maps) as above; $\G$ is called {\bf arrow space} 
 while $\BB$ is called {\bf unit space}.
Groupoids form a category: a {\bf morphism} of groupoids, simply denoted by $\Phi:\Sigma_1\to \Sigma_2$ is a commutative diagram 
\[
\begin{tikzcd}
\Sigma_1\arrow[d, shift left=-.5ex]\arrow[d, shift right=-.5ex]\arrow[r, "\Phi"]& \Sigma_2\arrow[d, shift left=-.5ex]\arrow[d, shift right=-.5ex]\\
\BB_1\arrow[r, "\phi"]& \BB_2
\end{tikzcd}
\]
respecting the multiplication of arrows, i.e.\ such that $\Phi(g\cdot h)=\Phi(g)\cdot \Phi(h)$ for all $h, g\in \Sigma_1$ such that $t(h)=s(g)$. 
 
Notice that, given a groupoid $\G$ over $\BB$, there is a natural notion of {\bf orbit}: the orbit $\mathcal{O}_x$ of $x\in \BB$ is the set of $y\in \BB$ such that there is some $g\in \G$ with $s(g)=x$, $t(g)=y$. Applied to $J^k\Gamma$, this recovers the notion of orbit of a pseudogroup (Definition \ref{def:transitive-orbits}). One also observes that the intersection $\G_x:=s^{-1}(x)\cap t^{-1}(x)$ possesses a group structure; it is called the {\bf isotropy group at $x$}. Isotropy groups at points in the same orbits are isomorphic. Furthermore, the action of $\G_x$ on $s^{-1}(x)$ by right multiplication has quotient $s^{-1}(x)/\G_x$ in bijection with the orbit of $x$ and the quotient map can be identified with $t:s^{-1}(x)\to \BB$; analogously for the action of $\G_x$ on $t^{-1}(x)$ by left multiplication.

When dealing with the {\bf jet groupoid} $J^k\Gamma\tto \BB$ of a Lie pseudogroup $\Gamma$, we also see that both $J^k\Gamma$ and $\BB$ carry smooth structures such that all the groupoid structure operations are smooth; and even more, $s$ and $t$ are surjective submersions. All in all $J^k\Gamma\rightrightarrows \BB$ is a {\bf Lie groupoid}. For details on (Lie) groupoids see e.g.\ \cite[Chapter 1]{MACKENZIE} or \cite[Section 5.1]{MOERDIJK}.

\begin{definition}
A Lie groupoid $\Sigma\tto \BB$ is a groupoid such that 
\begin{itemize}
\item $\Sigma$ and $\BB$ are smooth manifolds;
\item $s$ and $t$ are smooth surjective submersions;
\item all the other structure maps (multiplication, unit section and inversion map) are smooth.
\end{itemize}
\end{definition}
%

Lie groupoids form a category as well: a {\bf morphism of Lie groupoids} $\Sigma_1\tto \BB_1$ and $\Sigma_2\tto \BB_2$ is a morphism of groupoids 
\[
\begin{tikzcd}
\Sigma_1\arrow[d, shift left=-.5ex]\arrow[d, shift right=-.5ex]\arrow[r, "\Phi"]& \Sigma_2\arrow[d, shift left=-.5ex]\arrow[d, shift right=-.5ex]\\
\BB_1\arrow[r, "\phi"]& \BB_2
\end{tikzcd}
\]
where both $\Phi$ and $\phi$ are smooth maps. When $\Sigma\tto \BB$ is a Lie groupoid (see~\cite[Theorem 5.4]{MOERDIJK} for details):
\begin{itemize}
\item $s^{-1}(x)$ is an embedded submanifold of $\Sigma$, for all $x\in \BB$;
\item the isotropy group at $x$, $\Sigma_x=s^{-1}(x)\cap t^{-1}(x)$ is an embedded submanifold and a Lie group; 
\item the right action of $\Sigma_x$ on $s^{-1}(x)$ is smooth;
\item the orbit of $x\in \BB$, $\mathcal{O}_x$, is an immersed submanifold of $\BB$;
\item $t: s^{-1}(x)\to \BB$ is actually a right principal $\Sigma_x$-bundle over $\mathcal{O}_x$ (similarly replacing $s$ with $t$ -- in that case one gets a left principal bundle).
 \end{itemize}

A Lie groupoid $\Sigma\tto \BB$ is called {\bf transitive} when the {\bf anchor map}
\[
(s, t): \Sigma\to \BB\times \BB\quad g\to (s(g), t(g))
\]
is surjective. 
Transitivity is equivalent to asking that for each $x\in \BB$, $y\in \BB$ there is some $g\in \Sigma$ such that $s(g)=x$ and $t(g)=y$ -- i.e., $\Sigma$ possesses a single orbit: for each $x\in \BB$, $\mathcal{O}_x=\BB$. As a result, 
$t: s^{-1}(x)\to \BB$ is a principal $\Sigma_x$-bundle. 
\begin{remark}[Transitivity and topology of $\Sigma$]\label{rmk:transitivity_and_second_countability}
It is worth noticing here that if $\Sigma\tto \BB$ is transitive, then the anchor map
\[
(s, t): \Sigma\to \BB\times \BB\quad g\to (s(g), t(g))
\]
is also submersive. This is due to the fact that the arrow space $\Sigma$ is a smooth manifold; in particular, one needs that $\Sigma$ is a second countable topological space. One sees that the action of $\Sigma_x$ on the embedded submanifold $s^{-1}(x)$ is principal. The action is free and transitive along the fibres of $t:s^{-1}(x)\to \BB$, and such fibres are invariant under the action. Consequently, the principal quotient projection can be identified with $t:s^{-1}(x)\to \BB$. The topology on the quotient space makes it into a smooth immersed submanifold of $\BB$ (see e.g.\ Theorem 5.4 in~\cite{MOERDIJK}). However, due to second countability of $s^{-1}(x)$, the topology on the quotient space has to be given by the given manifold topology of $\BB$.

The fact that the anchor map of a transitive Lie groupoid is submersive is needed for many fundamental facts to hold true, e.g. Example~\ref{ex:ME_transitive}. Occasionally, one defines Lie groupoids by allowing the arrow space to be a possibly non-Hausdorff and non-second countable manifold\footnote{For a possible reason to do so, see Subsection~\ref{subs:germ_groupoid} and, in particular, Remark~\ref{rmk:germ_not_Lie}.}; see, e.g.,~\cite{MOERDIJK}. In such a situation, one typically defines transitivity by requiring the anchor map to be surjective {\it and submersive}. Cf.~\cite[Remarks A.8-A.10]{AccorneroCattafi1}.
\end{remark}

\begin{remark}[Transitivity of (Lie) pseudogroups]\label{rmk:transitivity_of_pseudo}
Observe that transitivity of a (not necessarily Lie) pseudogroup $\Gamma$ over $\BB$ (see Definition \ref{def:transitive-orbits}) is equivalent to the map
\[
(s, t): J^k\Gamma\to \BB\times \BB\quad j^k_x\varphi \mapsto (x, \varphi(x))
\]
being surjective. Consequently, a {\it Lie} pseudogroup $\Gamma$ is transitive if and only if its jet groupoid $J^k\Gamma\tto \BB$ is transitive. 
\end{remark}

\begin{example}[Isotropies of first jet groupoids]\label{ex:iso_cont}
Observe that the Lie group $G\subset \GL(n, \mathbb{R})$ is isomorphic to the isotropy group (at any point) of $J^1\Gamma_G$. Actually
\[
J^1\Gamma_G\cong \{(x, y, l):\ l:\mathbb{R}^n\to \mathbb{R}^n,\ l\in G\} = \RR^n \times \RR^n \times G.
\]
On the other hand
\[
J^1\Gamma_{\rm cont}\cong \{(x, y, l):\ l:\mathbb{R}^{2k+1}\to \mathbb{R}^{2k+1},\ l\in d_y\varphi_y\circ \Sp(k,1)\circ d_x \varphi^{-1}_x
\}
\]
where we identify $\mathbb{R}^{2k+1}$ with the Heisenberg group $\Hei$ and denote by $\varphi_x$ and $\varphi_y$ the diffeomorphisms of $\mathbb{R}^{2k+1}$ induced by right multiplication by $x\in \mathbb{R}^{2k+1}$ and $y\in \mathbb{R}^{2k+1}$ respectively.
It follows that the isotropy group of $J^1\Gamma_{\rm cont}$ is isomorphic to $\Sp(k, 1)$. Even more, it follows $J^1\Gamma_{\rm cont}\cong J^1\Gamma_{\Sp(k, 1)}$ as Lie groupoids (see also Proposition \ref{prop:isomorphism_source_first_jet}).
%
%
\end{example}

\subsubsection{Lie algebroids}\label{subsec:Lie-alg}


Finally, we briefly recall that any Lie groupoid comes with an associated infinitesimal object, a {\bf Lie algebroid}, i.e.\ a vector bundle $A\to \BB$ equipped with
\begin{itemize}
\item a vector bundle map, called {\bf anchor}, $\rho:A\to TM$;
\item a Lie bracket on the space of sections $\Gamma(A)$;
\end{itemize}  
which satisfies the Leibniz identity
\[
[\alpha, f\beta]=f[\alpha, \beta]+L_{\rho(\alpha)}(f)\beta \quad \quad \forall \alpha, \beta\in \Gamma(A), f\in C^\infty(\BB). \]
The Lie algebroid $A:={\rm Lie}(\Sigma)\to \BB$ of a Lie groupoid $\G\tto \BB$ is constructed as follows:
\begin{itemize}
\item the total space $A$ is given by $\ker(ds)|_\BB$ (recall that we can see $\BB$ as an embedded submanifold of $\Sigma$ using the unit (bi)section of $\Sigma\tto \BB$);
\item the anchor map is the restriction $dt|_A: A \to T\BB$;
\item the bracket on sections of $A$ is induced by the bracket of {\it right-invariant vector fields} on $\Sigma\tto \BB$.
\end{itemize}
Throughout this paper we will often write $A\to \BB$ in place of ${\rm Lie}(\Sigma)$ to denote the Lie algebroid of a groupoid $\Sigma\tto \BB$.

The construction of the Lie algebroid of a Lie groupoid is reminescent to the construction of the Lie algebra of a Lie group, which is indeed a particular case -- any Lie group $G$ is a Lie groupoid over the point. Part of the classical Lie theory can be extended to Lie algebroids and Lie groupoids -- e.g.~\cite{MARIUSRUI} and references therein.



Recall that a {\bf Lie subalgebroid} over $\BB$ of a Lie algebroid $A\to \BB$ is a vector subbundle such that the inclusion of $A'\to \BB$ into $A\to \BB$ is a Lie algebroid morphism; in particular $\Gamma(A')$ is a subalgebra of $\Gamma(A)$.
A subalgebroid $A'\to \BB$ of $A\to \BB$ is called {\bf ideal} if $\Gamma(A')$ is an ideal in $\Gamma(A)$.

\subsection{$\Gamma$-structures as principal bundles}
With the groupoid language at hand, one can conceptualise the notion of geometry on a manifold in terms of familiar objects.

\subsubsection{The germ groupoid of a pseudogroup}\label{subs:germ_groupoid} 
Let $\Gamma$ be a pseudogroup on $\BB$; for the discussion in this subsubsection, we do not need $\Gamma$ to be a \textit{Lie} pseudogroup. There is a further groupoid attached to $\Gamma$: the {\bf germ groupoid} $\Germ(\Gamma)\tto \BB$, whose space of arrows is
\[
\Germ(\Gamma):=\{\germ_x(\varphi):\ x\in \BB,\ \phi\in \Gamma\}.
\]
The source and target maps are given by 
\[
s: \germ_x(\varphi)\in \Germ(\Gamma) \mapsto x\in \BB
\]
and
\[
t: \germ_x(\varphi)\in \Germ(\Gamma) \mapsto \varphi(x)\in \BB,
\]
while the multiplication is defined by 
\[
\germ_{\varphi(x)}(\varphi')\cdot \germ_x(\varphi):= \germ_x(\varphi'\circ \varphi)\in \Germ(\Gamma),
\]
for $\germ_{\varphi(x)}(\varphi'),\ \germ_x(\varphi) \in \Germ(\Gamma)$.

The set $\Germ(\Gamma)$ can be equipped with the {\bf étale topology}, making $s$ and $t$ into local homeomorphisms. All the structure operations -- multiplication, inversion and the unit bisection -- are continuous with respect to this topology. In other words, $\Germ(\Gamma)$ is a {\bf topological groupoid}; in fact, it is an {\bf étale groupoid} -- a topological groupoid where the source and target maps are local homeomorphisms.
Notice that all the source fibres $s^{-1}(x)\subset \Germ(\Gamma)$, $x\in \BB$, are discrete in the étale topology.

\begin{remark}[Smooth structure on the germ groupoid]\label{rmk:germ_not_Lie}
Since the source map $s:\Germ(\Gamma)\to \BB$ is a local homeomorphism, one can use it to pull-back atlases of $\BB$ to $\Germ(\Gamma)$, showing that $\Germ(\Gamma)$ is in fact a non-Hausdorff and non-second countable manifold (see also the second part of Remark \ref{rmk:transitivity_and_second_countability}). This is actually true for any étale groupoid $\mathcal{G}\tto \BB$.
\end{remark}

%

We will use the notation $\mathcal{G}\tto \BB$ to denote {\it étale} groupoids -- i.e.\ Lie groupoids such that the source and target maps are local homeomorphisms -- and keep the notation $\G \tto \BB$ for arbitrary Lie groupoids.

Let us now discuss another important property of the groupoid $\Germ(\Gamma)$, for which we need the notion of bisection.

\begin{definition}
Let $\G\tto \BB$ be a Lie groupoid. A local section $\sigma:U\subset \BB\to \G$ of the source map is called {\bf bisection} if the composition $t\circ \sigma:U\to \BB$ is a diffeomorphism onto its image.
\end{definition}

%
%

\begin{definition}
An étale groupoid $\mathcal{G}\tto \BB$ is called {\bf effective} when, for any two local bisections $\sigma$, $\sigma'$, if $t\circ \sigma=t\circ \sigma'$ then $\sigma=\sigma'$.
\end{definition}

As anticipated, all étale groupoids arising as germ of pseudogroups are effective. The converse is also true: given any étale groupoid $\mathcal{G}$, the set
\[
\Gamma_{\mathcal{G}}:=\{t\circ \sigma:\ \sigma\ \text{is a local bisection of }\mathcal{G}\}
\]
defines a pseudogroup; moreover, the morphism of groupoids
\[
\mathcal{G}\to \Germ(\Gamma_\mathcal{G}), \quad g \mapsto \germ_x(t\circ \sigma_g),
\] 
where $\sigma_g$ is a bisection defined around $x = s(g)$ and such that $\sigma_g(x)=g$, is an isomorphism precisely when $\mathcal{G}$ is effective.

\begin{proposition}
There is a one to one correspondence between pseudogroups on $\BB$ and effective étale groupoids over $\BB$.
\end{proposition}
The proposition above is usually attributed to André Haefliger, see \cite[section I.6]{HAEFLIGERGAMMASTRUCTURES}. We are interested in making use of this correspondence to encode geometries on manifolds in a -- hopefully more manageable -- structure. 

\begin{definition}\label{def:action}
Let $\Sigma\tto \BB$ be a Lie groupoid. A {\bf left action} of $\Sigma$ on a manifold $P$ along a smooth map $\mu:P\to \BB$, or, more briefly, a {\bf (left) $\Sigma$-space} $\mu:P\to \BB$, is a smooth map 
\[
m_P: \G\tensor[_s]{\times}{_\mu} P\to P
\]
such that 
\begin{itemize}
\item $m_P(g_1, m_P(g_2, p))=m_P(g_1g_2, p)$, for all $g_1, g_2\in \Sigma$, $s(g_1)=t(g_2)$, $p\in P$, $\mu(p)=s(g_2)$;
\item  $m_P(1_x, p)=p$ for all $p\in P$ such that $\mu(p)=x\in \BB$, where $1_x$ denotes the identity in $\Sigma$ over $x$.
\end{itemize}
\end{definition}
The notion of right action/right $\Sigma$-space is defined analogously.

Given a $\Sigma$-space $\mu:P\to \BB$, the {\bf orbit space} of the action of $\Sigma$ is the quotient $P/\Sigma$ with respect to the equivalence relation
\[
p_1\sim p_2 \quad \text{if and only if}\quad p_2=g\cdot p_1 \text{ for some } g\in \Sigma, s(g)=\mu(p_1).
\]

\begin{definition}\label{def:PB}
Let $\Sigma\tto \BB$ be a Lie groupoid and $M$ be a manifold. A {\bf principal $\G$-bundle over $M$} is a $\Sigma$-space $\mu:P\to \BB$ together with a surjective submersion $\pi:P\to M$ such that 
\begin{itemize}
\item $\pi$ is $\G$-invariant, i.e.\ $\pi(g\cdot p)=\pi(p)$ for all $g\in \G$, $p\in P$ such that $s(g)=\mu(p)$;
\item the action of $\G$ is transitive on $\pi$-fibres, i.e.\ if $p_1$, $p_2$ are points in $P$ and $\pi(p_1) = \pi(p_2)$, then $p_2 =g\cdot p_1$ for some $g\in \G$ such that $s(g) = \mu(p_1)$;
\item the action is free: for all $p\in P$, $g\in \Sigma$ such that $s(g)=\mu(p)$, if $g\cdot p=p$ then $g=1_{\mu(p)}$;
\item the action is proper: the map
\[
\Sigma\tensor[_s]{\times}{_\mu}P \to P\times P, \quad (g, p) \mapsto (g\cdot p, p)
\]
is a proper map.
\end{itemize}
\end{definition}
The orbit space of a principal $\G$-bundle over $M$ possesses a canonical smooth structure making the quotient projection into a surjective submersion; the map $P/\G\to M$ induced by $\pi$ is a diffeomorphism when $P/\G$ is given such a topology. 

\begin{definition}\label{def:morph_PB}
A {\bf morphism} of principal $\Sigma$-bundles $\mu_1:P_1\to \BB$, $\mu_2: P_2\to \BB$ over $M$ is a map $F:P_1\to P_2$ commuting with the action of $\Sigma$, i.e.\ 
\begin{itemize}
\item $\mu_2\circ F=\mu_1$;
\item $F(g\cdot p)=g\cdot F(p)$ for all $g\in \Sigma$, $p\in P$, $\mu(p)=s(g)$.
\end{itemize}
\end{definition}


The following proposition follows from Proposition ${\rm I}.3.1$ in~\cite{JANEZ}.

\begin{proposition}\label{prp:Gamma-struct-are-pb}
Let $\Gamma$ be a pseudogroup on $\BB$ and $M$ a manifold with $\dim(M)=\dim(\BB)$. There is a bijective correspondence 
\[ 
 \left\{   \begin{array}{c}
            \Gamma \text{-structures on } M\\
             \text{(see subsubsection \ref{sss:local-models})}
            \end{array} 
\right\} 
\stackrel{1-1}{\longleftrightarrow}
\left\{   \begin{array}{c}
           \text{Isomorphism classes of principal } \Germ(\Gamma)\text{-bundles over $M$}\\
           \text{whose moment map is \'etale}
           \end{array} 
\right\}.
\]
\end{proposition}

We stress here that, since the arrow space $\Germ(\Gamma)$ is not Hausdorff and/or second countable in general, strictly speaking we cannot talk about {\it smooth} principal bundles. However, $\Germ(\Gamma)$ is always locally Euclidean (see Remark~\ref{rmk:germ_not_Lie}), hence the definition can be easily adapted.

There is a relatively simple description of the principal $\Germ(\Gamma)$-bundle associated to a $\Gamma$-structure on $M$ that makes use of a maximal $\Gamma$-atlas $\mathcal{A}=\{(U_i, \phi_i)\}_{i\in I}$ with change of coordinates $\{\varphi_{ij}\}_{i, j\in I}$, $\varphi_{ij}\in \Gamma$. Define
\[
\Germ(\mathcal{A}):=\{\germ_p(\phi_i)|\ p\in U_i\subset M\}.
\]
There is an action of $\Germ(\Gamma)$ on $\Germ(\mathcal{A})$ with moment map 
\[
\mu: \Germ(\mathcal{A})\to \BB, \quad \germ_p(\phi_i) \mapsto \phi_i(p).
\]
The action is given by
\[
\germ_{\phi(p)}(\varphi)\cdot \germ_p(\phi_i):=\germ_p(\varphi\circ \phi_i);
\]
notice that  $\germ_p(\varphi\circ \phi_i)\in \Germ(\mathcal{A})$ because $\mathcal{A}$ is a maximal atlas. This action is principal over $M$ with projection
\[
\pi: \Germ(\mathcal{A})\to M, \quad \germ_p(\phi_i)\mapsto p.
\]
The explicit description we have just given will be useful in the next section.
\begin{axiom}{MM}\label{axiom_moment_map}
Given a Lie groupoid $\Sigma\tto \BB$ and a $\Sigma$-principal bundle $\mu:P\to \BB$ over $M$, the moment map $\mu$ is always assumed to be a surjective submersion.
\end{axiom}

\subsubsection{Non integrable geometries}
Let $\Gamma$ be a Lie pseudogroup over $\BB$ (Definition~\ref{def:Lie_pseudo}); under the correspondence from Proposition \ref{prp:Gamma-struct-are-pb} we can consider a $\Gamma$-structure on $M$, 
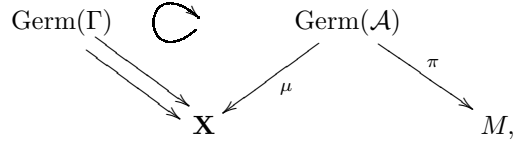
\[
\xymatrix{
\Germ(\Gamma) \ar@<0.25pc>[dr] \ar@<-0.25pc>[dr]  & \ar@(dl, ul) &  \Germ(\mathcal{A}) \ar[dl]^{\mu}\ar[dr]^{\pi} &    \\
&\BB  & &  M,}
\]
where $\mathcal{A}=\{(U_i, \phi_i)\}_{i\in I}$ is a maximal atlas with changes of coordinates $\{\varphi_{ij}\}_{i, j\in I}$, $\varphi_{ij}\in \Gamma$.
One can make use of the tower of Lie groupoids
\[
\dots \to J^k\Gamma\to J^{k-1}\Gamma\to \dots \to J^0\Gamma\tto \BB.
\] 
to look at the {\it $k$-th order data} of the atlas/the ``projections'' of the principal $\Germ(\Gamma)$-bundle to $J^k\Gamma$. Precisely, one has principal bundles
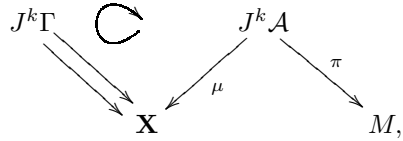
\[
\xymatrix{
J^k\Gamma \ar@<0.25pc>[dr] \ar@<-0.25pc>[dr]  & \ar@(dl, ul) &  J^k\mathcal{A} \ar[dl]^{\mu}\ar[dr]^{\pi} &    \\
&\BB  & &  M,}
\]
where now 
\[
J^k\mathcal{A}:=\{j^k_p\phi_i|\ p\in U_i\subset M\}.
\]
This point of view was extensively used in~\cite{FRANCESCO} to address the so called {\it formal integrability problem} for geometric structures on manifolds. We quickly go through the definitions that are relevant for our purposes; all the material in this subsection comes from~\cite{FRANCESCO}, to which we refer for more details.

Notice that we have the following particular case of the above construction. If $M$ is a manifold and $\mathcal{A}_M$ is a maximal atlas, then, for each $k$, the space $J^k\mathcal{A}_M$ consisting of {\it $k$-jets of charts} is a principal $J^k\Gamma^n$-bundle over $M$ (where $\Gamma^n := \Diff_{\rm loc}(\RR^n)$, see Example~\ref{ex:Gamma-n}). 


Let $\Gamma$ be a Lie pseudogroup over $\BB$ and $M$ be a manifold. Assume $M$ to possess a ${\rm Diff}_{\rm loc}(\BB)$-structure and denote the induced principal $J^k{\rm Diff}_{\rm loc}(\BB)$-bundle by $\Pi^k\subset J^k(M, \BB)$. Recall that $\Pi^k$ is the space of germs of maps in a maximal ${\rm Diff}_{\rm loc}(\BB)$-atlas. A {\bf $J^k\Gamma$-reduction} of $\Pi^k$ is a principal $J^k\Gamma$-bundle $P$ over $M$ such that $P\subset \Pi^k$ and the moment map and action of $J^k\Gamma$ on $P$ are obtained restricting the moment map and action of  $J^k{\rm Diff}_{\rm loc}(\BB)$ on $\Pi^k$. Notice that any $\Gamma$-structure on $\BB$ induces a $J^k\Gamma$-reduction of $\Pi^k$; in fact, a maximal $\Gamma$-atlas is necessarily contained in a maximal ${\rm Diff}_{\rm loc}(\BB)$-atlas, since $\Gamma\subset {\rm Diff}_{\rm loc}(\BB)$.

\begin{definition}\label{def:almost_gamma_struct}
Let $\Gamma$ be a Lie pseudogroup over $\BB$ and $M$ be a manifold. Assume $M$ to possess a ${\rm Diff}_{\rm loc}(\BB)$-structure and denote the induced principal $J^k{\rm Diff}_{\rm loc}(\BB)$-bundle by $\Pi^k\subset J^k(M, \BB)$. An {\bf almost $\Gamma$-structure of order $k$} on a manifold $M$ is a $J^k\Gamma$-reduction of $\Pi^k$.

An almost $\Gamma$-structure of order $k$ is called {\bf integrable} if there is a $\Gamma$-structure on $M$ with maximal atlas $\mathcal{A}$ such that $P=J^k\mathcal{A}$.
\end{definition}
Not all $\Gamma$-structures are integrable! The integrability corresponds precisely to existence of local models as explained above.

\begin{example}
Let $\Gamma_{\rm eucl}$ be the pseudogroups of isometries of the  euclidean metric in $\mathbb{R}^n$. A $\Gamma_{\rm eucl}$-structure on $M$ corresponds to a flat metric $g$ on $M$ (it is locally isometric to the Euclidean one by definition of $\Gamma$-atlas). The corresponding $J^1\Gamma_{\rm eucl}$-principal bundle can be encoded (passing to trivialisations) by an atlas $(\{U_i\}_{i\in I}, \{\phi_i\}_{i\in I}, \{\varphi_{ij}\}_{i, j\in I})$ where the transition functions are now induced by sections of $J^1\Gamma_{\rm eucl}$ which are {\it not holonomic}. This amounts to say that they are {\it formal isometries}: the maps themselves are not isometries on their domain, but they come paired with linear isometries of the tangent spaces. If one considers a $J^1\Gamma_{\rm eucl}$-principal bundle not coming from a $\Germ(\Gamma_{\rm eucl})$-one, one can use such an atlas to endow $M$ with a Riemannian metric $g$ which is, in general, not flat. The curvature of $g$ is precisely the obstruction to the existence of a $\Germ(\Gamma_{\rm eucl})$-principal bundle over the $J^1\Gamma_{\rm eucl}$-one~\cite{FRANCESCO}.

A similar discussion holds for all the $\Gamma_G$'s from Example~\ref{ex:G-struct}. The corresponding $J^1\Gamma_G$-principal bundles over $M$ are precisely $G$-structures on $M$~\cites{CHERN, STERNBERG}; and they are flat if and only if the bundle sits below a $\Germ(\Gamma_G)$-bundle. This is most appreciated when taking into full account  transitivity, and we will come back to it later on. 
\end{example}

\begin{example}
A similar discussion holds true for $\Gamma_{\rm cont}$ from Example~\ref{ex:cont-struct}. A principal $J^1\Gamma_{\rm cont}$-structure induces an atlas where the transition functions are formal contactomorphisms and the Darboux theorem holds precisely when the bundle is induced by a principal $\Germ(\Gamma_{\rm cont})$-bundle.
\end{example}

\subsection{Transitive groupoids and their isotropy groups}
\subsubsection{Transverse geometry of a transitive groupoid}\label{subs:ME}
The notion of {\it Morita equivalence} between Lie groupoids, which appeared first in \cite[Definition 2.1]{XUMORITAEQ} in the context of Poisson geometry, captures the idea of ``transverse geometry'' of a Lie groupoid $\Sigma\tto \BB$ (see \cites{JOAO, MATIAS, MOERDIJK} for this point of view).
\begin{definition}\label{def:classical-ME}
A {\bf principal bibundle} between $\G_1\tto \BB_1$ and $\G_2\tto \BB_2$ is a space $P$ together with two maps $\mu_1:P\to \BB_1$ and $\mu_2:P\to \BB_2$ such that
\begin{itemize}
\item $\mu_1:P\to \BB_1$ is a left $\G_1$-space and a principal $\Sigma_1$-bundle over $\BB_2$, with projection $\mu_2:P\to \BB_2$;
\item $\mu_2:P\to \BB_2$ is a right $\G_2$-space and a principal $\Sigma_2$-bundle over $\BB_1$, with projection $\mu_1:P\to \BB_1$;
\item the two actions on $P$ commute.
\end{itemize}
Two Lie groupoids $\Sigma_1\tto \BB_1$ and $\Sigma_2\tto \BB_2$ are called {\bf Morita equivalent} if there exists a principal bibundle between them.
\end{definition}
\begin{proposition}
Morita equivalence is an equivalence relation.
\end{proposition}
\begin{proposition}
Let $\Sigma_1\tto \BB_1$ and $\Sigma_2\tto \BB_2$ be Lie groupoids. If $\Sigma_1$ and $\Sigma_2$ are isomorphic, then they are Morita equivalent. 
\end{proposition}
The notion of Morita equivalence makes precise the intuitive fact that a transitive Lie groupoid can be encoded completely in terms of its isotropy group, as Propositions~\ref{ex:ME_transitive} and~\ref{prp:ME_iso_isotropy} below show.
\begin{proposition}[Transitive groupoids and isotropy groups]\label{ex:ME_transitive}
A transitive Lie groupoid is Morita equivalent to its isotropy group at any point. More explicitly, if $\G\tto \BB$ is transitive then the $s$-fibre $s^{-1}(x)$ at $x\in \BB$ is a principal bibundle between $\G$ and the isotropy group $\G_x$.
\end{proposition}

For Lie groups the notion of Morita equivalence boils down to the ordinary notion of isomorphism.
\begin{proposition}\label{prp:ME_iso_isotropy}
Let $\Sigma_1\tto \BB_1$ and $\Sigma_2\tto \BB_2$ be Morita equivalent Lie groupoids, and $P$ be a bibundle between them with moment maps $\mu_1:P\to \BB_1$ and $\mu_2:P\to \BB_2$. For all $x_1\in \BB_1$ and $x_2 \in \BB_2$ such that there exists some $p\in P$ with $\mu_1(p) = x$ and $\mu_2(p)=y$, it holds 
\[
 (\Sigma_1)_{x_1}\cong (\Sigma_2)_{x_2}.
\]
In particular, two Lie groups $G_1$ and $G_2$ are Morita equivalent if and only if they are isomorphic.
\end{proposition}

As a corollary of the propositions above, one has
\begin{corollary}
Let $\Sigma_1\tto \BB_1$ and $\Sigma_2\tto \BB_2$ be transitive Lie groupoids. Then $\Sigma_1$ and $\Sigma_2$ are Morita equivalent if and only if their isotropy groups are isomorphic.
\end{corollary}

\begin{example}
By Example~\ref{ex:iso_cont}, $J^1\Gamma_{{\rm Sp}(k,1)}$ and $J^1\Gamma_{\rm cont}$ are isomorphic; hence, they are Morita equivalent.
\end{example}

Quite some geometry of a Lie groupoid $\Sigma\tto \BB$ is preserved under Morita equivalence; see, e.g.,~\cite[Lemma 1.28]{JOAO}. More in general, Morita equivalence has deep implications. We are interested mainly in the following proposition, which is well known (e.g.\ the proof of Theorem 4.2 in~\cite{XUMORITAEQ}). 
\begin{proposition}\label{prop:princ_ME}
The categories of principal bundles of Morita equivalent groupoids are equivalent.
\end{proposition}
Our starting point in this paper and in \cite{AccorneroCattafi1} is to exploit Proposition~\ref{prop:princ_ME} for the transitive Lie groupoids arising from transitive pseudogroups, and to investigate its geometric consequences. 

For future reference, let us spell out the main idea behind the proof of Proposition~\ref{prop:princ_ME}. In order to do so, let us recall that, if $P$ is a left $\Sigma$-space, one can form the {\bf action groupoid} $\Sigma\ltimes P\tto P$. The arrow space is given by $\Sigma\ltimes P=\Sigma\tensor[_s]{\times}{_\mu}P$, the source map is the second projection and the target map sends $(g, p)$ with $\mu(p)=s(g)$ to $g\cdot p$. The multiplicaton is defined by 
\[
(h,  g\cdot p)\cdot (g, p)=(hg, p)
\]
for all $h, g\in \Sigma$, $p\in P$ with $\mu(p)=s(g)$, $t(g)=s(h)$. Notice that, if $P$ is a right $\Sigma$-space, one has a similar notion of (right) action groupoid $P\rtimes \Sigma$.

Let now $P$ be a bibundle between $\Sigma_1\tto \BB_1$ and $\Sigma_2\tto \BB_2$ and $Q\to M_1$ be a principal $\Sigma_1$-bundle. One can construct a principal $\Sigma_2$-bundle as follows. If $\mu_1:P\to \BB_1$ is the moment map of the $\Sigma_1$-action on $P$ and $\tau:P_1\to \BB_1$ is the moment map of the $\Sigma_1$-action on $Q$, one considers the pullback
\[
\mu_1^* P_1 = \{(p,q):\ p\in P,\ p_1\in P_1,\ \mu_1(p) = \tau(q)\}.
\]
$\mu_1^* P_1$ carries a canonical right action of $P\rtimes \Sigma_2$ defined by
\[
(p,g)\cdot (p,q) = (pg,q),\quad (p,q)\in \mu_1^*P_1,\ (p,g)\in P\rtimes \Sigma_2.
\] 
Moreover, $\mu_1^*Q$ also carries the left diagonal action of the action groupoid $\Sigma\ltimes P\tto P$, defined by
\[
(g,p)\cdot (p,q) = (gp,gq),\quad (p,q)\in \mu_1^*P_1,\ (g,p)\in \Sigma\ltimes P.
\] 
Such diagonal action turns out to be principal. Let us denote its quotient by $\bar{Q}$. The $P\rtimes \Sigma_2$ right action induces a principal left $\Sigma_2$-action\footnote{Recall that any right action can be turned into a left action by pre-composing with the inversion, and viceversa.} on $\bar{Q}$, whose moment map is induced by 
\[
(p,q) \to \mu_2(p),\quad (p,q)\in \mu_1^*P_1.
\]

\subsubsection{Transitive geometries}\label{section_transitive_geometries}
Let us go back to Lie pseudogroups $\Gamma$ on a manifold $\BB$ and consider an almost $\Gamma$-structure of order $k$ on a manifold $M$ (Definition \ref{def:almost_gamma_struct}), i.e.\ a principal bundle
\[
\xymatrix{
J^k\Gamma \ar@<0.25pc>[dr] \ar@<-0.25pc>[dr]  & \ar@(dl, ul) &  P \ar[dl]^{\mu}\ar[dr]^{\pi} &    \\
& \BB  & &  M.}
\]
From now on, we assume $\Gamma$ to be transitive. As a consequence, the $k$-th jet groupoid $J^k\Gamma$ is transitive (Remark~\ref{rmk:transitivity_of_pseudo}). By making use of (the construction behind) proposition~\ref{prop:princ_ME} and the Morita equivalence from Example \ref{ex:ME_transitive}, we get the principal bundle
\[
\xymatrix{
G^k \ar@<0.25pc>[dr] \ar@<-0.25pc>[dr]  & \ar@(dl, ul) &  P_{\rm iso} \ar[dl]^{}\ar[dr]^{\pi} &    \\
&\{x \}  & &  M}
\]
where $G^k:=(J^k\Gamma)_x$ is the isotropy group at some point $x\in M$ and $P_{\rm iso}$ is obtained quotenting $P\times_\BB J^k_x\Gamma$, where $J^k_x\Gamma$ is the $s$-fibre over $x$, by the diagonal action of $J^k\Gamma\ltimes P$.

\begin{example}[Higher order $G$-structures]\label{ex_higher_order_G-structures}

For $\BB = \RR^n$ one has a more explicit way of constructing (a principal bundle isomorphic to) 
$P_{\rm iso}$: after fixing $x\in \BB$ (without loss of generality $x = 0$), we simply pick $P_x:=\mu^{-1}(x)$, together with the restriction of the $J^k\Gamma$-action to a $G^k := (J^k\Gamma)_x$-action. Then we see that 
\[
G^k := (J^k\Gamma)_x=\{j^k_x\varphi:\ \varphi\in \Gamma,\ \varphi(x)=x\}\subset \GL^k(n, \mathbb{R}).
\]
Here $\GL^k(n,\RR)$ denotes the space of $k$-jets of all diffeomorphisms of $\RR^n$ sending $x$ to $x$, which is naturally a Lie group and is called {\bf general $k$-th order group} or {\bf jet group of order $k$}.

Notice that, if the original almost $\Gamma$-structure $P$ is integrable (in the sense of Definition \ref{def:almost_gamma_struct}), we have $P=J^k\mathcal{A}$ and we can write 
\[
P_x:=\mu^{-1}(x)=\{j^k_p\phi:\ p\in M,\ \phi\in \mathcal{A},\ \phi(p)=x\}\subset \Fr^k(M).
\]
Here $\Fr^k(M)$ denotes the space of $k$-jets of all diffeomorphisms $M \to \RR^n$ sending their base points to $x$; it is naturally a principal $\GL^k(n,\RR)$-bundle and is called {\bf frame bundle of order $k$}.

In other words, $P_x$ is a {\bf $G^k$-structure of order $k$} on $M$, i.e.\ a reduction of $\Fr^k(M)$ to a Lie subgroup $G^k \subset \GL^k(n, \mathbb{R})$; 
more precisely $P_x$ is given by those $k$-frames which are induced by the $\Gamma$-atlas. Such reductions of higher order frame bundles are well known to be a geometric mean of encoding structure on manifolds: important examples include affine, conformal or projective structures (see \cite[Section I.8]{Kobayashi95}).

Even when the almost $\Gamma$-structure $P$ is non-integrable, 
$P_{\rm iso}\cong P_x:=\mu^{-1}(x)$ is still a reduction of $\Fr^k(M)$ to $G^k=(J^k\Gamma)_x$. We will return to the integrability problem for these geometric structures later on, after encoding them in the more general framework of Cartan bundles in section \ref{subsection_higher_order_G_structures}.
%
%
\end{example}

\begin{example}\label{ex:classical_G-struct}
Let $\Gamma=\Gamma_G$ (Example~\ref{ex:G-struct}), $G\subset \GL(n, \mathbb{R})$ being a Lie subgroup. By taking almost $\Gamma$-structures of order 1 and passing to the isotropy group as explained above we recover {\bf classical $G$-structures} \cites{CHERN, STERNBERG}. Going back to the flat structures presented in Example~\ref{ex:G-struct}, their corresponding almost versions are
\begin{itemize}
\item Riemannian metrics;
\item almost complex structures;
\item almost symplectic manifolds (i.e.\ manifolds with the choice of a non-degenerate $2$-form);
\item codimension $k$ distributions;
\item codimension $1$ distributions with the choice of a $2$-form which is non degenerate along the distribution. \qedhere
\end{itemize}
\end{example}

\begin{example}
As for $\Gamma_{\rm cont}$ (Example~\ref{ex:cont-struct}), the discussion from Example~\ref{ex:iso_cont} shows that the  $G$-structure corresponding to an almost $\Gamma_{\rm cont}$-structure of order 1 is a $\Sp(k,1)$-structure.
\end{example}

Let $\Gamma$ be a transitive pseudogroup on $\BB$ and $M$ be a manifold. The discussion of this subsection shows that we can reinterpret almost $\Gamma$-structures of order $k$ on $M$ as principal $(J^k\Gamma)_x$-bundles over $M$. However, the last two examples suggest that, in general, this operation leads to some ``loss of information''. For instance, even though the $G$-structures induced by almost structures of order $1$ for $\Gamma_{\rm cont}$ and $\Gamma_{\Sp(k,1)}$ are the same, the groupoids $J^1\Gamma_{\rm cont}$ and $J^1\Gamma_{\Sp(k,1)}$ (which are isomorphic, as discussed in Example \ref{ex:iso_cont}, but not equal) behaves differently when one looks at integrable structures. To partially recover such information, we look back at the jet groupoids. In the following subsection, we describe the additional structure that underlies the integrability problem.

%

\subsection{The Cartan distribution}\label{subs:cart-dist}
Recall that, when $R\subset J^k(M, N)$ is a PDE, a solution is a local section taking values in $R$ and being {\it holonomic}; that is, a map 
\[
x\in U\subset M \to J^k(M, N),\quad x\mapsto j^k_x f_x
\]
such that $f_x=f$. In simpler words, the section is a Taylor expansion of a map $f:U\to N$. Jet spaces come equipped with a canonical distribution detecting holonomic sections.

\begin{definition}\label{def:CD}
The {\bf Cartan distribution} on the jet space $J^k(M, N)$ is the regular distribution $\CC^k$ whose integral sections are precisely the holonomic sections.
\end{definition}
Such a distribution indeed exists and is unique. We recall that it can be characterised as the kernel of the vector valued 1-form 
\[
\omega^k\in \Omega^1(J^k(M, N), VJ^{k-1}(M, N)),\quad \omega^k_{j^k_x f}:=d\pr - d(j^{k-1} f) \circ ds,
\]
where $\pr:J^k(M, N)\to J^{k-1}(M, N)$ is the canonical projection and $VJ^{k-1}(M, N)$ is the vertical bundle of $s: J^{k-1}(M, N)\to M$.
\begin{lemma}
Given a Lie pseudogroup $\Gamma$ over $\BB$,  the {\rm \bf Cartan form} $\omega^k\in \Omega^1(J^k(\BB, \BB), VJ^{k-1}(\BB, \BB))$ restricts to $J^k\Gamma$ as a regular 1-form valued in $VJ^{k-1}\Gamma$. Consequently, the Cartan distribution $\CC^k$ restricts to $J^k\Gamma$ as a regular distribution. 
\end{lemma}
See, e.g.,~\cite[Section 3.4]{ORI} for the proof, and for further details. We will keep the notations $\omega^k$ and $\CC^k$ to indicate the restrictions to $J^k\Gamma$ of the Cartan form/distribution on $J^k(\BB, \BB)$.

The distribution $\CC^k$ and the groupoid structure on $J^k\Gamma$ (see above) interact nicely; quite remarkably, this was noticed relatively late~\cite{MARIA}. By ``nicely'' we mean that 
\begin{itemize}
\item $\CC^k$ is closed under multiplication (which is actually a consequence of the existence of local holonomic bisections around each point);
\item $\CC^k\cap \ker (ds)$ is an involutive regular distribution.
\end{itemize}

In the terminology of~\cite{MARIA}, $(J^k\Gamma, \CC^k)$ is a {\bf Pfaffian groupoid}, see Definition~\ref{def_Pfaffian_groupoid-form} below. Moreover, if 
\[
\xymatrix{
J^k\Gamma \ar@<0.25pc>[dr] \ar@<-0.25pc>[dr]  & \ar@(dl, ul) &  P \ar[dl]^{\mu}\ar[dr]^{\pi} &    \\
& \BB  & &  M}
\]
is an almost $\Gamma$-structure of order $k$ over $M$, see Definition~\ref{def:almost_gamma_struct}, then $P$ carries the restriction $\theta^k$ of the Cartan form on $J^k(M,\BB)$, which turns out to be regular. Moreover, $\theta^k$ interacts nicely with the action of $J^k\Gamma$, where ``nicely'' has a similar meaning to the one spelled out above for $\omega^k$ on $J^k\Gamma$.  When $\omega^k$ and $\theta^k$ are taken into account, the action of $J^k\Gamma$ on $P$ can be understood as a {\bf principal Pfaffian bundle}, see Definition~\ref{def:princ-pfaff-bundles}.

Pfaffian groupoids and principal Pfaffian bundles are rather powerful tools and, as discussed in details in \cite{INPROGRESS}, provide the ``correct'' abstract framework to encode and study, respectively, Lie pseudogroups of symmetries and the geometric structures induced by them. We refer to~\cite{CrainicSalazarStruchiner15,CRAINICYUDILEVICH,CATTMARIUSMARIA,AccCr2020} for instances of applications. For our purposes, we make use of results from~\cite{AccorneroCattafi1} on Morita equivalences of Pfaffian groupoids -- and of the correct notion of {\it Pfaffian isotropy} identified there. With such a notion and a Pfaffian version of Proposition~\ref{prop:princ_ME} at hand, we will be able to correctly encode the geometry arising from transitive pseudogroups in terms of principal group bundles.


\section{The Pfaffian framework}\label{sec:Pfaffian-language}

As discussed in the previous section, geometries on manifolds can be described using principal groupoid bundles. Moreover, the groupoids involved in such a description (and, in fact, the bundles themselves) are jet manifolds, and as such they carry additional structures: the Cartan distribution.

Below, we recall a general framework, introduced in~\cite{MARIA} and later developed in~\cites{ORI,FRANCESCO}, to deal with jet groupoids and their principal bundles from an abstract point of view. This approach is quite advantegeous for a variety of reasons. In particular, it allows us to decorate the facts presented in subsection~\ref{subs:ME} and to identify the appropriate notion of isotropy for the jet groupoid of a transitive pseudogroup.



\subsection{Pfaffian groupoids}

In this subsection we introduce the language needed to carry on the discussion. We limit ourselves to what is strictly needed. Extensive details can be found in~\cite{MARIA,ORI,FRANCESCO}; for what concerns Morita equivalences in the transitive case, see~\cite{AccorneroCattafi1}; the complete treatment will appear in \cite{INPROGRESS}.

Recall that a (left) {\bf representation} of a Lie groupoid $\Sigma\tto \BB$ is a vector bundle $E \to \BB$ such that
\begin{itemize}
\item $E$ is a $\Sigma$-space, with the moment map $E\to \BB$ being the vector bundle projection;
\item $\Sigma$ acts by linear maps (hence by linear isomorphisms) between fibres.
\end{itemize} 

\begin{definition}\label{def_Pfaffian_groupoid-form}
A {\bf Pfaffian groupoid} $(\G, \omega, E)$ over $\BB$ consists of a Lie groupoid $\G \rightrightarrows \BB$ together with a representation $E \to \BB$ of $\G$ and a differential form $\omega \in \Omega^1 (\G, t^* E)$ such that
\begin{enumerate}
\item $\omega$ is {\bf multiplicative} i.e.\
\[
(m^*\omega)_{(g, h)}=(\pr_1^*\omega)_{(g, h)}+g\cdot (\pr_2^*\omega)_{(g, h)},\quad (g, h)\in \G\times_\BB\G;
\]
\item $\omega$ has constant rank; 
\item the subbundle
\[
\g(\omega):=\ker(\omega)\cap \ker(ds) \subset T\G
\]
is involutive;
\item it holds
\begin{equation*}
\ker (\omega) \cap \ker(dt) = \ker (\omega) \cap \ker(ds). 
\end{equation*}
\end{enumerate}
We call $\g(\omega)$ the {\bf symbol bundle} of $(\G, \omega, E)$. A {\bf holonomic bisection} of $(\G, \omega, E)$ is a local bisection $\sigma:U\subset \BB \to \G$ such that $\sigma^*\omega=0$. 
A Pfaffian groupoid $(\G,\omega, E)$ is called {\bf full} if the form $\omega$ is pointwise surjective.
\end{definition}
For full Pfaffian groupoids, we will occasionally use the notation $(\Sigma,\omega)$, omitting the coefficient space $E$, which can be recovered (including the $\Sigma$-representation, see Proposition~\ref{prp:coefficent-space-is-alg}) from $\omega$.


\begin{example}[Main example]\label{exm:jet-groupoids}
Given a Lie pseudogroup $\Gamma$, the jet groupoid $J^k\Gamma\tto \BB$ is Pfaffian when equipped with the Cartan form $\omega^k$ from subsection~\ref{subs:cart-dist}. The symbol bundle $\g(\omega^k)$ is given by the vertical bundle of the projection $J^k\Gamma\to J^{k-1}\Gamma$ restricted to $\BB$ -- as one readily checks using the explicit formula for $\omega^k$. The Pfaffian groupoid $(J^k\Gamma, \omega^k)$ is full: $\omega^k$ is surjective onto the vertical bundle $VJ^{k-1}\Gamma$ of the source projection of $VJ^{k-1}\Gamma$. Observe that $VJ^{k-1}\Gamma$ is isomorphic to the quotient of $VJ^k\Gamma$ with respect to the vertical bundle of $J^k\Gamma\to J^{k-1}\Gamma$ -- see also Proposition~\ref{prp:coefficent-space-is-alg}.

\end{example}

\begin{remark}\label{symbol_space_is_ideal}
By point $(iii)$ in Definition~\ref{def_Pfaffian_groupoid-form}, the subbundle
\[
\g_\BB:=\g(\omega)|_\BB\subset \ker(ds)|_\BB
\]
is a Lie subalgebroid of $A:={\rm Lie}(\Sigma)$. 
Actually, $\g_\BB$ is an {\it ideal} of $A$ (Proposition~\ref{prp:coefficent-space-is-alg}, point $(ii)$).
\end{remark}

We list below some properties of Pfaffian groupoids that are relevant for our discussion; see~\cite{MARIA} for details. In what follows, we denote by $T_\BB\Sigma$ the pullback of $T \Sigma \to \Sigma$ via the unit map $\BB \to \Sigma$.

\begin{proposition}\label{prp:coefficent-space-is-alg}
Let $(\Sigma, \omega, E)$ be a Pfaffian groupoid, and set $\g_\BB:=\g(\omega)|_\BB$.
\begin{enumerate}
\item The quotient $A/\mathfrak{g}_\BB\to \BB$ is a representation of $\Sigma\tto \BB$. The action is given by
\[
g\cdot [V_x] \mapsto [dR_{g^{-1}}\cdot dm({\rm hor}^\omega_g(dt(V_x)), V_x)]
\]
where $g\in \Sigma$, $x=s(g)$, $[V_x]\in \ker(ds)|_\BB/\g_\BB$, and ${\rm hor}^\omega_g(v_x)$ is any $s$-lift of $v_x$ tangent to $\ker(\omega)$.
\item $\g_\BB$ is an ideal of $A$. Equivalently, the quotient $A/\g_\BB$ possesses a Lie algebroid structure such that the quotient projection is a Lie algebroid map.
\item The inclusion $\omega(T_\BB\Sigma)\hookrightarrow E$ is a map of representations -- i.e.\ the bundle $\omega(T_\BB\Sigma)$ is a subrepresentation of $E$.
\item The $\Sigma$-representation on $\omega(T_\BB\Sigma)$ from point $(iii)$ is isomorphic to the $\Sigma$-representation on $A/\mathfrak{g}_\BB$ from point $(i)$. The isomorphism is given by
\[
\omega(T_\BB \Sigma)\to A/\mathfrak{g}_\BB,\quad \omega(v_g)\mapsto [v^{\ker(ds)}_g],
\]
where $v_g\in T_g\Sigma$ and $v_g=v^{\ker(ds)}_g+v_g^{\ker(\omega)}$, $v^{\ker(ds)}_g\in \ker_g(ds)$, $v_g^{\ker(\omega)}\in \ker_g(\omega)$.
\end{enumerate}
\end{proposition}
\begin{remark}
As a consequence of Proposition~\ref{prp:coefficent-space-is-alg}, {\it full} Pfaffian groupoids $(\Sigma, \omega, E)$ are completely encoded into the Lie groupoid $\Sigma$ and the form $\omega$; the coefficient space $E$ can be reconstructed as $E \cong A/\g_\BB$ with the representation from Proposition~\ref{prp:coefficent-space-is-alg}, point $(i)$. Moreover, the coefficient space of a {\it full} Pfaffian groupoid $(\Sigma, \omega, E)$ carries a canonical Lie algebroid structure, and in fact a quotient of $A:={\rm Lie}(\Sigma)$, by Proposition~\ref{prp:coefficent-space-is-alg}, points $(ii)$ and $(iv)$. 
\end{remark}

A full Pfaffian groupoid $(\G,\omega,E)$ can be equivalently encoded into the distribution
$\CC_\omega:=\ker(\omega)$. More precisely, the following proposition holds true.

\begin{proposition}\label{prp:dist-picture}
Let $(\G,\omega,E)$ be a Pfaffian groupoid. The distribution $\CC_\omega := \ker(\omega)$ is
\begin{itemize}
\item {\bf multiplicative} in the sense that $\CC_\omega\subset T\Sigma$ is a subgroupoid (the right hand side denotes the tangent groupoid $T\Sigma\tto T\BB$, where the structure maps are obtained differentiating those of $\Sigma\tto \BB$);
\item {\bf $s$-transversal} in the sense that $\CC_\omega+\ker(ds)=T\G$;
\item such that $\g(\CC_\omega):=\CC_\omega\cap \ker(ds)$ is involutive;
\item such that $\CC_\omega\cap \ker(ds)=\CC_\omega\cap \ker(dt)$.
\end{itemize}
Moreover, given a Lie groupoid $\Sigma$ and a distribution $\CC$ with the properties listed above, one has 
\[
T\Sigma/\CC\cong \ker(ds)\cong \g(\CC)
\] 
and the normal bundle projection $T\Sigma\to T\Sigma/\CC$ post-composed with the right multiplication induces the pointwise surjective Pfaffian form
\[
\omega_\CC: T\Sigma\to A/\left(\g(\CC)|_\BB\right).
\]
\end{proposition}
\begin{remark}
Let $(J^k\Gamma,\omega^k)$ be the Pfaffian groupoid from Example~\ref{exm:jet-groupoids}. The distribution $\CC_{\omega^k} = \ker(\omega^k)$ is the (restriction of the) Cartan distribution $\CC^k$ of $J^k\Gamma$, see Definition~\ref{def:CD}.
\end{remark}

\begin{remark}
One can equivalently define Pfaffian groupoids (including non-full ones) in terms of distributions, see e.g. \cite[Definition 2.9]{AccorneroCattafi1}. 
\end{remark}

\begin{remark}
A remark about our choice of terminology compared to the literature: in \cites{MARIA, ORI, FRANCESCO} a Pfaffian groupoid is not required to satisfy $\ker(\omega) \cap \ker(dt) = \ker(\omega) \cap \ker(ds)$. Pfaffian groupoids satisfying $\ker(\omega) \cap \ker(dt) = \ker(\omega) \cap \ker(ds)$ are called ``Lie-Pfaffian''. In such setting, point $(ii)$ of Proposition \ref{prp:coefficent-space-is-alg} does not hold. For our purpose it is more convenient to require the ``Lie-Pfaffian'' condition directly in the definition. These phenomena are better investigated in \cite{INPROGRESS}.
\end{remark}

A {\bf morphism of Pfaffian groupoids} is, as expected, a Lie groupoid morphism together with a morphism of representation, such that the 1-forms are preserved (see \cite[Definition 2.11]{AccorneroCattafi1} for the precise definition).


\subsection{Principal Pfaffian bundles}

Let $(\G, \omega, E)$ be a Pfaffian groupoid over $\BB$ (Definition \ref{def_Pfaffian_groupoid-form}).

\begin{definition}\label{def:princ-pfaff-bundles}
A {\bf principal $(\G, \omega, E)$-bundle} over $M$ is a principal $\G$-bundle $\pi:P\to M$ with moment map $\mu:P\to \BB$ together with a 
1-form $\theta\in \Omega^1(P, \mu^*E)$ such that 
\begin{itemize}
\item the form $\theta$ is {\bf multiplicative} with respect to the action $m_P:\G\times_\BB P\to P$ of $\G$ and the form $\omega$, i.e.\
\[
(m_P^*\theta)_{(g, p)}=(\pr_1^*\omega)_{(g,p)}+g\cdot (\pr_2^*\theta)_{(g, p)},\quad (g, p)\in \G\times_\BB P;
\]
\item $\ker(d\mu)\cap \ker(\theta)=\ker(d\pi)\cap \ker(\theta)$.
\end{itemize}
The principal bundle is called {\bf full} if $\theta\in \Omega^1(P, \mu^*E)$ is pointwise surjective. A local section 
\[
\sigma:U\subset M\to P
\]
is called {\bf holonomic} if $\sigma^*\theta=0$.

When the Pfaffian groupoid $(\G, \omega, E)$ is clear from the context, we will often use the simpler terminology {\bf principal Pfaffian bundles}.
\end{definition}

We will picture a principal $(\G, \omega, E)$-bundle as
\[
\xymatrix{
(\Sigma, \omega, E) \ar@<0.25pc>[dr] \ar@<-0.25pc>[dr]  & \ar@(dl, ul) &  (P, \theta) \ar[dl]^{\mu}\ar[dr]^{\pi} &    \\
&\BB  & &  M.}
\]

\begin{example}[Main example]\label{ex_almost_gamma_structures_as_pfaffian_bundle}
If $\Gamma$ is a Lie pseudogroup over $\BB$, an almost $\Gamma$-structure of order $k$ (Definition~\ref{def:almost_gamma_struct})
\[
\xymatrix{
J^k\Gamma \ar@<0.25pc>[dr] \ar@<-0.25pc>[dr]  & \ar@(dl, ul) &  P \ar[dl]^{\mu}\ar[dr]^{\pi} &    \\
&\BB  & &  M,}
\]
can be given the structure of principal $J^k\Gamma$-bundle by taking the restriction to $P$ of the Cartan distribution on $J^k(M, \BB)$, see subsection~\ref{subs:cart-dist}. If $P$ is integrable, then $P=J^k\mathcal{A}$ for some $\Gamma$-atlas $\mathcal{A}$ on $M$, Definition~\ref{def:almost_gamma_struct}. In this case, around each point $j^k_x\phi\in J^k\mathcal{A}$ we have the section
\[
y\in \dom(\phi)\subset M\mapsto j^k_y\phi,
\]
which is holonomic. This property is in fact equivalent to integrability of the $\Gamma$-structure $P$. That is, if, for each $p$, one has a holonomic section of $(P, \theta)$ through $p$, then $P\subset J^k(M, \BB)$ is integrable in the sense of Definition~\ref{def:almost_gamma_struct}. In fact, given an open cover $\{U_i\}_{i\in I}$ of $M$ such that for each $U_i$ one has a holonomic section $\sigma:U\subset M\to P\subset J^k(M,\BB)$, one can construct the $\Gamma$-atlas $\{(U_i, \mu\circ \sigma_i)\}_{i\in I}$.
\end{example}

\begin{remark}\label{rmk:inf_action_Pfaffian}
Given an action $m_P$ of a Lie groupoid $\Sigma$ on $\mu:P\to \BB$, one has an induced {\it infinitesimal action}, given by the map
\[
a: t^*A\to \ker(d\pi)\subset TP
\]
sending $\alpha_x\in A_{t(g)}=\ker_{t(g)}(ds)$ to $dm_P(\alpha_x, 0_p)\in \ker_p(d\pi)\subset T_pP$. When $P$ is a principal $\Sigma$-bundle over $M$, the infinitesimal action induces an isomorphism 
\[
\mu^*A \to \ker (d\pi), \quad \alpha_{\mu(p)}\mapsto dm_P(\alpha_{\mu(p)}, 0_p).
\]
Let now $(P,\theta)$ be a principal $(\G, \omega, E)$-bundle. The multiplicativity of $\theta$ implies $\theta(a(\alpha))=\omega(\alpha)$, for all $\alpha\in \ker(ds) \cong t^*A$. As a consequence, the isomorphism infinitesimal action induces an isomorphism $\mu^*A \cong \ker(d\pi)$ that restricts to an isomorphism
 \[
\mu^*(\ker(\omega)\cap\ker(ds))\cong \ker(\theta)\cap \ker(d\pi).
\]
We conclude that $\ker(d\pi)\cap \ker(\theta)$ is involutive since $\g(\omega) = \ker(\theta)\cap \ker(d\pi)$ is so.
\end{remark}

%

\subsection{Pfaffian isotropies and Pfaffian groups}\label{subs:Pf-groups}

We are interested into studying the structure carried by the isotropy groups of jet groupoids of Lie pseudogroups. We approach the problem by looking at the isotropy groups of Pfaffian groupoids. Let $(\Sigma,\omega,E)$ be a Pfaffian groupoid. The first remark is that the restriction of the Pfaffian form $\omega$ to the isotropy group $\Sigma_x$ is itself a Pfaffian form. In fact, from the multiplicativity of $\omega$ follows a suitable equivariance property (see \cite[Remark 2.17]{AccorneroCattafi1} and also~\cite[Lemma 3.4.11]{FRANCESCO}), which in turn can be used to prove (see~\cite[Proposition 2.18]{AccorneroCattafi1}) that

\begin{proposition}[Isotropies of Pfaffian groupoids]\label{prp:isotropies-pf-groups}
Let $(\Sigma, \omega, E)$ be a Pfaffian groupoid over $\BB$. Then the triple $(\Sigma_x, \omega|_{\Sigma_x}, E_x)$ is a Pfaffian group. Furthermore for all $x, y\in \BB$ and $g\in \Sigma$ such that $s(g)=x$, $t(g)=y$ the isomorphism of Lie groups
\[
\phi_g: \Sigma_x\to \Sigma_y, \quad h\mapsto ghg^{-1}
\]
and the isomorphism of representations
\[\psi_g: E_x\to E_y, \quad \alpha_x\mapsto g\cdot \alpha_y\]
make $(\psi_g, \phi_g)$ into an isomorphism of Pfaffian groups from $(\Sigma_x, \omega|_{\Sigma_x}, E_x)$ to $(\Sigma_y, \omega|_{\Sigma_y}, E_y)$.
\end{proposition}

\begin{remark}\label{remark:isotropy-not-full-not-Lie}
Let $(\Sigma, \omega, E)$ be a {\it full} Pfaffian groupoid over $\BB$ -- i.e.\ $\omega$ is pointwise surjective. Its isotropy $(\Sigma_x, \omega|_{\Sigma_x}, E_x)$ is almost never a full Pfaffian group; that happens if and only if the Lie groupoid $\Sigma$ is a bundle of Lie groups over $\BB$.
\end{remark}

The Pfaffian group $(\Sigma_x,\omega_x,E_x)$ should be regarded as the {\bf Pfaffian isotropy group} of $(\Sigma,\omega,E)$ at $x\in \BB$. This is clarified in the next subsection: once the appropriate notion of Pfaffian Morita equivalence is used (Definition~\ref{def:princ-Pfaff-bibundle}), one can show that the Pfaffian isotropy groups of Morita equivalent Pfaffian groupoids are Pfaffian isomorphic, and that a transitive Pfaffian groupoid is Pfaffian Morita equivalent to its Pfaffian isotropy group, see Proposition~\ref{lemma:Pfaffian_ME_of_groups}.

Since we will extensively works with Pfaffian isotropies and, more in general, with Pfaffian groups, we present below an equivalent characterisation and some basic properties.

\begin{proposition}[Characterisation of Pfaffian groups]\label{prp:char-pf-groups}
Given a Lie group $G$ and a representation $V$, there is a one to one correspondence
\[
 \left\{ \text{Pfaffian group structures } (G, \omega, V) \right\} \iff \left\{ \text{representation maps } l:\g\to V \right\}.
\]
\end{proposition}
The correspondence is given by $l: = \omega_e$ and by $\omega: = l \circ \omega_{\rm MC}$, where $\omega_{\rm MC} \in \Omega^1 (G,\g)$ is the usual Maurer-Cartan form of a Lie group (defined using right-multiplication). Notice also that, under this equivalence, the symbol bundle $\g(\omega)$ corresponds to $\ker(l)$. For more details, see~\cite{AccorneroCattafi1}.

The following result can be easily derived from the basic properties of Pfaffian groupoids, see Propositions~\ref{prp:coefficent-space-is-alg} and~\ref{prp:dist-picture}.

\begin{proposition}\label{prop_defining_h}
Let $(G,\omega,V)$ be a Pfaffian group. Then its symbol bundle at the identity
\[
 \h:= \ker(\omega_e) \subset T_e G = \g
\]
\begin{enumerate}
 \item is a Lie ideal of $\g$;
 \item coincides with the kernel of $l: \g \to V$ from Proposition \ref{prp:char-pf-groups};
 \item encodes the entire distribution $\ker(\omega) \subset TG$, i.e.\ for any $g \in G$ one recovers $\ker(\omega_g) = d_e R_g (\h)$.
\end{enumerate}
\end{proposition}

In view of this result, we will sometimes denote Pfaffian groups also by $(G,\h,V)$. 


\subsection{Pfaffian Morita equivalence}

Let $\Sigma\tto \BB$ be a Lie groupoid. We recall once more that, when $\mu:P\to \BB$ is a left $\Sigma$-space, we can form the {\bf action groupoid} $\Sigma\ltimes P\tto P$. The arrow space is given by $\Sigma\ltimes P=\Sigma\tensor[_s]{\times}{_\mu}P$, the source map is the second projection and the target map sends $(g, p)$ with $\mu(p)=s(g)$ to $g\cdot p$. The multiplicaton is defined by 
\[
(h,  g\cdot p)\cdot (g, p)=(hg, p)
\]
for all $h, g\in \Sigma$, $p\in P$ with $\mu(p)=s(g)$, $t(g)=s(h)$. Similarly, when $\mu:P\to \BB$ is a right $\Sigma$-space, one can define the action groupoid $P\rtimes \Sigma$.

\begin{definition}\label{def:princ-Pfaff-bibundle}
Let $(\Sigma_1, \omega_1, E_1)$ and $(\Sigma_2, \omega_2, E_2)$ be Pfaffian groupoids over $\BB_1$ and $\BB_2$ respectively. A {\bf principal Pfaffian bibundle} between them is a triple $(P, \theta, \Phi)$ where
\begin{itemize}
\item $(P, \theta)$ is a left principal $(\Sigma_1, \omega_1, E_1)$-bundle;
\item $\Phi:\mu_1^*E_1\to \mu_2^*E_2$ is an isomorphism of vector bundles;
\end{itemize} 
such that 
\begin{itemize}
\item $(P, \Phi\circ \theta)$ is a right principal $(\Sigma_2, \omega_2, E_2)$-bundle;
\item the two actions on $P$ commute;
\item $\Phi:\mu_1^*E_1\to \mu_2^*E_2$ is a morphism of
\begin{itemize} 
\item left $\Sigma_1\ltimes P$-representations, where $\mu_1^*E_1$ is the representation induced by the $\Sigma_1$-representation on $E_1$ and $\mu_2^*E_2$ carries the trivial representation;
\item right $P\rtimes \Sigma_2$-representations, where $\mu_2^*E_2$ is the representation induced by the $\Sigma_2$-representation on $E_2$ and $\mu_1^*E_1$ carries the trivial representation.
\end{itemize}
\end{itemize}
Two Pfaffian groupoids $(\G_1, \omega_1, E_1)$ and $(\G_2, \omega_2, E_2)$ are called {\bf Pfaffian Morita equivalent} if there exists a principal Pfaffian bibundle between them. 

A principal Pfaffian bibundle/Morita equivalence is called {\bf full} if $\theta$ is pointwise surjective.
\end{definition}

The definition of Pfaffian Morita equivalence appeared first in~\cite[Section 5.4]{FRANCESCO}. There, it is shown that
\begin{proposition}
Pfaffian Morita equivalence is an equivalence relation.
\end{proposition}

The evidence that Definition~\ref{def:princ-Pfaff-bibundle} should be regarded as a correct Pfaffian enhancement of the notion of Morita equivalence is provided by the fact that, if one adopts Definition~\ref{def:princ-Pfaff-bibundle}, the Pfaffian version of Proposition~\ref{prop:princ_ME} holds true. This is Theorem $3.15$ in~\cite{AccorneroCattafi1}. Explicitly, we have:

\begin{proposition}\label{prop:Pfaffian_principal_category}
Let $(\G_1, \omega_1, E_1)$ and $(\G_2, \omega_2, E_2)$ be Pfaffian groupoids. If they are Pfaffian Morita equivalent, there is a
one to one correspondence
\[ 
 \left\{   \begin{array}{c}
            \text{Principal } (\G_1, \omega_1, E_1) \text{-bundles}\\
             \text{up to isomorphisms}
            \end{array} 
\right\} 
\stackrel{1-1}{\longleftrightarrow}
\left\{   \begin{array}{c}
           \text{Principal } (\G_2, \omega_2, E_2) \text{-bundles}\\
           \text{up to isomorphism}
           \end{array} 
\right\}.
\]\
\end{proposition}

We are interested into applying the notion of Pfaffian Morita equivalence to transitive Lie pseudogroups -- hence to transitive Pfaffian groupoids.

We start by observing that, just as isomorphic Lie groupoids are Morita equivalent, another (trivial) example of Pfaffian Morita equivalence is given of course by Pfaffian isomorphism (see \cite[Lemma 3.11]{AccorneroCattafi1}):

\begin{proposition}\label{lemma:morphism_which_are_PME}
Let $(\G_1, \omega_1, E_1)$, $(\G_2, \omega_2, E_2)$ be Pfaffian groupoids over $\BB_1$, $\BB_2$. An isomorphism of Pfaffian groupoids between them induces naturally a principal Pfaffian bibundle. If $(\G_1, \omega_1, E_1)$ is full, the Pfaffian bibundle is full as well.
\end{proposition}

Notice now that a {\it Pfaffian} Morita equivalence of Pfaffian groupoids is in particular a Morita equivalence. Hence, Proposition~\ref{prp:ME_iso_isotropy} implies that two Morita equivalent transitive Pfaffian groupoids have isomorphic isotropy groups. However, a {\it Pfaffian} Morita equivalence preserves more structure than an ordinary Morita equivalence. In particular, the {\it Pfaffian} isotropy groups of Pfaffian Morita equivalent Pfaffian groupoids are Pfaffian isomorphic; together with Proposition~\ref{lemma:morphism_which_are_PME}, this implies the following:

\begin{proposition}\label{lemma:Pfaffian_ME_of_groups}
Let $(\Sigma_1,\omega_1,E_1)$ be a Pfaffian groupoid over $\BB_1$ and $(\Sigma_2,\omega_2,E_2)$ be a Pfaffian groupoid over $\BB_2$, and let $(P,\theta)$ be a principal Pfaffian bibundle between them with moment maps $\mu_1:P\to \BB_1$ and $\mu_2:P\to \BB_2$. For all $x_1\in \BB_1$ and $x_2\BB_2$ such that there exists $p\in P$ with $\mu_1(p) = x_1$ and $\mu_2(p)=x_2$, it holds $((\Sigma_1)_{x_1},(\omega_1)_{x_1},(E_1)_{x_1})\cong ((\Sigma_2)_{x_2},(\omega_2)_{x_2},(E_2)_{x_2})$. In particular, two Pfaffian groups are Pfaffian Morita equivalent if and only if they are Pfaffian isomorphic.
\end{proposition}


%

We now restrict our attention to Pfaffian groupoids $(\Sigma,\omega,E)$ over $\BB$ which are {\it transitive}, that is,
\[
\text{for all }x, y\in \BB \text{ there is } g \in \G \text{ such that } s(g)=x,\ t(g)=y .
\]
Recall also that that a transitive Lie groupoid $\G\tto \BB$ is Morita equivalent to its isotropy Lie group $\G_x$, for any $x\in \BB$; see Proposition~\ref{ex:ME_transitive}. The bibundle realising the equivalence between $\G$ and the isotropy $\G_x$ at (say) $x\in \BB$ is the $s$-fibre $s^{-1}(x)$ over $x$. A Pfaffian version of this statement holds if one uses the Pfaffian isotropy described in subsection~\ref{subs:Pf-groups}.

\begin{proposition}\label{prop_transitive_Pfaffian_Morita_equivalence}
Any full transitive Pfaffian groupoid $(\G, \omega, E)$ is Pfaffian Morita equivalent to the non full Pfaffian group(oid) $(G, \bar{\omega}_{\rm MC}, V)$, where 
\begin{itemize}
\item $G=\Sigma_x$ is the isotropy group of $\G$ at $x\in \BB$;
\item $V=E_x$;
\item $\bar{\omega}_{\rm MC} = l \circ \omega_{\rm MC}$ is the composition of the Maurer-Cartan form $\omega_{\rm MC}:G\to \mathfrak{g}$ and the linear map $l: \mathfrak{g}\to V$ induced by $\omega_x$ (see Proposition~\ref{prp:char-pf-groups}).
\end{itemize}
\end{proposition}
Of course, thanks to Propositions~\ref{prp:isotropies-pf-groups} and~\ref{lemma:Pfaffian_ME_of_groups}, the Pfaffian Morita equivalence class of $(G, \bar{\omega}_{\rm MC}, V)$ is independent from the choice of $x\in \BB$. Therefore, it follows from Propositions~\ref{prop:Pfaffian_principal_category} that

\begin{corollary}\label{corollary_combination_props}
For any transitive full Pfaffian groupoid $(\G, \omega, E)$, with isotropy Pfaffian group $(G,\h,V)$ (see Proposition \ref{prop_transitive_Pfaffian_Morita_equivalence}), there is a one to one correspondence
\[ 
 \left\{   \begin{array}{c}
            \text{Principal } (\G, \omega, E) \text{-bundles}\\
             \text{up to isomorphisms}
            \end{array} 
\right\} 
\stackrel{1-1}{\longleftrightarrow}
\left\{   \begin{array}{c}
           \text{Principal } (G, \h, V) \text{-bundles}\\
           \text{up to isomorphism}
           \end{array} 
\right\}.
\]
\end{corollary}

Corollary~\ref{corollary_combination_props}, combined with Example~\ref{exm:jet-groupoids} and Remark~\ref{rmk:transitivity_of_pseudo}, provides the motivation needed for studying principal $(G, \h, V)$-bundles or, in the terminology of the next section, {\it Cartan bundles}.

%

\section{Cartan bundles}

In view of Corollary \ref{corollary_combination_props}, in order to study geometric structures defined by transitive pseudogroups, it is enough to study principal Pfaffian bundles for the action of their isotropy (Pfaffian) groups. We can therefore finally arrive to our main definition.

\begin{definition}\label{def:Cartan-bundle}
A {\bf Cartan bundle} on $M$ is a principal $(G,\h, V)$-bundles, for some Pfaffian group $(G,\h, V)$. We will sometimes write {\bf Cartan $(G,\h, V)$-bundle} if we want to emphasise the Pfaffian group $(G,\h, V)$.
\end{definition}

To implement this idea, one needs to allow for $(G,\h, V)$ to be not necessairly full, see Remark~\ref{remark:isotropy-not-full-not-Lie}. On the other hand, our purpose is to study the geometry of ``concrete'' geometric structures, described by jet bundles with their Cartan forms. Accordingly, it is not restrictive to assume the following:

\begin{axiom}{F}\label{axiom_full}
From now on, unless otherwise specified, principal $(G, \h, V)$-bundles $(P, \theta)$ are assumed to be full -- i.e.\ $\theta \in \Omega^1 (P,V)$ is assumed to be pointwise surjective, see Definition~\ref{def:princ-pfaff-bundles}. 
\end{axiom}

What we proved until now allows us to recover the original definition of Cartan bundle in \cite{FRANCESCOPAP}.

\begin{corollary}\label{cor_Cartan_bundles_via_forms}
A Cartan bundle is equivalently defined by a (left) principal $G$-bundle $P$ over $M$ together with a representation $V \in \Rep(G)$, a map of representations $l:\g\to V$ and a differential form $\theta \in \Omega^1(P, V)$ such that 
\begin{itemize}
 \item $\theta$ is surjective
 \item $\theta$ is $G$-equivariant;
 \item $\theta (\alpha^\dagger_p)=l(\alpha)$ $\forall p \in P$; 
 \item $\ker(\theta) \subset \ker(d\pi)$;
 \item $\ker(\theta)$ is involutive.
\end{itemize}
\end{corollary}

Recall that $\alpha^\dagger \in \mathfrak{X}(P)$ denotes the {\bf fundamental vector field} $p \mapsto a_p(\alpha)$ associated to the infinitesimal action of $\alpha \in \g$ on $P$.

\begin{proof}
Let $(P,\theta)$ be a principal $(G,\h,V)$-bundle. Recall that $\h$ is the kernel of a representation map $l:\g\to V$, and that such representation map is precisely $\omega_e:\g\to V$, where $\omega\in \Omega^1(G,V)$ is the Pfaffian form on $G$. With this in mind, notice that $\theta$ is pointwise surjective by assumption, i.e. Axiom \ref{axiom_full}. From the multiplicativity of $\theta$ it follows that $\theta$ is $G$-equivariant and that 
\[
\theta (\alpha^\dagger_p)=l(\alpha)\quad  \forall p \in P.
\]
By the second point of Definition~\ref{def:princ-pfaff-bundles}, one has that $\ker(\theta)\subset \ker(d\pi)$. To prove that $\ker(\theta)$ is involutive, one makes use of Remark~\ref{rmk:inf_action_Pfaffian}.

On the other hand, if $(P,\theta)$ has the properties listed in the corollary, Proposition~\ref{prp:char-pf-groups} allows to endow $G$ with a $V$-valued Pfaffian form $\omega$ such that $\ker(\omega_e) = \h$ and $(P,\theta)$ is a principal $(G,\h,V)$-bundle.
\end{proof}

Due to its prominence in the next sections, we discuss the foliation on $P$ defined by involutive distribution $\ker(\theta) \subset \ker(d\pi)$, and its relation with the ideal $\h$. Recall that a foliation $\mathcal{F}$ on $P$ is {\bf vertical} if its leaves are tangent to the vertical bundle of $P \to M$.

\begin{proposition}[Foliation underlying a Cartan bundle]\label{properties_ideal_k}
Given a Cartan $(G,\h,V)$-bundle $(P,\theta)$,
\begin{enumerate}
 \item $\h= \{ \alpha \in \g \mid \alpha^\dagger \in \Gamma(\ker(\theta)) \}$;
 \item $\h$ is a subrepresentation of $\g$ (endowed with the adjoint $G$-representation);
 \item the vertical foliation $\mathcal{F} = \ker(\theta)$ coincides with the foliation $\mathcal{F}_\h$ given by the image of $\h \subset \g$ under the infinitesimal action on $P$.
\end{enumerate}
\end{proposition}

This shows that, of course, one could also take part (i) as an alternative definition of $\h$ in terms of the Cartan bundle $(P,\theta)$ (and not of the Pfaffian group), and prove directly that $\h$ is a Lie ideal of $\g$ (by checking that it is $\Ad_g$-invariant for every $g \in G$).

\begin{proof} 
The first part follows by combining Corollary \ref{cor_Cartan_bundles_via_forms} with part (ii) of Proposition \ref{prop_defining_h}.

The second claim holds by the $G$-equivariance of $\theta$. Inded, for any $\alpha \in \g$, $g \in G$ one has
 \[
  \theta_p ( (g \cdot \alpha)^\dagger_p ) = \theta_p (a_p (d_e C_g (\alpha))) = \theta_p (d_{pg} R_{g^{-1}} (a_{pg} (\alpha)) ) =
  \]
  \[
  = (R_{g^{-1}}^* \theta)_{pg} (a_{pg}(\alpha)) = g \cdot \theta_{pg} (\alpha^\dagger_{pg}) = g \cdot \theta_p (\alpha^\dagger_p),
 \]
hence $\alpha \in \h$ if and only if $g \cdot \alpha \in \h$.

For the third part, if $L \subset P$ is a leaf of $\mathcal{F}_\h$ through $p \in P$, then $T_q L = a_q (\h)$ for every $q \in L$. But since $\h^\dagger_q = \ker(\theta_q)$, we have precisely $T_q L = \ker(\theta_q)$, i.e.\ $L$ is a leaf of the foliation $\mathcal{F} = \ker(\theta)$.
\end{proof}

\subsection{Examples}\label{exm:G-struct-as-C-bundles}

\subsubsection{G-structures}\label{subsection_G_structures}
We start looking at the case when $\F_\h$ consists of the whole foliation by fibres of $P\to M$; this happens if and only if $\h=\g$. We refer to this as {\bf simple} case, since $\F_\h$ is a simple foliation on $P$ and the leaf space is canonically diffeomorphic to $M$. From the form point of view, in this case $\theta$ is an equivariant form on $P$ vanishing on vertical vectors; since we ask for it to be pointwise surjective, $\dim(V)=\dim(M)$. 

This is the setting of {\bf abstract $G$-structures}, that is principal $G$-bundles equipped with a representation of $G$ on $\mathbb{R}^n$ and an equivariant $\mathbb{R}^n$-valued ($n=\dim(M)$) form whose kernel is the vertical bundle.
\begin{proposition}\label{simple_Cartan_bundles_are_G_structures}
Simple Cartan bundles are in one to one correspondence with abstract $G$-structures. 
\end{proposition}
One recovers classical $G$-structures by looking at the frame bundle $\pi: \Fr(M)\to M$ equipped with the {\bf tautological form} $\theta_{\rm taut}\in \Omega^1(\Fr(M), \mathbb{R}^n)$; recall that $\theta_p(V_p)$ is given by the components of $d\pi(V_p)$ with respect to the frame $p$. If $\pi:P\to M$ is a $G$-reduction, $G\subset \GL(n, \mathbb{R})$, the restriction of $\theta_{\rm taut}$ makes $P$ into a simple Cartan bundle.

\subsubsection{Cartan geometries}\label{subsection_Cartan_geometries}
On the other extremum of the spectrum we have the case when $\F_\h$ is the foliation by points; this happens if and only if $\h = \g$. We refer to this case as {\bf parallelisable}, since $\theta\in \Omega^1(P, V)$ is now an absolute parallelism on $P$; that is, we have $\dim(V)=\dim(P)$ and 
\[
\theta:TP\to P\times V
\]
is an isomorphism. Notice that $\theta$ is equivariant and it coincides with the identity on fundamental vectors, thanks to multiplicativity. In other words, $(P, \theta)$ is a Cartan geometry.
\begin{definition}\label{def_Cartan_geometry}
A {\bf Cartan geometry} on $M$ is the datum of 
\begin{itemize}
\item a principal $G$-bundle $P\to M$;
\item a $G$-representation $V$ containing the adjoint representation of $G$ on $\g$ as a subrepresentation;
\item an equivariant form $\theta\in \Omega^1(P, V)$, called {\bf Cartan connection}, which is a pointwise isomorphism and such that
\[
\theta(\alpha^\dagger)=\alpha, \quad \alpha\in \g.
\]
\end{itemize}
\end{definition}

Cartan geometries may be thought of as curved versions of Klein geometries and are extensively studied in the literature -- see e.g.~\cites{SHARPE, CAPSLOVAK,ALEKSEEVSKYMICHOR} for detailed introductions and further references. 

\begin{remark}[Lie brackets on $V$]\label{remark_bracket_Cartan_geometry}
Cartan geometries are usually defined with the additional requirement that $V$ is Lie algebra (notice that, under such assumption, the Lie algebra of $\mathfrak{g}$ is automatically a subalgebra, because we ask for the representation of $G$ on $V$ to extend the adjoint representation). However, such Lie algebra structure plays a role only for what concerns {\it flatness} for Cartan geometries; the concept of flatness itself is defined by means of a Maurer-Cartan equation satisfied by $\theta$. We prefer to think about Cartan geometries -- and, more in general, Cartan bundles -- as objects for which flatness is defined only {\it with respect to the choice of a bracket on $V$}; see subsection~\ref{section_integrability_Cartan_bundles}.
\end{remark}

\begin{proposition}
Parallelisable Cartan bundles are in one to one correspondence with Cartan geometries.
\end{proposition}

\begin{example}
Simple examples of Cartan geometries are given by classical $G$-structures $P\to M$ equipped with a compatible connection -- that is, a connection whose parallel transport sends frames in $P\subset \Fr(M)$ to frames in $P\subset \Fr(M)$. Such a connection can be encoded into a principal connection 1-form $\omega\in \Omega^1(P, \g)$ on the principal bundle $P$. The form $\theta=\theta_{\rm taut}+\omega$ is then a Cartan connection.
\end{example}

\subsubsection{Higher order $G$-structures}\label{subsection_higher_order_G_structures}

By taking $G$-structures $P\subset \Fr^k(M)\to M$ of order $k > 1$ (Example \ref{ex_higher_order_G-structures}), 
one gets Cartan bundles which are neither simple nor parallelisable.
To show this claim, we consider first the sequence of Lie groups 
\[
GL^k(n, \mathbb{R}) \xrightarrow{\tau^{k, k-1}} \GL^{k-1}(n, \mathbb{R}) \xrightarrow{\tau^{k-1, k-2}} \dots \xrightarrow{\tau^{2, 1}} \GL(n, \mathbb{R}).
\]
Rebaptising $G$ as $G^k$ and restricting the projection $\tau^{k, k-1}$ to $G^k$, we get a sequence
\[
G^k\to G^{k-1}\to \dots \to G^1
\]
and its infinitesimal counterpart
\[
\g^k\to \g^{k-1}\to \dots \to \g^1.
\]
We denote the kernel by $\g^{k, k-1}:=\ker(d\tau^{k, k-1})\subset \g^k$, so that
\[
\omega^k:= \omega_{\rm MC}\circ d\tau^{k, k-1}: TG^k\to \g^{k-1}
\] 
defines a $G^k$-equivariant surjective form $\omega^k \in \Omega^1 (G^k, \g^{k-1})$. When $k=1$ one gets the zero form. We can postcompose $\omega^{k}$ with the section $\g^{k-1}\to \g^{k-1} \times \RR^n$ sending $a\in \g^{k-1}$ to $(a, 0)$; the resulting form, which we keep denoting $\omega^{k}$, is in fact a Pfaffian form on $G^k$, with image given by $\g^{k-1}$ and kernel $\g^{k, k-1}$.

Similarly, consider the sequence of principal bundles
\[
\Fr^k(M) \xrightarrow{\pi^{k, k-1}} \Fr^{k-1}(M) \xrightarrow{\pi^{k-1, k-2}} \dots \xrightarrow{\pi^{2,1}} \Fr^1(M) \to M.
\]
Rebaptising $P$ as $P^k$, and restricting the projection $\pi^{k, k-1}$ to $P^k$, we get a sequence
\[
P^k\to P^{k-1}\to \dots \to P^1\to M.
\]
We consider then the higher-order tautological form $\theta^k \in \Omega^1 (\Fr^k(M), \g^{k-1} \times \RR^n)$ on $\Fr^k(M)\to M$, which is defined as in the $k=1$ case (see e.g.\ \cite{Kobayashi61} or \cite[Section IV.16]{KolarMichorSlovak93}), and which restricts to $P^k \subset \Fr^k(M)$ as
\[
\theta^k \in \Omega^1 (P^k, \g^{k-1} \times \RR^n), \quad \quad (\theta^k)_{p^k}:= \phi_{p^k}\circ d\pi^{k, k-1}: TP^k\to \g^{k-1} \times \RR^n.
\]
Here, $\phi_{p^k}$ denotes the left splitting of the sequence 
\[
0\to \ker \left(d \pi^{k-1}: TP^{k-1}\to TM \right)\to T P^{k-1}\to T M\to 0
\]
induced by the $k$-frame $p^k=j^k_x\psi \in P^k$, where $x\in M$ and $\psi:U\subset M\to \mathbb{R}^n$ is an isomorphism sending $x$ to $0$.

\begin{remark}\label{first_component_theta_2}
For future reference, notice that the tautological forms at two consecutive orders are related by
\[
(\pi^{k,k-1})^*\theta^{k-1} = (d\tau^{k-1, k-2}, \id_{\RR^n}) \circ \theta^k. 
\]
Applying such equation recursively, we see that the $\RR^n$-component of $\theta^k$ is simply the ordinary tautological form $\theta^1 = \theta_{\rm taut}$: more precisely
\[
 (\pi^{k,1})^*\theta_{\rm taut} = \pr_{\RR^n} \circ \theta^k. \qedhere
\]
\end{remark}

In order for $P^k$ to define a Cartan bundle, the last ingredient is the representation of $G^k$ on the space $V^k:= \g^{k-1} \times \RR^n$, which is induced by the adjoint representation of $G^{k-1}$ on $\g^{k-1}$:
\[
j^k_0 \phi \cdot (\alpha, v) = (j^{k-1}_0 \phi \cdot \alpha, \phi(v)).
\]

\begin{proposition}
 $(P^k,\theta^k)$ is a Cartan bundle for the Pfaffian group $(G^k, \g^{k, k-1}, \g^{k-1} \times \RR^n)$.
\end{proposition}
\begin{proof}
It is immediate to check that $\theta^k$ is $G^k$-equivariant, pointwise surjective, and its kernel coincides with $\ker(d\pi^{k,k-1})$, which is involutive and inside $\ker(d\pi^k)$. Then $(P^k,\theta^k)$ is a Cartan bundle by Corollary \ref{cor_Cartan_bundles_via_forms}.

To see that $\h = \g^k_{k-1}$, we use the characterisation of part (i) of Proposition \ref{properties_ideal_k}. Since
\[
 d_p\pi^k_{k-1}(\alpha^\dagger_p) = (d\tau^k_{k-1}(\alpha))^\dagger_p
\]
we unravel the definition of $\theta^k$ from subsection \ref{subsection_higher_order_G_structures} and see that, for $\alpha \in \g^k$,
\[
 \alpha \in \h\iff 0 = \theta^k_p (\alpha^\dagger_p) = \phi_p ( d_p\pi^k_{k-1}(\alpha^\dagger_p) ) \iff  d_p\pi^k_{k-1}(\alpha^\dagger_p) = 0 \iff d\tau^k_{k-1}(\alpha) = 0. \qedhere
\] 
\end{proof}

\subsubsection{Fibres of transitive Pfaffian groupoids}

If $(\G, \omega)$ is a transitive full Pfaffian groupoid, the $t$-fibre $t^{-1}(x)$ at any point is a Cartan bundle when equipped with $\theta_g:=\omega_g|_{t^{-1}(x)}$. Alternatively, one can look at the $s$-fibres with the form $\theta_g:=g^{-1}\cdot \omega_g|_{s^{-1}(x)}$. This can be used to prove Proposition~\ref{prop_transitive_Pfaffian_Morita_equivalence}.
.

Notice that, if $\mathfrak{g} (\omega)_x = 0$, where $x\in \BB$, then the $s$-fibre of $(\Sigma,\omega)\tto \BB$ is a parallelisable Cartan bundle; on the other hand, if $\mathfrak{g}(\omega)_x$ coincide with the isotropy Lie algebra of $\Sigma$, then  $s$-fibre of $(\Sigma,\omega)\to \BB$ is a simple Cartan bundle. More generally, there is a bijective correspondence between transitive full Pfaffian groupoids and Cartan bundles, where the symbol $\g(\omega)$ corresponds to the distribution $\ker(\theta)$. This line of reasoning provides a different way to discover Cartan bundles -- a way that was explored in~\cite{FRANCESCOPAP}.

\subsection{The geometry of Cartan bundles}\label{section_geometry_of_Cartan_bundles}

Given any Cartan bundle $(P,\theta, V)$, we consider the following vector subspace:
\[
W := \theta(\ker(d\pi)) \subset V.
\] 

%


\begin{proposition}\label{properties_W}
 Let $(P,\theta)$ be a Cartan bundle, with $\theta \in \Omega^1 (P,V)$. Then
\begin{enumerate}
 \item 
 $W$ is well defined, i.e.\ $\theta_p(\ker(d_p\pi)) = \theta_q(\ker(d_q \pi))$ for every $p,q \in P$;
 \item $W$ and $V/W$ inherit a $G$-representation from $V$;
\item there is an isomorphism of representations $W \cong \g/\h$ (where $\h$ is the subrepresentation from Part (iii) of Proposition \ref{properties_ideal_k});
\item the following dimensional relations hold:
 \[
\dim (M) \leq \dim (V) \leq \dim(P), \quad \quad \dim(V/W) = \dim(M), \quad \quad \dim(W) \leq \dim(G);
\]
\item the associated vector bundle $P[V/W] \to M$ is canonically isomorphic to $TM \to M$.
\end{enumerate}
\end{proposition}


\begin{proof}
Part (i) follows from the definition of Cartan bundle, which prescribes $\theta_p(\alpha^\dagger_p) = \theta_q (\alpha^\dagger_q)$ for any $\alpha \in \g$, and the fact that $\g^\dagger = \ker(d\pi)$.

For part (ii) we consider any $g \in G$, $\alpha \in W$, and check that
$$ g \cdot \alpha = g \cdot \theta_p(v) = \theta_{p \cdot g^{-1}} (d R_{g^{-1}} (v)) \in W,$$
where we used the fact that $\theta$ is $G$-equivariant, and $d R_{g^{-1}} (v) \in \ker(d\pi)$ if $v \in \ker(d\pi)$. Then the quotient vector space $V/W$ is a $G$-representation, with $g \cdot [\alpha] := [g \cdot \alpha]$.

In part (iii) we consider the map
 \[
  \Phi: W \to \g/\h, \quad \theta_p(v) \mapsto [\alpha]_\h,
 \]
where $\alpha \in \g$ is such that $v = \alpha^\dagger_p$ (recall that, for a Cartan bundle, $\theta_p(a_p(\alpha)) = \theta_q(a_q(\alpha))$ so the choice of $p \in P$ is irrelevant). First of all, $\Phi$ is well-defined: if $\theta_p(v) = \theta_p (w)$ and $w = \beta^\dagger_p$, then $v - w = (\alpha - \beta)^\dagger_p \in \ker(\theta_p)$. This implies $\alpha - \beta \in \h$, i.e.\ $[\alpha]_\h=[\beta]_\h$.

Moreover, $\Phi$ is a linear isomorphism: it is surjective since $\g^\dagger_p = \ker(d_p\pi)$ and injective because $\alpha \in \h$ implies $\alpha^\dagger_p \in \ker(\theta_p)$. Last, it is an isomorphism of $G$-representations:
\[
 g \cdot \Phi (\theta_p(v)) = g \cdot \Phi(\theta_p ( \alpha^\dagger_p )) = g \cdot [\alpha]_\h= [g \cdot \alpha]_\h= \Phi (\theta_p (g \cdot \alpha)^\dagger_p ) = \Phi (g \cdot \theta_p (\alpha^\dagger_p)) = \Phi (g \cdot \theta_p(v)).
\] 

For part (iv) we use the rank-nullity theorem, for any $p \in P$:
$$ \dim(\ker(\theta_p)) + \dim(\Ima(\theta_p)) = \dim(T_p P).$$
\begin{itemize}
 \item The upper bound of $\dim(V)$ follows from the surjectivity of $\theta_p$, while the lower bound follows from the condition $\dim(\ker(\theta_p)) \leq \dim(\ker(d_p \pi)) = \dim (P) - \dim(M)$.
 \item The second claim follows from the definition of $W$:
 \[
\dim(W)=\dim(\ker(d_p\pi)) - \dim(\ker(\theta_p)) = \dim(\ker(d_p\pi)) - \dim(P) + \dim(\Ima(\theta_p)) = \dim(M) + \dim(V).  
 \]
\item The third claim follows by combining the two previous results with $\dim(P) = \dim(M) + \dim(G)$.
\end{itemize}

Last, the isomorphism in part (v) is given by the fibrewise linear map
 \[
 \Phi: P[V/W] \to TM, \quad [p,[\alpha_p]_W ] \mapsto \Phi_p (\alpha_p) := d_p \pi (v_p)
 \]
for any $v_p \in T_p P$ such that $\alpha_p = \theta_p (v_p) \in \Ima(\theta_p) = V$. Let us check that $\Phi$ is well defined:
\begin{itemize}
 \item if $\theta_p (v_p) = \theta_p (v'_p)$, then $v_p - v'_p \in \ker(\theta_p) \subset \ker(d_p\pi)$, hence $d_p \pi (v_p) = d_p \pi (v'_p)$;
 \item if $[\alpha_p]_W = [\alpha'_p]_W$, then $\alpha_p - \alpha'_p \in W = \theta (\ker(d\pi))$, hence $\Phi_p (\alpha_p - \alpha'_p) = 0$;
 \item if $[p, [\alpha_p]_W] = [p \cdot g, [g^{-1} \cdot \alpha_p]_W ]$, by the $G$-equivariance of $\theta$ we obtain
 \[
g^{-1} \cdot \alpha_p = g^{-1} \cdot \theta_p (v_p) = \theta_{p \cdot g} (dR_g (v_p)),
 \]
 and by the $G$-invariance of $\pi$ we conclude that
 \[
\Phi_{p \cdot g} (g^{-1} \cdot \alpha_p) = d_{p \cdot g} \pi (dR_g (v_p)) = d_p \pi (v_p) =\Phi_p (\alpha_p).  
 \]
\end{itemize} 
 We conclude that $\Phi$ is a surjective (since $\pi$ is a submersion) and injective (by definition of $W$) vector bundle morphism, hence it is an isomorphism.
 \end{proof}

\begin{example}\label{examples_W}
 Part (iii) of Proposition \ref{properties_W} shows that the smaller $\h$ is, the bigger $W$ becomes. More precisely, using part (i) of Proposition \ref{properties_W}, we get $0 \leq \dim(W) \leq \dim(P) - \dim(M) = \dim(G)$, so that we recover the two extreme cases of Cartan bundles.
 \begin{itemize}
  \item The case $\h=\g$ corresponds to $\dim(W) = 0$, i.e.\ $W =0$ ($G$-structures, section \ref{subsection_G_structures}). Indeed, since every fundamental vector field $\alpha^\dagger$ is vertical, it belongs automatically to $\ker(\theta)=\ker(d\pi)$ and the underlying foliation coincides with the vertical one.
  \item The case $\h=0$ corresponds to $\dim(W) = \dim(G)$, i.e.\ $W \cong \g$ (Cartan geometries, section \ref{subsection_Cartan_geometries}). Indeed, by definition of Cartan connection, $\theta(\alpha^\dagger)=\alpha \neq 0$ and the underlying foliation is the trivial foliation by points.
 \end{itemize}
 
 On the other hand, higher order $G^k$-structures $(P^k,\theta^k)$ interpolate between these two examples: their $W$ is isomorphic to $\g^{k-1} = \Lie(G^{k-1})$. Indeed, the image of $P^k$ by the jet projection $\pi^k_1: \Fr^k(M) \to \Fr(M)$, defines a classical first-order $G^1$-structure $\pi: (P^1,\theta^1, \RR^n) \to M$, therefore
\[
 v \in \ker(d\pi^k) \then d\pi^k_1(v) \in \ker(d\pi) = \ker(\theta^1).
\]
Since the last component of $\theta^k_p(v) \in \g^{k-1} \times \RR^n$ coincides with $\theta^1_{\pi^k_1(p)}(d_p \pi^k_1(v))$ (Remark \ref{first_component_theta_2}), we conclude that
\[
 W = \theta^k (\ker(d\pi^k)) = \g^{k-1} \oplus 0.
\]
 As alternative proof, one can also use $\h = \g^k_{k-1}:= \ker(d\tau^k_{k-1}) \subset \g^k$ (see section \ref{subsection_higher_order_G_structures}) and Part (iii) of Proposition \ref{properties_W} to conclude that
\[
W \cong \g^k/\h= (\g^k_{k-1} \oplus \g^{k-1})/\g^k_{k-1} \cong \g^{k-1}. \qedhere
\]
\end{example}

Let us picture the spaces we introduced in the following commutative diagram of vector bundles (over $P$ in the front face and over $M$ in the back face):

\[
 \begin{tikzcd}
                                                                                     &                                                                                            &                                                          &  & {\color{teal}\ker(\theta)/G} \arrow[r, color=teal, hook] \arrow[d, Rightarrow, no head] & {\color{teal}\ker(d\pi)/G} \arrow[r, color=teal, two heads] \arrow[d, hook]        & {\color{teal}P[W]} \arrow[d, hook]      \\
{\color{teal}\ker(\theta)} \arrow[d, Rightarrow, no head] \arrow[r, color=teal, hook] \arrow[rrrru, color=teal, two heads] & {\color{teal}\ker(d\pi)} \arrow[d, hook] \arrow[r, "\theta"', color=teal, two heads] \arrow[rrrru, color=teal, two heads]        & {\color{teal}P \times W} \arrow[d, hook] \arrow[rrrru, color=teal, two heads]      &  & {\color{purple}\ker(\theta)/G} \arrow[r, color=purple, hook] \arrow[d, two heads]           & {\color{purple}TP/G} \arrow[r, color=purple, two heads] \arrow[d, Rightarrow, no head] & {\color{purple}P[V]} \arrow[d, two heads] \\
{\color{purple}\ker(\theta)} \arrow[d, hook] \arrow[r, color=purple, hook] \arrow[rrrru, color=purple, two heads]                & {\color{purple}TP} \arrow[d, Rightarrow, no head] \arrow[r, "\theta"', color=purple, two heads] \arrow[rrrru, color=purple, two heads] & {\color{purple}P \times V} \arrow[d, two heads] \arrow[rrrru, color=purple, two heads] &  & {\color{blue}\ker(d\pi)/G} \arrow[r, color=blue, hook]                                  & {\color{blue}TP/G} \arrow[r, color=blue, two heads]                                & {\color{blue}P[V/W] \cong TM}           \\
{\color{blue}\ker(d\pi)} \arrow[r, color=blue, hook] \arrow[rrrru, color=blue, two heads]                                  & {\color{blue}TP} \arrow[r, "\pr_{V/W} \circ \theta"', color=blue, two heads] \arrow[rrrru, color=blue, two heads]                & {\color{blue}P \times V/W} \arrow[rrrru, color=blue, two heads]                    &  &                                                               &                                                          &                            
\end{tikzcd}
\]

The diagram simplifies in the usual examples (which correspond to the two extreme cases of $W$):
\begin{itemize}
 \item for an abstract $G$-structure ($W=0$) the first row become trivial, while the second and third rows coincide. Here part (iv) of Proposition \ref{properties_W} recovers the standard isomorphism $P[\RR^n] \cong TM$ given by $[p,\theta_p(v)] \mapsto d_p\pi(v)$, which in the ``concrete'' case $P \subset \Fr(M)$ boils down to $[p,\alpha] \mapsto p(\alpha)$.
\item for a Cartan geometry ($W=\g$) the first and second rows become isomorphisms, while part (iv) of Proposition \ref{properties_W} recovers the standard result $P[V/\g] \cong TM$ (see e.g.\ Theorem 3.15 of \cite[Chapter 5]{SHARPE}.
\end{itemize}

\begin{remark}[Lie algebroid brackets on $P\lbrack V \rbrack$ and $P \lbrack W \rbrack$]
 Recall that, for a Cartan geometry $(P,\theta)$, the associated vector bundle $P[V]$ is the {\bf adjoint tractor bundle}. After fixing a Lie algebra structure on $V=\z$ (see Remark \ref{remark_bracket_Cartan_geometry}), one obtains two different structures of Lie algebroids on $P[\z]$ \cite[Proposition 1.5.7]{CAPSLOVAK}, whose interplay is fundamental for the class of {\it regular} parabolic Cartan geometries (see section 3.1.7 of \cite{CAPSLOVAK} for further information):
\begin{itemize}
 \item since $\z$ is automatically compatible with the adjoint representation of $G$, it induces on $P[\z]$ a structure of {\it Lie algebra bundle}, whose Lie bracket is defined pointwise and therefore does not depend on the Cartan connection $\theta$, but only on the choice of the model geometry $\z$;
 \item since $\theta$ induces a vector bundle isomorphism $P[\z] \to TP/G$, it also induces on $P[\z]$ the structure of {\it Atiyah algebroid structure} from $TP/G$.
\end{itemize}
 Notice also that the isomorphism $P[\z] \cong TP/G$ restricts to an isomorphism $P[\g] \cong T^\pi P/G$; however, in this case $T^\pi P/G$ is the isotropy Lie algebra bundle of $TP/G$, and its pointwise Lie bracket coincides with that of $P[\g]$.

 Similarly, for an arbitrary Cartan bundle, one can consider the associated vector bundle $P[V]$, which becomes a Lie algebra bundle after the choice of a Lie bracket on $V$ which is compatible with the $G$-representation (as it was $V=\z$ for Cartan geometries). However, despite $\theta$ induces an isomorphism of vector bundles 
 \[
  P[V] \to (TP/\h^\dagger)/G,
 \]
there is no natural Lie algebroid bracket on the right-hand side to transport to $P[V]$: the reason is that $\Gamma(\h^\dagger)$ is a Lie ideal in $\Gamma(T^\pi P)$ but not in $\mathfrak{X}(P)$. As a consequence, the theory of adjoint tractor bundles for Cartan geometries cannot be generalised to arbitrary Cartan bundles.

On the other hand, the isomorphism above restricts to
\[
  P[W] \cong (T^\pi P/\h^\dagger)/G,
 \]
and in this case the pointwise Lie bracket induced by $W \cong \g/\h$ coincides with the one of the natural Lie algebroid structure (with trivial anchor) on the right-hand side. The chain of isomorphisms between the quotients bundles
\[
 P[V/W] \cong \frac{P[V]}{P [W]} \cong \frac{(TP/\h^\dagger)/G}{(T^\pi P/\h^\dagger)/G} \cong (TP/T^\pi P)/G \cong TM
\]
recovers therefore part (iv) of Proposition \ref{properties_W}.
\end{remark}

\subsubsection{First-order Cartan bundles}


\begin{definition}\label{definition_first_order_cartan_bundle}
Given a Cartan $G$-bundle $(P,\theta,V)$, denote by $K \subset G$ the kernel of the representation $G \to \GL(V/W)$ (part (ii) of Proposition \ref{properties_W}). Then $P$ is called of \textbf{first order} if $K = 0$.
\end{definition}

\begin{example}
It is clear that
\begin{itemize}
\item a $G$-structure ($V=\RR^n, W=0$) is automatically of first order, since $G \to \GL(V) = \GL(n,\RR)$ is an inclusion; \footnote{It is possible to consider a more general notion of $G$-structure, where $G$ is not necessarily a group of matrices and therefore $G \to \GL(n,\RR)$ is not necessarily injective. This can be useful to handle objects such as spin structures, but we will not pursue such generality in this paper}
\item for a Cartan geometry ($W=\g$) Definition \ref{definition_first_order_cartan_bundle} boils down to the standard definition of first order (see e.g.\ \cite[Chapter 5, Definition 3.20]{SHARPE}).
\end{itemize}
In the next section we will discuss examples of Cartan bundles which are not of first order.
\end{example}

In order to understand what lies behind this definition, notice that, since $K\subset G$ is a normal Lie subgroup, the $G$-action $m: P \times G \to P$ descends to an action of the Lie group $G_1 := G/K$ on the manifold $P_1:= P/K$:
\[
 \overline{m}: P_1 \times G_1 \to P_1, \quad ([p],[g]) \mapsto [pg].
\]
Moreover, since $\pi: P \times M$ is $G$-invariant, the projection $\pi_1: P_1 \to M, [p] \mapsto \pi(p)$ defines a principal $G_1$-bundle. Last, the $G$-representation on $V_1:=V/W$, denoted by $g \cdot [\alpha]$, descends to a $G_1$-representation on $V_1$ since $K$ is its kernel:
\[
 [g]_{G_1} \cdot [\alpha]_{V_1} := g \cdot [\alpha]_{V_1}.
\]

\begin{proposition}\label{underlying_Cartan_bundle}
 Any Cartan bundle $(P,\theta)$ has two underlying structures
 \begin{itemize}
  \item a Cartan $G_1$-bundle $(P_1,\theta_1, V_1)$;
 \item an abstract $G$-structure $(P,\overline{\theta})$.
 \end{itemize}
\end{proposition}

The situation is represented in the following diagram:
\[
 \begin{tikzcd}
TP \arrow[rr, "\theta", two heads] \arrow[dd, "\pr", two heads] \arrow[rrdd, "\overline{\theta}", two heads] &  & V \arrow[dd, "\pr", two heads] \\
                                                                                                        &  &                                \\
TP_1 \arrow[rr, "\theta_1", two heads]                                                                  &  & V_1:=V/W                           
\end{tikzcd}
\]

As one sees from the diagram, in the first order case ($K=0$), one has $P_1 = P$ and $\overline{\theta}$ coincides with $\theta_1$. In particular:
\begin{itemize}
 \item for a first-order Cartan geometry one recovers its underlying $G$-structure $(P,\overline{\theta})$, without reducing the principal bundle.
 \item For a (first-order) $G$-structure the reduction procedure becomes trivial since $W = 0$, therefore also $\overline{\theta}$ coincides with $\theta$. However, this will not the case for ``intermediate'' Cartan bundles (see the discussion for higher-order $G$-structures in Examples \ref{example_second_order_Cartan_bundles} and \ref{example_higher_order_Cartan_bundles}).
\end{itemize}

\begin{proof}
The form $\theta \in \Omega^1 (P,V)$ descends to a form $\theta_1 \in \Omega^1 (P_1,V_1)$, by setting
 \[
  (\theta_1)_{[p]} ([v]) := [ \theta_p (v) ]_{V_1}.
 \]
 Such form is well-defined since $\theta$ is $G$-equivariant and the elements of $K$ act on $V_1$ as the identity. Indeed, if $[p] = [p']$ and $[v_p] = [v'_{p'}]$, then
 \[
 \theta_{p'}(v'_{p'}) = \theta_{pg} (d_p R_g (v_p)) = g^{-1} \cdot \theta_p (v_p),
 \]
 but $[g^{-1} \cdot \alpha]_{V_1} = g^{-1} \cdot [\alpha]_{V_1} = [\alpha]_{V_1}$ since $g \in K$.
 
 Let us prove that $(P_1, \theta_1)$ is a Cartan bundle on $M$. We see immediately that $\theta_1$ is surjective since $\theta$ is so, and is $G_1$-equivariant since $\theta$ is $G$-equivariant. 
 
 Moreover, $\ker(\theta_1)$ contains the equivalence classes $[v] \in TP_1=TP/K$ such that $\theta(v) \in W$ for any representative $v \in TP$. This means that $\theta (v) = \theta (v')$ for some $v' \in \ker(d\pi) \subset TP$; but $v-v' \in \ker(\theta) \subset \ker(d\pi)$ as well, hence also the vector $v$ must be in the vertical bundle. We conclude that any representative of $[v] \in \ker(\theta_1)$ is in $\ker(d\pi)$, therefore $\ker(\theta_1)$ sits inside $\ker(d\pi_1)$.
 
 Last, for every $\alpha \in \g_1 = \g/\k$ and $[p],[p'] \in P_1$ we have
 \[
  (\theta_1)_{[p]} ([\alpha]^\dagger_{[p]}) = (\theta_1)_{[p]} ([\alpha^\dagger_p]) = [\theta_p (\alpha^\dagger_p)]_{V_1} = [\theta_{p'}(\alpha^\dagger_{p'})]_{V_1} = (\theta_1)_{[p']} ([\alpha]^\dagger_{[p']}),
 \]
where the second step follows from differentiating the equality $\pr_{P_1} \circ m(p,\cdot) = \overline{m}([p],\cdot) \circ \pr_{G_1}$.

 For the second claim we notice that $V_1=V/W$ has dimension equal to $\dim(M)$ by part (iii) of Proposition \ref{properties_W}, so we consider the composition of $\theta$ with the projection $V \to V_1$, i.e.\ the form $\overline{\theta} \in \Omega^1 (P,V_1)$ defined by
 $$ \overline{\theta}_p (v):= [\theta_p(v)]_{V_1}.$$
It is immediate to check that $\overline{\theta}$ is pointwise surjective and $G$-equivariant, since $\theta$ is so. Moreover, $\ker(\overline{\theta}) \subset \ker(d\pi)$ since
\[
 v \in \ker(\overline{\theta}) \then \theta (v) \in W \then \theta(v)=\theta(v') \text{ with } v' \in \ker(d\pi) \then v - v' \in \ker(\theta_p) \subset \ker(d\pi) \then v \in \ker(d\pi).
\]
On the other hand, if $v \in \ker(d\pi)$, $\theta(v) \in W$ by definition, hence $\overline{\theta}(v)=0$; we conclude that $\ker(\overline{\theta}) =\ker(d\pi)$, so $(P,\overline{\theta})$ is an abstract $G$-structure on $M$.
\end{proof}

\subsubsection{Higher-order Cartan bundles}\label{section_higher_order_Cartan_bundles}

 Given a Cartan $G$-bundle $(P,\theta,V)$ not of first order (i.e.\ $K \neq 0$), one can apply Proposition \ref{underlying_Cartan_bundle} to the reduced Cartan $G_1$-bundle $(P_1,\theta_1,V_1)$, using the $G_1$-representation $V_1$, the subrepresentation $W_1 \subset V_1$ and the kernel $K_2$ of the representation $G_1 \to \GL(V_1/W_1)$.

 \begin{definition}
A Cartan bundle $(P,\theta,V)$ is of {\bf second order} if the reduced Cartan bundle $(P_1,\theta_1,V_1)$ is of first order; equivalently, $K_2 = 0$ or $P_2 = P_1$.
 \end{definition}

 \begin{example}[second-order $G$-structures]\label{example_second_order_Cartan_bundles}
Let $P^2 \subset \Fr^2 (M)$ be a classical second-order $G^2$-structure. As discussed in section \ref{subsection_higher_order_G_structures}, we can interpret it as a Cartan bundle $\pi^2: (P^2,\theta^2,V^2) \to M$, with $V^2 = \g^1 \times \RR^n$ and $W^2 \cong \g^1$ (Example \ref{examples_W}). We claim that $P^2$ is second-order also as Cartan bundle, and its reduction (the Cartan $(G^2)_1=G^2/K$-bundle $(P^2)_1 = P^2/K$) coincides with the underlying first-order structure (the $G^1$-structure $P^1$).

Indeed, the first component of the $G^2$-representation on $V^2 = \g^1 \times \RR^n$ coincides with the $G^1$-representation on $\RR^n$; then the quotient representation $G^2 \to \GL(V^2/W^2) \cong \GL(\RR^n)$ boils down to $j^2_0 f \cdot v = d_0f(v)$ and its kernel $K$ is isomorphic to $G^2/G^1$. It follows immediately that $G^2/K \cong G^1$, as we wanted.

On the other hand, the map
\[
\Phi: P^2/K \to \pi^2_1(P^2), \quad [v] \mapsto \pi^2_1(v) 
\]
is well defined since $\pi^2_1$ is invariant under the action of $K$ and is surjective since $\pi^2_1$ is surjective. It is also injective because any two elements $v, w \in P^2$ s.t.\ $\pi^2_1(v)=\pi^2_1(w)$ are in the same fibre of $P^2 \to M$, therefore are related by an element $g \in G^2$; but since their image in $P^1$ coincide, the component of $g$ in $G^1$ (which would relate $\pi^2_1 (v)$ and $\pi^2_1 (w)$) is killed, so it must be $g \in K$, i.e.\ $[v] = [w]$.

Since $\Phi$ preserves the $G^1$-action, it is an isomorphism of principal bundles, undere which the form $(\theta^2)_1$ is precisely the tautological form $\theta^1$ of $P^1$; therefore we have an isomorphism of Cartan bundles (actually, of $G^1$-structures).
\end{example}

More generally, if we set $P_0=P$, $G_0=G$, $V_0=V$, $W_0=W$ and $K_0=K$, we can apply Proposition \ref{underlying_Cartan_bundle} recursively and obtain a sequence of surjective submersive morphisms between Cartan bundles over $M$:
\[
\begin{tikzcd}
P_0=P \arrow["G_0=G"', loop, distance=2em, in=125, out=55] \arrow[r, two heads] \arrow[rrd, two heads] & P_1=P_0/K_0 \arrow[r, two heads] \arrow["G_1=G_0/K_0"', loop, distance=2em, in=125, out=55] \arrow[rd, two heads] & P_2=P_1/K_1 \arrow[r, two heads] \arrow["G_2=G_1/K_1"', loop, distance=2em, in=125, out=55] \arrow[d, two heads] & \ldots \arrow[r, two heads] & P_{i+1}=P_i/K_i \arrow["G_{i+1}=G_i/K_i"', loop, distance=2em, in=125, out=55] \arrow[lld, two heads] \\
                                                                                                       &                                                                                                                   & M                                                                                                                &                             &                                                                                                              
\end{tikzcd}
\]
 
 \begin{definition}
A Cartan bundle $(P,\theta,V)$ is of {\bf order $i+1$} if $(P_i, \theta_i, V_i)$ is of order $1$; equivalently,
if $(P_1, \theta_1, V_1)$ is of order $i$.
 \end{definition}
 
 \begin{example}[Higher order G-structures]\label{example_higher_order_Cartan_bundles}
With the same arguments of Example \ref{example_second_order_Cartan_bundles}, one sees that a classical $G^k$-structure $P^k \subset \Fr^k(M)$ of order $k$, together with its tautological form $\theta^k \in \Omega^1 (P^k, \g^{k-1} \times \RR^n)$ (see section \ref{subsection_higher_order_G_structures}), is a Cartan bundle of the same order. Since here $W^k$ is $\g^{k-1} = \Lie (G^{k-1})$ (Example \ref{examples_W}), we have $V^k/W^k \cong \RR^n$ and $K \cong G^k/G^{k-1}$, with $G^{k-1}$ image of $G^k$ via the jet projection $\tau^k_{k-1}: \GL^k (\RR^n) \to \GL^{k-1}(\RR^n)$.
 \end{example}

  
\begin{remark}
 Notice that, at each reduction step $i$, there are only two possibilities:
\begin{itemize}
\item $K_i$ is zero dimensional, i.e.\ $G_{i+1} = G_i$;
\item $K_i$ has positive dimension, i.e.\ the dimension of $G_{i+1}$ is strictly lower than the dimension of $G_i$.
\end{itemize} 
Consequently, since the dimension of $G$ is finite, after a finite number of steps (say at some $i_0 \in \NN$), only one of the two following option holds:
\begin{itemize}
\item the reduction sequence stabilises to a Cartan bundle $(P_{i_0},\theta_{i_0},V_{i_0})$ of order 1, with $G_{i_0}$ of positive dimension;
\item the reduction sequence stabilises to the ``trivial'' Cartan bundle $P_{i_0} = M \to M$, where $G_{i_0}$ is zero dimensional. \qedhere
\end{itemize}
\end{remark}
 
%

\subsection{Integrability of Cartan bundles}\label{section_integrability_Cartan_bundles}

In this section we will work with a Pfaffian group $(G, \omega, V)$, and we will always denote it by $(G, \h, V)$, for $\h=\ker(\omega_e)$, as explained after Proposition \ref{prop_defining_h}. Recall that we need to allow Pfaffian groups to be not necessarily full; on the other hand, by Axiom \ref{axiom_full}, all Cartan bundles $(P,\theta)$ will be full, i.e.\ $\theta$ is pointwise surjective.

\begin{definition}\label{def_lift}
Let $(G, \h, V)$ be a Pfaffian group. A {\bf coefficient extension} $Z$ for $(G, \h, V)$ consists of a vector space $Z$ such that there exists a short exact sequence of representations of $G$:
\[
 \begin{tikzcd}
0 \arrow[r] & \h \arrow[r, "i", hook] & Z \arrow[r, "p", two heads] & V \arrow[r] & 0.
\end{tikzcd} 
\]
Let now $(P, \theta)$ be a Cartan $(G, \h, V)$-bundle. A {\bf lift} $(Z,\eta)$ of $\theta$ consists of
\begin{enumerate}
 \item a coefficient extension $Z$ for $(G, \h, V)$ whose dimension is equal to the dimension of $P$;
\item a $Z$-valued 1-form $\eta\in \Omega^1(P, Z)$ on $P$ such that 
\begin{itemize}
\item $\eta$ is $G$-equivariant;
\item $\eta$ is compatible with the infinitesimal action, i.e.\
\[
\eta(\alpha^\dagger) = i(\alpha),\quad \alpha\in \h;
\]
\item $\eta$ is a lift of $\theta$ to $Z$ in the sense that
\[
p\circ \eta = \theta.
\]
\end{itemize}
\end{enumerate}
We will also call a vector space $Z$ as in (i) a {\bf coefficients extension adapted to $(P, \theta)$}, and a form $\eta$ as in (ii) an {\bf $Z$-lift} of $\theta$.
\end{definition}

%

\begin{remark}[lifts and Cartan connections]\label{remark_relation_Cartan_connection}
As we already stressed repeatedly, when $\h = 0$, we are dealing with a Cartan connection $\theta$ (Definition~\ref{def_Cartan_geometry}). Then the notions of coefficient extension and of lift are trival ($Z = V$ and $\eta = \theta)$. 


On the other hand, if $\h = \g$ (and $\theta$ is the tautological form of a $G$-structure), then a lift $\eta$ is a Cartan connection with the additional compatibility condition $p \circ \eta = \theta$.

In general, a lift $(Z,\eta)$ is not strictly speaking a Cartan connection on $P$, since $Z$ may not contain the adjoint representation on $\g$. Nevertheless, $\eta$ defines an {\bf absolute parallelism}, i.e.\ it is a pointwise isomorphism, since $\dim(Z) = \dim(P)$ and $\ker(\eta)=0$. To check this last claim, assume $\eta(v) = 0$, so that $v \in \ker(\theta) \subset \ker(d\pi)$, and therefore $v = \alpha^\dagger$ for some $\alpha \in \h$; but $0 = \eta (v) = i(\alpha)$, which means $v = 0$.
\end{remark}

\begin{remark}
Since we assume $(P, \theta)$ to be full (Axiom \ref{axiom_full}), we have a short exact sequence
\[
 \begin{tikzcd}
0 \arrow[r] & P \times \h \arrow[r, hook] & TP \arrow[r, "\theta", two heads] & P \times V \arrow[r] & 0.
\end{tikzcd} 
\]
If $(Z,\eta)$ is a lift of $\theta$, we also have the short exact sequence of representations
\[
 \begin{tikzcd}
0 \arrow[r] & \h \arrow[r, "i", hook] & Z \arrow[r, "p", two heads] & V \arrow[r] & 0
\end{tikzcd} 
\]
The $Z$-lift $\eta\in \Omega^1(P,Z)$ induces then an isomorphism $TP \to P \times Z$ (see Remark \ref{remark_relation_Cartan_connection}) and therefore the commutative diagram
\[
\begin{tikzcd}
            &                                                        & P \times Z \arrow[rd, "p", two heads]          &                      &   \\
0 \arrow[r] & P \times \h \arrow[r, "i", hook] \arrow[ru, "i", hook] & TP \arrow[r, "p", two heads] \arrow[u, "\eta", "\cong"'] & P \times V \arrow[r] & 0
\end{tikzcd} \qedhere
\]
\end{remark}

In order to discover a suitable definition of ``integrability'' for Cartan bundles, the notion of coefficient extension need to be enriched with a suitable bracket (see Definition \ref{def_infinitesimal_model} below). One could naively think to impose a {\it Lie} bracket -- and in some examples this would be correct -- but in other important cases the natural candidate for a bracket is not Lie, forcing us to weaken such hypothesis.

More precisely, recall that an {\bf almost Lie bracket} on a vector space $V$ is a bilinear and skew-symmetric map $[ \cdot, \cdot ]: V \times V \to V$, which does not necessarily satisfy Jacobi identity. A pair $(V, [ \cdot,\cdot ])$ is called an {\bf almost Lie algebra}. We stress that, from now on, {\it we will use the standard notation with gothic letters for both Lie algebras and almost Lie algebras}. The main instance of almost Lie algebras in our story is discussed below.

\begin{remark}[Almost Lie algebras and semidirect products]\label{rk_semidirect_product}
Recall that, if a Lie algebra $\g$ acts linearly on another Lie algebra $\k$ by derivations, i.e.\ the representation $\g \to \gl(\k)$ restricts to $\g \to {\rm Der}(\k)$, then we can form the semidirect product $\g \ltimes \k$, i.e.\ the Cartesian product $\g \times \k$ together with the bracket
\[
 [(\alpha, v), (\beta, w)]_{\g \ltimes \k} := \left( [\alpha,\beta]_{\g}, [v,w]_{\k} + \alpha (w) - \beta (v) \right).
\]
However, if the action of $\g$ on $\k$ is not by derivations, the formula above still makes sense and defines a bilinear and skew-symmetric bracket on $\g \times \k$. Actually, this remain true even if $\g$, $\k$ or both of them are almost Lie algebras.

In the following we will still denote the product $\g \times \k$ with such an almost Lie bracket structure by $\g \ltimes \k$ and call it {\bf semidirect product}.
\end{remark}

\begin{definition}\label{def_infinitesimal_model}

Let $(G, \h, V)$ be a Pfaffian group. A coefficient extension $Z$ is called a {\bf Cartan-type extension} if there is an almost Lie algebra structure on $Z$, denoted by $\z$, such that the inclusion $i:\h\hookrightarrow \z$ of representations of $G$ is also an almost Lie algebra map.\footnote{For $\h=\g$ -- in the context of Cartan geometries -- a Cartan-type extension when $\z$ satisfies Jacobi recovers the notion of {\bf model geometry}, e.g.~\cite[Chapter 5, Definition 1.1]{SHARPE}.}.

If $(P, \theta)$ is a Cartan $(G, \h, V)$-bundle, then Cartan-type extension for $(P, \theta)$ is simply called {\bf Cartan-type extension adapted to $(P,\theta)$}.

\end{definition}

\begin{definition}\label{def_a_flatness}
Let $(P, \theta)$ be a Cartan $(G, \h, V)$-bundle and $\z$ be an Cartan-type extension for $(P,\theta)$. A Cartan bundle $(P, \theta, V)$ is called {\bf $\z$-flat} if $\theta$ admits a $\z$-lift $\eta$ which is flat, in the sense that its {\bf curvature 2-form}
\[
 \Omega^\eta_{\z}:= d \eta + \frac{1}{2} [\eta, \eta] \in \Omega^2 (P, \z)
\]
vanishes.
\end{definition}

We stress that, despite the fact that bracket on $\z$ is not necessarily a {\it Lie} bracket, the Maurer-Cartan-type equation defining $\Omega^\eta_\z$ still makes sense.

\subsubsection{Examples of flatnesses}

The notion of $\z$-flatness is of course strictly related to that of {\it integrability} of various objects in differential geometry, as shown in the following examples.

\begin{example}[Integrability of Cartan geometries]
When $(P,\theta)$ is a Cartan geometry (Definition \ref{def_Cartan_geometry}), so $\eta = \theta$ (Remark \ref{remark_relation_Cartan_connection}), $\z$-flatness is nothing but the usual notion of flatness for Cartan geometries, see e.g.~\cite[Chapter 5, Definition 3.1]{SHARPE} and \cite[Section 1.5.1]{CAPSLOVAK}\footnote{As we pointed out in Remark \ref{remark_bracket_Cartan_geometry}, it is customary to have a Lie bracket on $\z$, or rather a model geometry $(\g,\z)$, as part of the definition of Cartan geometry, rather than choosing it later. We notice however that the definition of flatness works also when $\z$ is an almost Lie algebra.}.
\end{example}

\begin{example}[Integrability of $G$-structures]\label{ex:int-GS}
Let $(P,\theta_{\rm taut})\to M$ be a classical $G$-structure, i.e.\ $P\subset \Fr(M)$, $G\subset \GL(n, \mathbb{R})$ and $\theta_{\rm taut}\in \Omega^1(P,\mathbb{R}^n)$ the tautological 1-form. 
In particular, $\ker(\theta_{\rm taut}) = \ker(d\pi)$ and $\h = \g$. The representation of $G\subset \GL(\mathbb{R}^n)$ on $\mathbb{R}^n$ is the canonical one. We denote $\mathbb{R}^n$ with the abelian Lie bracket by $\mathfrak{ab}$.



The Lie algebra $\g \subset \gl(n)$ acts on $\mathfrak{ab}$ by derivations, hence we have the short exact sequence of Lie algebras (in fact, a split extension)
\[
 \begin{tikzcd}
0 \arrow[r] & \mathfrak{ab} \arrow[r, hook] & \g \ltimes \mathfrak{ab} \arrow[r, two heads] & \g \arrow[r] & 0,
\end{tikzcd} 
\]
where $\g \ltimes \mathfrak{ab}$ is the semidirect product, i.e.\ it is endowed with the Lie bracket
\[
[(\alpha, v), (\beta, w)]_{\g \ltimes \mathfrak{ab}} := \left( [\alpha, \beta]_\g, \alpha(w)- \beta(v) \right),\quad \quad (\alpha, v), (\beta, w) \in \g \ltimes \mathfrak{ab}.
\]

The Cartan bundle $(P,\theta_{\rm taut})$ is $\g \ltimes \mathfrak{ab}$-flat if and only if $P$ is integrable as a classical $G$-structure. To prove such claim, first notice that any principal connection $\tau\in \Omega^1(P,\g)$ provides a $\g \ltimes \mathfrak{ab}$-lift of $\theta_{\rm taut}$ by setting 
\[
\eta := \tau + \theta_{\rm taut} \in \Omega^1(P,\g \ltimes \mathfrak{ab}).
\]

Consider the curvature 2-form 
\[
\Omega^\eta_{\g \ltimes \mathfrak{ab}}  := d\eta+\frac{1}{2}[\eta, \eta]_{\g \ltimes \mathfrak{ab}}\in \Omega^2(P, \g \ltimes \mathfrak{ab}).
\] 
Since $\eta = \tau + \theta_{\rm taut}$, we can project the curvature onto its components in $\g$ and $\mathfrak{ab}$, obtaining
\[
\Omega^\eta_\g := d\tau+\frac{1}{2}[\tau, \tau]_{\g}\in \Omega^2(P, \g)
\] 
and
\[
\Omega^\eta_{\mathfrak{ab}} := d\theta_{\rm taut}+ \tau \wedge \theta_{\rm taut}\in \Omega^2(P, \mathfrak{ab}).
\] 

If the lift $\eta$ is flat, the Maurer-Cartan equation
\[
\Omega^\eta_{\g \ltimes \mathfrak{ab}} =0
\]
holds. The $\g$- and the $\mathfrak{ab}$-components need to vanish separately, hence we get
\[
\Omega^\eta_\g =0\quad \text{ (flatness of } \tau \text{)}
\]  
and
\[
\Omega^\eta_\mathfrak{ab} = 0.
\]
This implies integrability of the $G$-structure $(P,\theta_{\rm taut})$ in the classical sense -- i.e.\ there are local coordinates where the $G$-structure is isomorphic to its linear model, cf.\ Example~\ref{ex:G-struct}. In fact, pulling back the last equation by a local frame $\sigma:U\to P$ which is $\tau$-flat, we get
\[
d(\sigma^* \theta_{\rm taut}) = 0.
\]
Using the definition of $\theta_{\rm taut}$, the claim follows. Viceversa, if the $G$-structure is integrable, we observe that we can choose the connection $\tau\in \Omega^1(P,\g)$ to be flat\footnote{It is not hard to see that if $(P,\theta)$ is integrable, then $P$ admits a trivialising cocycle with constant transition functions.}, i.e.\
\[
\Omega^\eta_\g =0
\]  
holds. Moreover, since $\ker(\tau)$ and $\ker(\theta_{\rm taut})$ are transversal, and since integrability over an open $U$ corresponds to
\[
d(\sigma^* \theta_{\rm taut}) = 0
\]
for all frames $U$ which are $\tau$-flat, we also get
\[
\Omega^\eta_\mathfrak{ab} = 0,
\]
hence the flatness of $\eta:=\tau + \theta_{\rm taut} \in \Omega^1(P, \g \ltimes \mathfrak{ab})$.
\end{example}

\begin{example}[$\k$-integrability of $G$-structures]\label{ex:h-int-GS} 
With a simple generalisation of the previous example one can describe as $\z$-flat Cartan bundles, where $\z$ is suitably chosen, also $G$-structures that are integrable {\it with respect to a chosen (in general not abelian) Lie bracket on $\mathbb{R}^n$}. Below, we will use $\k$ to denote $\mathbb{R}^n$ with such a Lie bracket.

A a classical $G$-structure $(P,\theta_{\rm taut})$ on a manifold $M$ is {\bf $\k$-integrable} when, for all $x\in M$ there is an open $U_x\subset M$ and a local frame $(V_1, \dots V_n)$ defined over $U_x$ such that 
\begin{itemize}
\item $(V_1, \dots V_n):U_x\to \Fr(M)$ takes values in $P$;
\item $(\mathrm{span}\{V_1, \dots V_n\}, [\cdot ,\cdot ]_{\mathfrak{X}(M)})$, where $[\cdot ,\cdot ]_{\mathfrak{X}(M)}$ is the usual Lie bracket of vector fields, is isomorphic to $\k$.
\end{itemize}
As far as we are aware, the notion of $\k$-integrability for $G$-structures appeared first (and only) in~\cite[Example II.20]{MOLINOINTEGR}.

Unlike in the case $\k = \mathfrak{ab}$, here the action of the Lie algebra $\g \subset \gl(n)$ on $\k$ is not by derivations. Nevertheless, we can still consider the ``semidirect product'' discussed in Remark \ref{rk_semidirect_product}, i.e.\ the almost Lie algebra $\z := \g \ltimes \k$. As in Example \ref{ex:int-GS}, any principal connection $\tau \in \Omega^1(P,\g)$ induces the Cartan connection $\eta := \tau + \theta_{\rm taut} \in \Omega^1(P,\z)$. The curvature
\[
\Omega^\eta_\z := d\eta + \frac{1}{2}[\eta,\eta]_\z \in \Omega^2(P,\z)
\]
 of $\eta$ splits in the two components
\[
\Omega^\eta_\g := d\tau +\frac{1}{2}[\tau,\tau]_\g \in \Omega^2(P,\g)
\]
and
\[
\Omega^\eta_\k := d\theta_{\rm taut}+\frac{1}{2}[\theta_{\rm taut},\theta_{\rm taut}]_\k+\tau\wedge \theta_{\rm taut}\in \Omega^2(P, \k).
\]
The exact same argument as in the previous example (which did not require the Jacobi identity on $\z$) shows that $\k$-integrability is equivalent to $\z$-flatness.
\end{example}

\begin{example}[Integrability of (almost) contact structures]\label{ex_integrability_contact_structures}
A standard example of $G$-structures integrable with respect to a non-abelian Lie algebra is given by contact structures. Indeed, underlying a contact structure on $M^{2k+1}$ there is always an $\Sp(k, 1)$-structure -- sometimes called {\bf almost contact structure}\footnote{we warn the reader that this name is ambiguous and in the literature it is also used with other meanings.} -- where $\Sp(k,1)$ is the group of linear symmetries of the canonical symplectic foliation on $\mathbb{R}^{2k+1}$, see Example~\ref{ex:G-struct}. This $\Sp(k, 1)$-structure is never integrable in the usual sense (a flat $\Sp(k, 1)$-structure needs to be a symplectic foliation). However, thanks to the Darboux theorem, the $\Sp(k, 1)$-structures underlying contact structures are integrable with respect to the $(2k+1)$-dimensional Heisenberg Lie algebra $\mathfrak{hei}$. Indeed, around each point there is a local frame that behaves like $\mathfrak{hei}$ under the Lie bracket of vector fields.
\end{example}

\begin{example}[Integrability of higher order $G$-structures]
 The discussion for $G$-structures of first order from Example \ref{ex:int-GS} can be generalised also to higher order ones, the final outcome being that their integrability is equivalent to $\RR^n$-flatness as Cartan bundles.
 
 In order to understand this claim, we need to introduce the notion of reductivity as well as recall the standard language of structure equations, which are both discussed in the next sections.
\end{example}

\subsubsection{Reductive Cartan bundles}

The way in which we dealt with $\z$-flatness in Examples~\ref{ex:int-GS} and~\ref{ex:h-int-GS} is by means of a principal connection and a suitable split of the Cartan-type extension. In the context where $\h$ is possibly smaller than $\g$, this approach can be adapted by replacing connections with an appropriate generalisation (Definition \ref{def:h-conn}) and by introducing an important property (Definition \ref{def:reductive}) which was automatically true in our examples.

\begin{definition}\label{def:h-conn}
Let $(P, \theta)$ be a Cartan $(G, \h, V)$-bundle. A {\bf $\h$-connection} on $(P,\theta)$ is a 1-form $\tau\in \Omega^1(P,\h)$ such that
\begin{itemize}
\item $\tau(\alpha^\dagger) = \alpha$ for all $\alpha\in \h$;
\item $\tau$ is $G$-equivariant.
\end{itemize}
We say that $\tau$ is {\bf flat} if its curvature
\[
 \Omega^\tau:= d\tau + \frac{1}{2} [\tau,\tau]_\h \in \Omega^2 (P,\h)
\]
vanishes.
\end{definition}

Of course, if $\h = \g$ we recover the standard notions of connection and of flatness. For arbitrary Cartan bundles, as we prove below, the existence of lifts and $\h$-connections are strictly related to each other and to the splitting of the sequence of $G$-representations from the definition of coefficient extension. This fact will be repeatedly used in the next section to prove new results on $\k$-integrability.

\begin{theorem}\label{properties_lifts}
 Let $(P,\theta)$ be a Cartan bundle. 
 \begin{enumerate}
  \item Fix a lift $(Z,\eta)$ (Definition \ref{def_lift}). Then the sequence of $G$-representations 
 \[
 \begin{tikzcd}
0 \arrow[r] & \h \arrow[r, "i", hook] & Z \arrow[r, "p", two heads] & V \arrow[r] & 0
\end{tikzcd} 
\]
 splits if and only if $P$ admits a $\h$-connection $\tau \in \Omega^1 (P,\h)$ (Definition \ref{def:h-conn}). 
 \item Fix a coefficient extension $Z$ (Definition \ref{def_lift}) and a $\h$-connection $\tau$. Then the sequence of $G$-representations
  \[
 \begin{tikzcd}
0 \arrow[r] & \h \arrow[r, "i", hook] & Z \arrow[r, "p", two heads] & V \arrow[r] & 0
\end{tikzcd} 
\]
 splits if and only if $P$ admits a $Z$-lift $\eta \in \Omega^1 (P,Z)$.
 \item Fix a coefficient extension $Z$ and a right splitting $r$ of the sequence of $G$-representations
 \[
 \begin{tikzcd}
0 \arrow[r] & \h \arrow[r, "i", hook] & Z \arrow[r, "p", two heads] & V \arrow[r] & 0.
\end{tikzcd} 
\]
Then any $Z$-lift $\eta \in \Omega^1 (P,Z)$ must be of the form
\[
 \eta = i \circ \tau + r \circ \theta,
\]
for some $\h$-connection $\tau$.
  \end{enumerate}
\end{theorem}

In the case $\h = \g$ this theorem recovers the known results for $G$-structures $(P,\theta)$ and (reductive) Cartan geometries $(P,\eta)$, e.g.~\cite[Appendix A.2-A.3]{SHARPE} and \cite{LOTTA}.

\begin{proof}
 If the sequence admits a left splitting $l: Z \to \h$, one sets
 \[
  \tau:= l \circ \eta,
 \]
 which is by construction a $\h$-connection: indeed, $l$ preserves the $G$-representations and
 \[
\tau(\alpha^\dagger) = l (\eta(\alpha^\dagger)) = l (i(\alpha)) = \alpha.  
 \]
 Conversely, given $\tau$, one defines the following map
 \[
  r: V \to Z, \quad \quad  r(\theta_p (v) ) := \eta_p (v) - i (\tau_p (v) ).
 \]
where we use that $\theta_p$ is surjective for any $p \in P$ (Axiom \ref{axiom_full}).

 Notice that $r$ is well defined: indeed, if $v-w \in \ker(\theta)$, by the characterisation of $\h$ from Proposition \ref{properties_ideal_k} we have $v-w = \alpha^\dagger_p$, for some $\alpha \in \h$, hence, by definition of lift and of $\h$-connection,
\[
 r (\theta_p(v-w)) = \eta_p (v-w)  - i (\tau_p (v-w) ) = \eta_p (\alpha^\dagger_p) - i (\tau_p (\alpha_p^\dagger)) = i (\alpha) - i (\alpha) = 0.
\]
 Moreover, $r$ is a right splitting since $\eta$ is a lift:
 \[
  (p \circ r) (\theta_p (v) ) ) = p (\eta_p (v) ) - \cancel{p (i (\tau_p (v) ) )} = \theta_p(v).
 \]

The proof for (ii) works with the same arguments: given a right splitting $r$ one defines a lift
\[                                                           
\eta:=i\circ\tau+r\circ\theta,
\]
and given $\eta$ one defines a right splitting
\[
r: V \to Z, \quad \quad  r(\theta_p (v) ) := \eta_p (v) - i (\tau_p (v) ).
\]

Last, (iii) follows directly from the previous two points.
\end{proof}


We arrive finally to the main definition of this section.

\begin{definition}\label{def:reductive}
A Cartan-type extension $\z$ for a Pfaffian group $(G, \h, V)$ (Definition \ref{def_infinitesimal_model}) is called {\bf reductive} when the sequence of $G$-representations
\[
 \begin{tikzcd}
0 \arrow[r] & \h \arrow[r, "i", hook] & \z \arrow[r, "p", two heads] & V \arrow[r] & 0
\end{tikzcd} 
\]
admits a left splitting
\[
l:\z\to \h
\]
such that
\begin{itemize}
\item $l$ is a map of almost Lie algebras, i.e. $l ([\cdot, \cdot ]_\z ) = [l (\cdot), l(\cdot) ]_\h$;
\item the representation of $\h$ on $V$ is compatible with the almost Lie bracket on $\z$:
\[
 \alpha(v) = p ([i(\alpha),r(v)]_\z),
\]
where
\[
 r: V \to \z
\]
is the right splitting induced by $l:\z \to \h$.
\end{itemize}
A Cartan $(G, \h, V)$-bundle $(P, \theta)$ together with Cartan-type extension $\z$ (Definition \ref{def_infinitesimal_model}) which is reductive is also called {\bf $\z$-reductive}.
\end{definition}


We show now that that, fixing a coefficient extension, any almost Lie bracket on $V$ induces a Cartan-type extension $\h \ltimes \k$ which is automatically reductive; and conversely, any reductive Cartan-type extension arises (up to isomorphism) from this construction.


\begin{theorem}[Reductivity and semidirect products]\label{characterisation_reductive_extensions}
Let $(G, \h, V)$ be a Pfaffian group. 
\begin{enumerate}
\item Given the coefficient extension $Z =\h \times V$, with splittings $l = \pr_\h$ and $r = \pr_V$, and an almost Lie bracket on $V=\k$, denoted by $[\cdot, \cdot]_\k$, then the semidirect product (see Remark \ref{rk_semidirect_product})
\[
\z = \h \ltimes \k 
\]
is a reductive Cartan-type extension.
 \item Given a reductive Cartan-type extension $\z$ for $(G, \h, V)$, there is a unique almost Lie bracket on $V=\k$,
 \[
  [v,w]_\k := p ([r(v),r(w)]_\z),
 \]
 such that $r: \k \to \z$ is an almost Lie algebra morphism and
\[
\z \cong \h\ltimes \k, \quad \quad z \mapsto (l(z),p(z))
\]
is an isomorphism of almost Lie algebras (where $\h\ltimes \k$ is the semidirect product from Remark \ref{rk_semidirect_product}).
\end{enumerate}
\end{theorem} 

\begin{proof}
For (i) it is enough to check that the linear map $l$ is an almost Lie algebra morphism, i.e.\
\[
 l ( [z_1,z_2]_\z) = l \Big[ (l(z_1), p(z_1)), (l(z_2), p(z_2) ) \Big]_{\h \ltimes \k} = 
 \]
 \[
 = l \Big( [l(z_1), l(z_2)]_\h, [p(z_1),p(z_2)]_\k + l(z_1)(p(z_2)) - l(z_2)(p(z_1)) \Big) = [l(z_1), l(z_2)]_\h,
\]
and the representation of $\h$ on $V=\k$ is recovered by
\[
 p ([i(\alpha),r(v)]_\z ) = p \Big[ l(i(\alpha)), \cancel{p(i(\alpha))}, (\cancel{l(r(v))}, p(r(v)) ) \Big]_{\h \ltimes \k} = p [(\alpha,0), (0,v)]_{\h \ltimes \k} 
 = p (0, \alpha(v)) = \alpha(v).
\]

For (ii), we notice first that, since $l$ is a morphism of almost Lie algebras, its kernel $\ker(l) = \Ima(r) \subseteq \z$ is an almost Lie subalgebra. Then, due to the injectivity of $r$, the bracket on $V=\k$ is uniquely defined by
\[
 r ([v,w]_\k ) = [r(v), r(w)]_\z \quad \quad \forall r(v), r(w) \in \Ima(r),
\]
so that by construction $r$ preserves the bracket. Applying $p$ to both terms one obtains precisely
\[
[v,w]_\k = p ([r(v),r(w)]_\z). 
\]
Moreover, notice that any $z \in \z$ can be written as $r(p(z)) + i (l (z))$, since the sequence splits, therefore
\[
 p ([z_1,z_2]_\z) = p \Big( [r(p(z_1)),r(p(z_2))]_\z \Big) + p \Big( [r(p(z_1)),i(l(z_2))]_\z \Big) + p \Big( [i(l(z_1)),r(p(z_2))]_\k \Big) + p \Big( [i(l(z_1)),i(l(z_2))]_\z \Big),
\]
and, using the hypothesis $\alpha(v) = p ( [i(\alpha),r(v)]_\z )$ together with the fact that both $r$ and $i$ preserves the brackets,
\[
 p ([z_1,z_2]_\z) = [p(z_1),p(z_2)]_\k - l(z_2) (p(z_1)) + l(z_1) (p(z_2)).
\]

We conclude that, for all $z_1\in \z$, $z_2\in \z$
 \[
 \Big( l ([z_1, z_2]_\k), p  ([z_1, z_2]_\k) \Big) =
 \] 
  \[
  = \Big( [ l(z_1), l(z_2) ]_\h, [p (z_1), p(z_2)]_\k + l(z_1) (p(z_2)) - l(z_2) (p(z_1)) \Big) = \Big[(l(z_1), p(z_1)), (l(z_2), p(z_2)) \Big]_{\h \ltimes \k},
 \]
i.e.\ the linear isomorphism
\[
\z \to \h\ltimes \k, \quad \quad z \mapsto (l(z),p(z))
\]
is an isomorphism of almost Lie algebras.
\end{proof}

\begin{example} Definition~\ref{def:reductive} is, once more, inspired by Cartan geometries. More precisely:
\begin{itemize}
\item when $\h=0$, i.e.\ $(P,\theta)$ is a Cartan geometry, any Cartan-type extension is trivially reductive;
\item when $\h=\g$, i.e.\ $(P,\theta)$ is a $G$-structure, reductive Cartan-type extension $\z$ correspond to reductive model geometries $(\g,\z)$, see e.g.~\cite[Chapter 5, Definition 3.4.2]{SHARPE} or \cite{LOTTA}. \qedhere
\end{itemize}
\end{example} 

\begin{example}[Reducivity of higher order $G$-structures]\label{ex_higher_order_structures_reductive}
Recall from subsection \ref{subsection_higher_order_G_structures} that higher-order $G^k$-structures $P^k \subset \Fr^k (M)$, together with their tautological form $\theta^k \in \Omega^1 (P^k, V^k:= \g^{k-1} \times \RR^n)$, can be viewed as Cartan bundles with $\h=\g^k_{k-1}$. Then $(P^k,\theta^k)$ admits a Cartan-type extension which is always reductive; as usual, we prove this claim in details for the case $k=2$; the general case can be obtained by means of an inductive procedure.

First notice that, since $\g^1 \subset \Hom(\RR^n,\RR^n)$, $\g^1$ acts on $\RR^n$ by derivation, hence the space $\k:= \g^1 \times \RR^n$ has the following semidirect Lie algebra structure:
\[
 [(\alpha, v), (\beta, w)]_\k := \left( [\alpha,\beta]_{\g^1}, [v,w]_{\RR^n} + \alpha (w) - \beta (v) \right)
\]

In turn, $\g^2_1$ acts on $\k$ as well, even if not necessarily by derivations; then we take the semidirect product (in the sense of Remark \ref{rk_semidirect_product}) $\z:= \g^2_1 \ltimes (\g^1 \ltimes \RR^n) = \g^2_1 \ltimes \k$. Notice that, as a vector space, $\z$ is isomorphic to $\g^2 \times \RR^n$, since $\g^2/\g^2_1 \cong \g^1$, while, as Lie algebra, its bracket is given again by the semidirect product
\[
 [ (A, \alpha, v), (B, \beta, w) ]_{\z} := \left( [A,B]_{\g^2_1}, [(\alpha, v),(\beta, w)]_\k + A(w,\beta) - B (v,\alpha) \right).
\]
Here $A \in \g^2_1 \subset \Hom(\RR^n,\g^1)$ is interpreted as a linear map $\RR^n \times \g^1 \to \RR^n \times \g^1$ sending $(v,\alpha)$ to $(0,A(v))$, therefore the Lie bracket becomes
\[
 [ (A, \alpha, v), (B, \beta, w) ]_{\z} := \left( [A,B]_{\g^2_1},  [\alpha,\beta]_{\g^1} + A(w) - B(v), [v,w]_{\RR^n} + \alpha(w) - \beta(v) \right).
\]

Then the short exact sequence of $G^2$-representations
\[
 \begin{tikzcd}
0 \arrow[r] & \g^2_1 \arrow[r, "i", hook] & \g^2_1 \ltimes (\g^1 \ltimes \RR^n) \arrow[r, "p", two heads] & \g^1 \ltimes \RR^n \arrow[r] & 0
\end{tikzcd} 
\]
admits the left and right splittings
\[
 l: (A,\alpha,v) \mapsto A, \quad \quad r: (\alpha, v) \mapsto (0, \alpha,v).
\] 
It is immediate to check that $l$ preserves indeed the Lie algebra brackets. Last, one computes, for every $A \in \g^2_1$ and $(\alpha,v) \in \g^1 \ltimes \RR^n$,
\[
 p [i(A), r(\alpha,v)]_\z = p [ (A,0,0), (0,\alpha,v) ]_\z = p (0,A(v),0) = (A(v),0),
\]
which is precisely the representation of $\g^2_1$ on $\g^1 \ltimes \RR^n$.
%
\end{example}

\subsubsection{Curvature of reductive Cartan bundles}

The following proposition follows readily from the definition of curvature of a vector-valued 1-form and from Theorem \ref{characterisation_reductive_extensions}.


\begin{proposition}\label{properties_curvatures}
Let $(P, \theta)$ be a Cartan $(G, \h, V)$-bundle and let $\z$ be a reductive Cartan-type extension (Definition \ref{def:reductive}). 
For any $\z$-lift $\eta \in \Omega^1 (P,\z)$, which splits in $\eta = \tau + \theta$ as per part (iii) of Theorem \ref{properties_lifts}, the following identities holds:
\[
\Omega^\eta_{\h} = \Omega^\tau, \quad \quad \quad \Omega^\eta_{\k} = \Omega^\theta+\tau\wedge\theta.
\]
Here, $\Omega^\eta_{\h}$ and $\Omega^\eta_{\k}$ are the components of the curvature of $\eta$ (Definition \ref{def_a_flatness}) in $\h$ and $\k$ respectively, while $\Omega^\tau$ and $\Omega^\theta$ are the curvatures of $\tau$ and $\theta$ respectively.
\end{proposition}
By ``curvature'' of $\theta$ we mean the usual 2-form $\Omega^\theta := d\theta + \frac{1}{2} [\theta, \theta]_{\k} \in \Omega^2 (P,\k)$, analogously to the curvature of $\tau$ introduced in Definition \ref{def:h-conn}. The geometric meaning of $\Omega^\theta$ will become clearer in the application to the integrability of $G$-structures (see Theorem \ref{flatness_iff_integrability} below).

In the series of examples below, we discuss the various ways in which the curvatures appearing in Proposition~\ref{properties_curvatures} vanish.

\begin{example}[Realisations]
The experienced reader will recognise in $\Omega^\eta_\k = \Omega^\theta + \tau \wedge \theta$ (a version of) one of the famous Cartan's structure expression; see also the next subsection.

Indeed, if $\Omega^\eta_\k = 0$, then the Cartan's structure equation holds. Such equation, in this exact form, appears in~\cite[Definition 1.17]{CRAINICYUDILEVICH}, where $(P,\theta)$ is called a {\bf realisation}.

In \cite{CRAINICYUDILEVICH} the authors investigate Cartan's work thoroughly using the Pfaffian formalism. We stress that they work in greater generality than us, without assuming transitivity. Moreover, their discussion has deeper aims, and reaches deeper conclusions, than ours. On the other hand, the relative simplicity of the transitive case that we are concerned with allows us to discuss interesting examples while making explicit the geometric meaning of the objects involved.
\end{example}

\begin{example}[Flat $\h$-connections: horizontal Lie foliations]\label{ex_Lie_foliations}
If $\Omega^\eta_\h = 0$, then $\tau$ is a so-called {\bf Maurer-Cartan form on the manifold $P$}, i.e.\ a Lie algebra-valued 1-form satisfying $d\tau + \frac{1}{2}[\tau,\tau] = 0$, which is moreover pointwise surjective.

Kernels of such forms define a special kind of foliations, called {\bf Lie foliations}, introduced by Fedida \cite{Fedida71} and further developed by Molino \cite{Molino} (see also \cite[Section 4.3.1]{MOERDIJK}). In our case, $\ker(\tau) \subset TP$ is also $G$-equivariant and horizontal, i.e.\ transverse to the fibres of $\pi$.



If furthermore $\h = \g$, then $\tau$ is a flat principal connection on $P$. The leaves of $\ker(\tau)$ can be locally parametrised by sections of $P$ and the pullback of $\Omega^\eta_\k$ to $M$ via such sections coincides with the pullback of $\Omega^\theta$. Notice that, even when $\h \neq \g$, we still have $\Omega^\eta_\k = \Omega^\theta$ along the leaves of $\ker(\tau)$.
\end{example}

\begin{example}[$\k$-integrable $G$-structures]
If $\Omega^\eta = 0$, then both terms $\Omega^\eta_\k$ and $\Omega^\eta_\h$ vanish. If additionally $\k$ is a Lie algebra and $\h = \g$, we recover the $\k$-integrability of $G$-structures discussed in Example~\ref{ex:h-int-GS}. The pullback of $\Omega^\eta_\k$ vanishes for all $\tau$-flat sections if and only if the $G$-structure is $\k$-integrable; see also Theorem \ref{flatness_iff_integrability}. 
\end{example}

\begin{example}[Flat Cartan geometries]
If $\Omega^\eta = 0$ and the almost Lie algebra $\z$ is {\it Lie}, the pair $(P,\eta)$ is a flat Cartan geometry modelled on $(\z,\g)$. As explained e.g.\ in~\cite[Chapter 5, Section 5]{SHARPE} and \cite[Proposition 1.5.2]{CAPSLOVAK}, in this case $P$ carries a local Lie group structure. If the appropriate topological assumptions hold, then $P$ is a Lie group with Lie algebra given by $\z$, and the pair $(P,G)$ is a {\bf Klein geometry}, i.e.\ $M \cong P/G$ is a homogeneous space.\footnote{Notice that, in the literature of parabolic geometries, e.g.\ \cite{CAPSLOVAK}, it is common to denote Klein geometries by $(G,P)$, where $P$ is a parabolic subgroup of $G$. This is an unfortunate (and involuntary) clash of notations. For us $G$ is a subgroup of $P$, and furthermore we assume no parabolicity.}

If furthermore $\h = \g$, then the Klein geometry $(P,G)$ is {\it reductive} (see \cite[Chapter 4, Definition 3.2]{SHARPE}), i.e.\ the Lie algebra $\g$ admits a $G$-invariant complement in $\z$ (the image of $\k$ via the injective map $r:\k\to \z$).
\end{example}

\begin{example}[Vertical Lie foliations]
Finally, if $\Omega^\theta = 0$ and $\k$ is a Lie algebra, then $\theta$ defines a {\bf Lie foliation} on $P$ (see Example \ref{ex_Lie_foliations}), which is moreover $G$-equivariant and vertical.
\end{example}

\subsubsection{Reductivity and $\z$-flatness}

Thanks to Theorems \ref{properties_lifts} and \ref{characterisation_reductive_extensions}, the notion of reductivity can be applied to reinterpret the $\z$-flatness of important classes of Cartan bundles. Our main tools in this section will be the well known {\bf Cartan structure equations}, which are presented in various equivalent ways in the literature, see e.g.\ \cite[Chapter I]{Kobayashi95}, \cite[Chapter 2]{SingerSternberg65} and \cite[Chapter VII]{STERNBERG}. We recall them below to fix the notations we will use.

Let $P$ be a $G$-structure on $M^n$, with $\theta \in \Omega^1 (P,\k)$ its tautological form (where $\k=\RR^n$ as a vector space, with the choice of some -- not necessarily abelian -- {\it almost} Lie bracket) and $\tau \in \Omega^1 (P,\g)$ a connection 1-form. The standard Cartan structure equations are
\begin{equation*}
  \left\{\begin{array}{@{}l@{}}
     d \theta = T (\theta \wedge \theta) - \tau \wedge \theta \\
     d \tau = R (\theta \wedge \theta) - \tau \wedge \tau
  \end{array}\right.\,,
\end{equation*}
where $T: \wedge^2 \RR^n \to \RR^n$ and $R: \wedge^2 \RR^n \to \g$ denote, respectively, the torsion and the curvature functions of $\tau$. Here we adopt the following standard conventions (see also \cite[Remark 4.67]{Crainicnotes} for a more general treatment):
\[
 T(\theta \wedge \theta) (v,w) := T (\theta(v), \theta(w) ), \quad \quad R (\theta \wedge \theta) (v,w) := R (\theta(v), \theta(w) ).
\]
Similarly, using the evaluation map $\g \times \k \to \k$,
\[
(\tau \wedge \theta) (v,w) := \tau(v) (\theta(w)) - \tau(w) (\theta(v)) = - (\theta \wedge \tau)(v,w).
\]
On the other hand, for the Lie algebra bracket of $\g \times \g \to \g$ we adopt the notation
\[
[\tau, \tau]_\g (v,w) := [\cdot, \cdot]_\g (\tau, \tau) (v,w) = 2 [\tau(v),\tau(w)]_\g,
\]
and we also set, to get rid of the coefficients,
\[
(\tau \wedge \tau) (v,w) := \frac{1}{2} [\tau, \tau]_\g  (v,w) = [\tau(v), \tau(w)]_\g.
\]
We will use the analogous notation $[\theta,\theta]$ or $[\eta,\eta]$, with an appropriate (almost) Lie algebra structure on their coefficients.

\begin{theorem}\label{flatness_iff_integrability}
A $G$-structure $P$ is $\g \ltimes \k$-flat (as a Cartan bundle) if and only if there exists a connection $\tau$ such that
\begin{itemize}
 \item $R = 0$ (i.e.\ $\tau$ is flat);
 \item $T = -[ \cdot , \cdot ]_{\k}$ (i.e.\ the torsion coincides with (minus) the Lie bracket on $\k$).
\end{itemize}
\end{theorem}

Since, by Examples~\ref{ex:int-GS} and~\ref{ex:h-int-GS}, $\g \ltimes \k$-flatness was equivalent to the $\k$-integrability of a $G$-structure,
\begin{itemize}
\item for the abelian algebra $\k = \mathfrak{ab}$, we recover the standard result, that integrability is equivalent to the existence of a flat and torsion-free connection;
 \item for an arbitrary Lie algebra $\k$, we obtain a different characterisation of $\k$-integrability.
\end{itemize}

\begin{proof}
Fixing the Cartan-type extension $\z:= \g \ltimes \k$, the $\z$-flatness of $(P,\theta)$ prescribes the existence of a $\z$-lift $\eta \in \Omega^1 (P,\z)$ such that its curvature $\Omega^\eta$ vanishes. Since the Cartan-type extension $\z$ is reductive, Theorem \ref{properties_lifts} forces $\eta$ to be the sum of $\theta \in \Omega^1 (P,\k)$ with a $\h$-connection $\tau \in \Omega^1 (P,\g)$; but in this case $\h=\g$, so $\tau$ is an ordinary connection on $P$.

Since the semidirect product on $\g \ltimes \k$ is
\[
 [(\alpha,v), (\beta,w)]_{\g \ltimes \k} := \left( [\alpha,\beta]_\g, [v,w]_\k + \alpha (w) - \beta (v) \right).
\]
one computes (or uses Proposition \ref{properties_curvatures})
\[
\frac{1}{2} [\eta,\eta]_\z (v,w) = [ (\tau(v), \theta (v)) , (\tau (w), \theta (w)) ]_\z = \left( \tau \wedge \tau, \theta \wedge \theta + \tau \wedge \theta \right) (v,w)
\]
and, using the structure equations,
\[
 \Omega^\eta = d\eta + \frac{1}{2} [\eta, \eta]_\z = \left( R (\theta \wedge \theta), T (\theta \wedge \theta) + \theta \wedge \theta \right).
\]
We conclude that $\Omega^\eta = 0$ is equivalent to $R = 0$ and $T(v,w) = - [v,w]_\k$.
\end{proof}

\begin{remark}[Integrability in coordinates]\label{rk_integrability_coordinates}
The condition that $T$ coincides with the Lie bracket on $\k$ can be reinterpreted by saying that it is encoded by the structure constants of $\k$.

More precisely, if $(e_i)$ is a basis of $\k$ and $(A_\alpha)$ a basis of $\g$, with $c^\alpha_{\beta \gamma}$ structure constants of $\g$ and $d^i_{jk}$ structure constants of $\k$, then the structure equations are 
 \begin{equation*}
  \left\{\begin{array}{@{}l@{}}
     d \theta^i = T^i_{jk} (\theta^j \wedge \theta^k) - (A_\alpha)^i_j \tau^\alpha \wedge \theta^j \\
     d \tau^\alpha = R^\alpha_{ij} (\theta^i \wedge \theta^j) - c^\alpha_{\beta \gamma} \tau^\beta \wedge \tau^\gamma
  \end{array}\right.\,.
\end{equation*}
so that the same steps of the proof of Theorem \ref{flatness_iff_integrability} bring us to
\[
 (\Omega^\eta)^{i,\alpha} = \left( (T^i_{jk} + d^i_{jk}) \theta^j \wedge \theta^k, R^\alpha_{jk} \theta^j \wedge \theta^k \right).
\]
Then the flatness of $\eta$ means precisely $R^\alpha_{jk} = 0$ and $T^i_{jk} = - d^i_{jk} = d^i_{kj}$.
\end{remark}

Let us move now to a 2nd order $G^2$-structure $P^2$ on $M$. Since it is a Cartan bundle of order 2 (Example \ref{example_second_order_Cartan_bundles}), the discussion of Section \ref{section_higher_order_Cartan_bundles} allows us to view $P^2$ as a principal bundle on $P^1$ (equivalently, one obtains the same conclusion by interpreting $P^2$ as the first prolongation of $P^1$). Then a $\h$-connection $\tau^2$ on $P^2 \to M$ becomes just a principal connection on $P^2 \to P^1$ and we can write therefore the structure equations for $(P^2,\theta^2,\tau^2)$:



\begin{equation*}
  \left\{\begin{array}{@{}l@{}}
     d \theta^2 = T^2 (\theta^2 \wedge \theta^2) - \tau^2 \wedge \theta^2 \\
     d \tau^2 = R^2 (\theta^2 \wedge \theta^2) - \cancel{\tau^2 \wedge \tau^2}
  \end{array}\right.\,,
\end{equation*}
where the term $\tau^2 \wedge \tau^2$ vanishes since the Lie algebra $\g^1$ is abelian. 


 Moreover, considering $\tau^1 \in \Omega^1 (P^1,\g^1)$ from Lemma \ref{projection_h_connection}, we can use Remark \ref{first_component_theta_2} to split the tautological form $\theta^2$ as
 \[
  \theta^2 = (\pi^2_1)^*(\theta^1, \tau^1).
 \]
The first structure equation splits then in its $\k$- and $\g^1$-component, obtaining
\begin{equation*}
  \left\{\begin{array}{@{}l@{}}
     d \theta^1 = T^1 (\theta^1 \wedge \theta^1) - \tau^1 \wedge \theta^1 \\
     d \tau^1 = R^1 (\theta^1 \wedge \theta^1) + \tau^1 \wedge \tau^1 - \tau^2 \wedge \theta^1 \\
     d \tau^2 = R^2 (\theta^2 \wedge \theta^2)
     \end{array}\right.\,,
\end{equation*}
where $\tau^2 \wedge \theta^1 \in \Omega^2 (P^2, \g^1)$ is defined by
\[
 (\tau^2 \wedge \theta^1) (v,w) := \tau^2 (v) (\theta^1 (d\pi^2_1(w)) ) - \tau^2 (w) (\theta^1 (d \pi^2_1(v)) ).
\]

\begin{lemma}\label{projection_h_connection}
 Let $P^2$ be a $G^2$-structure. A $\h$-connection $\tau^2 \in \Omega^1 (P^2,\h)$ descends to an ordinary connection $\tau^1 \in \Omega^1 (P^1,\g^1)$ on the induced $G^1$-structure $P^1$, uniquely defined by
\[
(\pi^2_1)^*\tau^1 = d\pr \circ \tau^2, \quad \quad \text{ with }  d\pr: \g^2 \to \g^1.
\]
\end{lemma}
\begin{proof}
Since $P^2$ is reductive, we take as $\tau^1$ the $\h$-connection from Theorem \ref{properties_lifts}. Since the Cartan bundle $(P^1,\theta^1)$ has $\h=\g^1$, $\tau^1$ is an ordinary connection for the $G^1$-structure $P^1 \to M$. The properties of $\tau^2$ force $\tau^1$ to satisfy $(\pi^2_1)^*\tau^1 = d\pr \circ \tau^2$.
\end{proof}

\begin{proposition}\label{flatness_iff_integrability_higher_order}
A $G^2$-structure $P^2$ on $M$ is $\z^2:=\g^2_1 \ltimes \g^1 \ltimes \k$-flat (as a Cartan bundle) if and only if there is a $\h$-connection $\tau^2 \in \Omega^1 (P^2, \g^2_1)$ (since $\h=\g^2_1$ in this case) such that
\begin{itemize}
 \item $R^2 = 0$ (i.e.\ $\tau^2$ is flat);
 \item the induced 1st order $G^1$-structure $P^1$ is $\z^1:= \g^1 \ltimes \k$-flat
\end{itemize}
\end{proposition}


From our point of view $\k = \RR^n$ is always the abelian Lie algebra, but the same story can also be written for other Lie algebras, obtaining the analogue of $\k$-integrability also for higher order structures.

\begin{proof}
 The $\z^2$-flatness prescribes the existence of a $\z^2$-lift $\eta^2 \in \Omega^1 (P^2,\z^2)$ such that $\Omega^{\eta^2}=0$. Since the Cartan-type extension $\z^2$ is reductive (Example \ref{ex_higher_order_structures_reductive}), Theorem \ref{properties_lifts} forces $\eta^2$ to be the sum of a $\h$-connection $\tau^2 \in \Omega^1(P^2,\g^2_1)$ with $\theta^2 \in \Omega^1 (P^2, \z^1)$.
 
 Moreover, considering the connection $\tau^1 \in \Omega^1 (P^1,\g^1)$ induced from $\tau^2$ by Lemma \ref{projection_h_connection}, we have a $\z^1$-lift $\eta^1 = (\tau^1, \theta^1) \in \Omega^1 (P^1,\z^1)$ such that
 \[
\eta^2 = (\tau^2, (\pi^2_1)^*\eta^1) = \left( \tau^2, (\pi^2_1)^*(\tau^1, \theta^1) \right).
 \]
 Denoting $\bar{v} = d\pi^2_1(v)$ and $\bar{w} = d\pi^2_1(w)$, one computes (or uses Proposition \ref{properties_curvatures})
 \[
  \frac{1}{2} [\eta^2, \eta^2]_{\z^2} (v,w) = [ (\tau^2(v), \tau^1(\bar{v}), \theta^1(\bar{v}) ), (\tau^2(w), \tau^1(\bar{w}), \theta^1(\bar{w}) ) ]_{\z^2} = 
  \]
  \[ =  \left( \tau^2 \wedge \tau^2, \tau^1 \wedge \tau^1 + \tau^2 \wedge \theta^1, \theta^1 \wedge \theta^1 + \tau^1 \wedge \theta^1 \right) (v,w),
 \]
and, using the structure equations discussed above,
\[
 \Omega^{\eta^2} = d \eta^2 + \frac{1}{2} [\eta^2, \eta^2 ]_{\z^2} = \left( R^2 (\theta^2 \wedge \theta^2), \Omega^{\eta^1} \right).
\]
Then such a lift $\eta^2$ is flat if and only if both the $\h$-connection $\tau^2$ and the lift $\eta^1$ and are flat.
\end{proof}

We conclude by noticing that the main ideas of this section are not specific to $G$-structures (of any order). Indeed, one could adapt Lemma \ref{projection_h_connection} to an arbitrary Cartan bundle of order $k$, projecting on a Cartan bundle of order $k-1$ (see section \ref{section_higher_order_Cartan_bundles}), and use Theorem \ref{properties_lifts} to find out consequences and/or alternative characterisation of flatness with respect to reductive Cartan-type extensions.

\section{Back to transitive Lie pseudogroups}

We now go back to considering transitive Lie pseudogroups $\Gamma$. Thanks to the insights gained by looking at almost $\Gamma$-structures abstractly -- i.e., as Cartan bundles -- and to the machinery developed, we will be able to conceptually clarify certain aspects that are not immediately transparent.


We first recollect below useful notations and identities used repeatedly in this section.

\begin{remark}\label{properties_phi}
 Given a (free and transitive) $K$-action on $\BB$
 \[
  m^K: K \times \BB \to \BB, \quad \quad (k,x) \mapsto k \cdot x
 \]
we will denote the ``translation'' by $k \in K$ by
 \[
  T_k := m^K (k,\cdot):\BB\to \BB, \quad \quad x \mapsto k \cdot x,
 \]
 the orbit map of $x \in \BB$ by
\[
 \Phi_x := m^K (\cdot,x): K \to \BB, \quad \quad k \mapsto k \cdot x,
\]
 and the divisor map by
 \[
 \Phi: \BB \times \BB \to K, \quad \quad (y,x) \mapsto \Phi^{-1}_x(y),
\]
so that $\Phi(y,x) \cdot x = \Phi^{-1}_x(y) \cdot x = y$.


In the proofs we will use the following identities, for every $x,y,z \in \BB$,
 \[
  ( \Phi^{-1}_x (y) )^{-1} = \Phi^{-1}_y (x), \quad \quad \Phi^{-1}_y(z) \Phi^{-1}_x(y) = \Phi^{-1}_x(z),
\]
which can be checked directly using the divisor map.

Last, we compute the following differentials, for every $v_1 \in T_k K$, $v_2 \in T_x \BB$,
\[
d_{(k,x)}m^K (0_k,v_2) = d_x T_k (v_2), \quad \quad d_{(k,x)}m^K (v_1,0_x) = d_e\Phi_x (v_1),
\]
\[
 d_{(k \cdot x,x)}\Phi (d_{(k,x)} m^K (v_1,0_x), 0_x) = v_1,
\]
which can be checked by taking tangent curves.
\end{remark}

\subsection{Reductive extensions induced by free and transitive actions}


\begin{proposition}\label{transitive_pseudogroup_reductive_extension}
 Let $\Gamma$ be a transitive pseudogroup on $\BB$, with isotropy $G = (J^1 \Gamma)_{x_0}$, for an arbtitrary $x_0 \in \BB$.
 If $K$ is a Lie group acting freely and transitively on $\BB$, then the Pfaffian group $(G,\g,\k)$ 
 admits a reductive Cartan-type extension (Definition \ref{def:reductive})
 \[
\z = \g \ltimes \k, \quad \quad \g = \Lie(G), \quad \k = \Lie (K).
 \] 
\end{proposition}

\begin{proof}
Applying the orbit-stabiliser theorem to the free and transitive smooth $K$-action on $\BB$ and the point ${x_0} \in \BB$, the orbit map (see Remark \ref{properties_phi})
\[
 \Phi_{x_0}: K \to \BB, \quad k \mapsto k \cdot x_0,
\]
becomes a diffeomorphism. Then, the following defines a smooth action of $G$ on $K$:
\[
 G \times K \to K, \quad (j^1_{x_0} \phi, k) \mapsto \Phi_{x_0}^{-1} (d_{x_0} \phi (\Phi_{x_0} (k))).
\]
Indeed, such map is well defined since $d_{x_0} \phi$ depends only on its 1-jet at ${x_0}$, and it is a group action since
\[
 e_G \cdot k = j^1_{x_0} (\id_{\RR^n}) \cdot k = \Phi_{x_0}^{-1} ( d_{x_0} \id_{\RR^n} ( \Phi_{x_0} (k))) = k
\]
and
\[
 j^1_{x_0} \phi_1 \cdot (j^1_{x_0} \phi_2 \cdot k) = j^1_{x_0} \phi_1 \cdot \Phi_{x_0}^{-1} (d_{x_0} \phi_2 (\Phi_{x_0} (k))) = \Phi_{x_0}^{-1} (d_{x_0} \phi_1 \circ \Phi_{x_0} \circ \Phi_{x_0}^{-1} \circ d_{x_0} \phi_2 (\Phi_{x_0} (k) ) ) = 
\] 
\[ 
 = \Phi_{x_0}^{-1} (d_{x_0} (\phi_1 \circ \phi_2) (k) ) = j^1_{x_0} (\phi_1 \circ \phi_2) \cdot k = (j^1_{x_0} \phi_1 \cdot j^1_{x_0} \phi_2) \cdot k.
\]

By differentiating this action, one obtains a Lie algebra action (not necessarily by derivations!) of $\g$ on $\k$. Hence, one can consider the semidirect product $\z = \g \ltimes \k$ in the sense of Remark \ref{rk_semidirect_product}. Then $\z$ is a Cartan type extension for the Pfaffian group $(G,\g,\k)$ by part (i) of Theorem \ref{characterisation_reductive_extensions}. 
\end{proof}

Proposition~\ref{transitive_pseudogroup_reductive_extension} asserts that free and transitive actions on the unit manifold $\BB$ induce reductive Cartan-type extensions for the isotropy group of the first jet groupoid $J^1\Gamma$ of a transitive pseudogroup $\Gamma$ on $\BB$. In the rest of this section, we explore the case when the group $K$ acting freely and transitive on $\BB$ is {\it contained} in $\Gamma$, in the sense of Definition~\ref{def:generated-pseudogroup}.

\begin{example}\label{ex_Cartan_ext_Gamma_G}
In the case of the pseudogroup $\Gamma_G$ on $\RR^n$ (Example \ref{ex:G-struct}), the Lie group $K = (\RR^n,+)$ is contained in $\Gamma_G$ (Example \ref{ex_gamma_G_contains_Rn}). We have $T_k (x) = k+x$; then, choosing $x_0 = 0$, the diffeomorphism $\Phi_{x_0}$ from Proposition~\ref{transitive_pseudogroup_reductive_extension} is just the identity, and we recover the $G$-action of invertible matrices on vectors in $\RR^n$ and the reductive extension discussed in Example \ref{ex:int-GS}.
\end{example}

\begin{example}\label{ex_Cartan_ext_Gamma_cont}
In the case of the pseudogroup $\Gamma_{\rm cont}$ on $\RR^{2k+1}$ (Example \ref{ex:cont-struct}), the Lie group $K = \Hei_{2k+1} \subset \GL(k+2,\RR)$ is contained in $\Gamma_{\rm cont}$ (Example \ref{ex_gamma_cont_contains_heis}). Then, if
\[
 k = \begin{pmatrix}
1 & a & c\\
0_k & I_{k,k} & b\\
0 & 0_k & 1
\end{pmatrix} \in K, \quad \quad \text{ with } a,b \in \RR^k, c \in \RR,
\]
we can interpret $x_0 = ({\bf x},{\bf y},z) \in \RR^{2k+1}$ as a matrix
\[
 ({\bf x},{\bf y},z) = \begin{pmatrix}
1 & {\bf x} & z \\
0_k & I_{k,k} & {\bf y} \\
0 & 0_k & 1
\end{pmatrix},
\]
so that $\Phi (k) = k \cdot ({\bf x},{\bf y},z)$ is the standard matrix multiplication. Setting $x_0 = ({\bf x},{\bf y},z) = 0$ we find, of course
\[
 \Phi_{x_0} (k) = (a,b,c),
\]
which recovers the reductive extension discussed in Example \ref{ex_integrability_contact_structures}.
\end{example}

The discussion above shows the origin of two different integrabilities of an $\Sp(k,1)$-structure:
\begin{itemize}
 \item the classical $G$-structure integrability, i.e.\ $\g \ltimes \mathfrak{ab}$-flatness, corresponding to codimension 1 symplectic foliations; 
 \item the integrability to a contact structure, i.e.\ $\g \ltimes \mathfrak{hei}$-flatness, corresponding to contact structures. 
\end{itemize}
In both cases, the ``right'' Cartan-type extension to consider arises canonically from the Lie groups contained into the pseudogroups underlying the structure, i.e.\ respectively $\Gamma_{\Sp(k,1)}$ and $\Gamma_{\rm cont}$.

\subsection{Jets of pseudogroups containing a Lie group --  the groupoid structure}\label{section_groupoid_structure}

Below, we will investigate the general structure of the first jet groupoids of arbitrary pseudogroups containing a Lie group (Definition~\ref{def:generated-pseudogroup}). We begin by showing how this property impacts on the global structure of $J^0\Gamma$ and $J^1\Gamma$.






\begin{lemma}\label{lem:isomorphism_source_zero_jet}
Let $\Gamma$ be a transitive Lie pseudogroup on $\BB$. If $K$ is a Lie group acting freely and transitively on $\BB$, then it induces a canonical Lie groupoid isomorphism 
\[
J^0\Gamma \cong K\ltimes \BB.
\]
\end{lemma}
\begin{proof}
It is enough to compose the standard Lie groupoid isomorphism
\[
 J^0 \Gamma \to \BB \times \BB, \quad \quad j^0_x \varphi \mapsto (\varphi(x),x),
\]
induced by the transitivity of $\Gamma$, with the inverse of the Lie groupoid isomorphism
\[
K\ltimes \BB \to \BB\times \BB,\quad \quad (k,x) \mapsto (k\cdot x,x),
\]
induced by the free and transitive action of $K$.
\end{proof}

Now, recall that, for any transitive Lie groupoid $\G \tto \BB$, the source-fibre $s^{-1}(x_0)$ at some $x_0 \in \BB$ defines a right principal $G$-bundle $t: s^{-1}(x_0) \to \BB$, for $G = \G_{x_0}$ the isotropy group at $x_0$.


\begin{lemma}\label{lem:source_trivial}
Let $\Gamma$ be a transitive Lie pseudogroup on $\BB$ containing a Lie group $K$ (Definition~\ref{def:generated-pseudogroup}). For any $x_0 \in \BB$, the principal $G$-bundle $s^{-1}(x_0) \to \BB$ admits a canonical global trivialisation induced by $K$. 
\end{lemma}

\begin{proof}
It is enough to check that the following is a global section
\[
\BB\to s^{-1}(x_0),\quad \quad x \mapsto j^1_{x_0} T_{\Phi(x,x_0)},
\]
where $\Phi$ is the divisor from Remark \ref{properties_phi}.
\end{proof}

It follows from the lemmas above that

\begin{proposition}\label{prop:isomorphism_source_first_jet}
Let $\Gamma$ be a transitive Lie pseudogroup on $\BB$ containing a Lie group $K$ (Definition~\ref{def:generated-pseudogroup}). For every $x_0 \in \BB$ there is a Lie groupoid isomorphism
\[
J^1\Gamma \cong (K\ltimes \BB) \times G,
\]
where $G = (J^1 \Gamma)_{x_0}$ is the isotropy group at $x_0$.
\end{proposition}

Here $(K\ltimes \BB)\times G \tto \BB$ denotes the direct product between the action groupoid $K\ltimes \BB \tto \BB$ and the group bundle $G\times \BB \tto \BB$. More explicitly, its multiplication is
\[
(k_2,k_1\cdot x,g_2)\cdot (k_1,x,g_1) = (k_2k_1,x,g_2g_1),
\]
while source and target are given by the source and target of $K\ltimes \BB$:
\[
s: (k,x,g)\mapsto x, \quad \quad t: (k,x,g) \mapsto k\cdot x.
\] 

\begin{proof}
Recall that, since the Lie groupoid $J^1\Gamma \tto \BB$ is transitive, it is isomorphic to the gauge groupoid
\[
\Gaug (s^{-1}(x_0) ) :=  \Big( s^{-1}(x_0)\times s^{-1}(x_0) \Big) /G
\]
of its $s$-fibre $s^{-1}(x_0)$. From Lemma~\ref{lem:source_trivial} it follows that
\[
\Gaug (s^{-1}(x_0) ) \cong \Big( (\BB \times G) \times (\BB \times G) \Big) / G \cong (\BB\times \BB)\times G,
\]
where the right-hand side is the direct product of the pair groupoid of $\BB$ with the bundle of Lie groups $G\times \BB\tto \BB$. The claim follows from the proof of Lemma~\ref{lem:isomorphism_source_zero_jet}.
\end{proof}

The approach with the gauge groupoid allows us to avoid long computations. However, in order to investigate the Pfaffian structure of $J^1 \Gamma$, we need to write down an explicit formula for an isomorphism 
in terms of the point $x_0\in \BB$ and the action of $K$ on $\BB$: 
\[
\Psi_{x_0}: J^1\Gamma \to (K\ltimes \BB)\times G, \quad \quad 
j^1_x\varphi \mapsto \Big( \Phi_x^{-1}(\varphi(x)) , x, d_{x_0} ( T^{-1}_{\Phi^{-1}_{x_0}(\varphi(x))} \circ \varphi \circ T_{\Phi^{-1}_{x_0}(x)}) \Big).
\]
The proof that $\Psi_{x_0}$ is a well defined bijection which preserves the groupoid multiplication is lengthier and it involves the properties of $\Phi_{x_0}^{-1}$ from Remark \ref{properties_phi}.

\begin{remark}\label{rk_formula_isomorphism}
 In the same spirit of the proof of Lemma \ref{lem:isomorphism_source_zero_jet}, the formula written above for $\Psi_{x_0}$ is better understood as the composition of the Lie groupoid isomorphism (depending on $x_0$)
 \[
  \tilde{\Psi}_{x_0}: J^1 \Gamma \to (\BB \times \BB) \times G, \quad \quad j^1_x \varphi \mapsto (\varphi(x),x,d_{x_0} ( T^{-1}_{\Phi^{-1}_{x_0}(\varphi(x))} \circ \varphi \circ T_{\Phi^{-1}_{x_0}(x)})),
 \]
with the {\it canonical} Lie groupoid isomorphism
\[
 \Theta: (\BB \times \BB) \times G \to (K \ltimes \BB) \times G, \quad \quad (y,x,g) \mapsto (\Phi^{-1}_x(y),x,g).
\]
 This point of view will be used in the next section.
\end{remark}

\subsection{Jets of pseudogroups containing a Lie group -- the Pfaffian structure}

We examine now the natural Pfaffian structure of $(K \ltimes \BB) \times G$ and relate it to the original one of of $J^1 \Gamma$ coming from jet bundles (Example \ref{exm:jet-groupoids}). 

\begin{proposition}\label{prop_Pfaffian_isomorphism_zero_jet}
Let $K$ be a Lie group acting freely and transitively on $\BB$. Then the Lie groupoid $K \ltimes \BB \tto \BB$ is a full Pfaffian groupoid when equipped with the Pfaffian form (in fact, a flat multiplicative connection)
\[
d\pr_1 \in \Omega^1(K\ltimes \BB, T^{\rm vert}(K\ltimes \BB)),\quad d_{(k,x)}\pr_1 (v_1,v_2) = v_1.
\]

Consider now a Lie subgroup  $G$ of $\GL(T_{x_0}\BB)$, for some $x_0 \in \BB$. Then the Lie groupoid $(K \ltimes \BB) \times G \tto \BB$ is a full Pfaffian groupoid, together with the representation on $(\Phi^{-1}_{x_0})^*TK \to \BB$
 \[
  (K \ltimes \BB \times G) \times (\Phi^{-1}_{x_0})^*TK \to (\Phi^{-1}_{x_0})^*TK, \quad \quad ((k,x,A), v) \mapsto \bar{A}(v),
 \]
 where
 \[
  \bar{A} := d_{x_0} \Phi^{-1}_{\Phi^{-1}_{k \cdot x} (x_0) \cdot x_0 } \circ A \circ d_{\Phi^{-1}_{x_0}(x)} \Phi_{\Phi^{-1}_x(x_0) \cdot x_0}: T_{\Phi^{-1}_{x_0}(x)} K \to T_{\Phi^{-1}_{x_0}(k \cdot x)} K,
 \]
and the differential form
%
%
\[
 \omega \in \Omega^1 (K \ltimes \BB \times G, T^{\rm vert}(K \ltimes \BB)), \quad \quad \omega_{(k,x,A)}(v_1,v_2,v_3) := d_{(k \cdot x, x)} \Phi \Big(d_{(k,x)}m^K (v_1, v_2) - \tilde{A}(v_2) , 0_x \Big),
\]
where
\[
\tilde{A}:= d_{x_0} T_{\Phi^{-1}_{x_0}(k \cdot x)} \circ A \circ d_x T^{-1}_{\Phi^{-1}_{x_0}(x)}: T_x \BB \to T_{k \cdot x}\BB,
\]
using the notations introduced in Remark \ref{properties_phi}.
\end{proposition}
\begin{proof}
For the first part it is enough to check that the kernel of $d\pr_1$ is closed under multiplication. Moreover, the symbol bundle (Definition \ref{def_Pfaffian_groupoid-form}) vanishes, i.e.\ the kernel of $d\pr_1$ is complementary to the kernel of $ds$, and defines therefore a flat (multiplicative) connection.

The second part is a long but straightforward computation.
\end{proof}


\begin{proposition}\label{prop_holonomic_bisections_jets}
In the setting of Proposition \ref{prop_Pfaffian_isomorphism_zero_jet}, the following are holonomic bisections (Definition \ref{def_Pfaffian_groupoid-form}) of
$((K \ltimes \BB) \times G \tto \BB, \omega)$:
\[
 \sigma_k: \BB \to (K \ltimes \BB) \times G, \quad \quad x \mapsto (k,x,e_G) \quad \quad \forall k \in K.
\]
Similarly, the following are holonomic bisections of $(K \ltimes \BB \tto \BB, d\pr_1)$:
\[
 \pr \circ \sigma_k: \BB \to K \ltimes \BB, \quad \quad x \mapsto (k,x) \quad \quad \forall k \in K.
\]
\end{proposition}

\begin{proof}
Everything is immediate besides the proof that $\sigma_k$ is holonomic for every $k \in K$. We notice first that, if $A = e_G$, then
 \[
\tilde{A} = d_{x_0} T_{\Phi^{-1}_{x_0}(k \cdot x)} \circ d_x T^{-1}_{\Phi^{-1}_{x_0}(x)} = d_x T_k
\]
where we used the identities for $\Phi^{-1}$ from Remark \ref{properties_phi} 
Using also the identity for $d_{(k,x)} m^K (0,v) = d_x T_k (v)$ from the same remark, we conclude that 
\[
 (\sigma_k^*\omega)_x(v) = \omega_{(k,x,e_G)} (0,v,0) = d_{(k \cdot x, x)} \Phi \Big( d_{(k,x)} m^K (0,v) - \tilde{A}(v), 0_x \Big) = 0. \qedhere
\]
\end{proof}

Observe that, given a transitive Lie pseudogroup $\Gamma$, the holonomic bisections of $(K \ltimes \BB, d\pr_1)$ are not necessarily 0-jets of elements of $\Gamma$, under the isomorphism $J^0\Gamma \cong K\ltimes \BB$ from Lemma~\ref{lem:isomorphism_source_zero_jet}. However, this happens precisely when $\Gamma$ contains $K$. On the other hand, for $((K \ltimes \BB) \times G, \omega)$ we need such hypothesis already to be able to write the isomorphism $J^1 \Gamma \cong (K \ltimes \BB) \times G$ from Lemma \ref{prop:isomorphism_source_first_jet}. 

\begin{proposition}\label{prop_Pfaffian_isomorphism_first_jet}
 Let $\Gamma$ be a transitive Lie pseudogroup on a manifold $\BB$, which contains a Lie group $K$ (Definition~\ref{def:generated-pseudogroup}), and let $G$ be the isotropy group of $J^1 \Gamma$ at $x_0$. Then one can promote the isomorphism
 \[
 \Psi_{x_0}: J^1\Gamma \to (K\ltimes \BB)\times G  
 \]
 given below Proposition~\ref{prop:isomorphism_source_first_jet}  to an isomorphism of Pfaffian groupoids.
 
Moreover, the canonical projection 
\[
\pr: (K\ltimes \BB)\times G \to K\ltimes \BB, \quad \quad (k,x,g) \mapsto (k,x)
\]
is a groupoid morphism, and its canonical section
\[
\sigma: K\ltimes \BB \to (K\ltimes \BB)\times G,\quad (k,x) \mapsto (k,x,e)
\]
is a Pfaffian groupoid morphism (with respect to the structures of Proposition~\ref{prop_Pfaffian_isomorphism_zero_jet}).
\end{proposition}


Indeed, as already shown in Proposition \ref{prop_holonomic_bisections_jets}, holonomic bisections of $K \ltimes \BB$ are sent by $\sigma$ to holonomic bisections of $(K \ltimes \BB) \times G$ (and this would have been enough to conclude that $\sigma$ is a Pfaffian morphism, since the Pfaffian form on $K \ltimes \BB$ is a flat multiplicative connection).

\begin{proof}
 One checks first that the isomorphism $\tilde{\Psi}_{x_0}: J^1 \Gamma \to (\BB \times \BB) \times G$ (from Remark \ref{rk_formula_isomorphism}) transports the Cartan form $\omega^1$ of $J^1 \Gamma \subset J^1 (\BB,\BB)$ (see Section \ref{subs:cart-dist} for the explicit expression) to the following Pfaffian form $\tilde{\omega}$ on $(\BB \times \BB) \times G$:
 \[
  \tilde{\omega}_{(k \cdot x,x,A)}(v_1,v_2,v_3) := (v_1 - \tilde{A}(v_2), 0).
\] 
 In turn, when further composing with the isomorphism
\[
 \Theta: (\BB \times \BB) \times G \to (K \ltimes \BB) \times G, \quad \quad (y,x,g) \mapsto (\Phi(x,y),x,g),
\]
one recovers precisely $\omega$, i.e.\ $\Theta^* \omega = d\Phi \circ \tilde{\omega}$. The first claim follows from $\Psi_{x_0} = \Theta \circ \tilde{\Psi}_{x_0}$.
%
%

For the second part, it is straighforward to see that both $\pr$ and $\sigma$ are groupoid morphism, so it is enough to check that $\sigma$ preserves the Pfaffian forms:
\[
 (\sigma^*\omega)_{(k,x)} (v_1,v_2) = d_{(k \cdot x, x)} \Phi \Big(d_{(k,x)}m^K (v_1,v_2) - d_x T_k (v_2), 0_x \Big) =
 \]
 \[
 = d_{(k \cdot x, x)} \Phi \Big(d_{(k,x)}m^K (v_1,v_2) - d_{(k,x)} m^K (0_k,v_2), 0_x \Big) =
 \]
 \[
 = d_{(k \cdot x, x)} \Phi \Big( d_{(k,x)}m^K (v_1,0_x), 0_x \Big) = v_1 = d_{(k,x)} \pr_2 (v_1,v_2),
\]
where we used the equalities from Remark \ref{properties_phi}.
\end{proof}


To conclude, we think that already the (groupoid-theoretical!) structure emerging from subsection \ref{section_groupoid_structure} deserves a separate investigation. Furthermore, we believe that the {\it Pfaffian} structure emerging from this subsection is deeply related to the flatness introduced in subsection \ref{section_integrability_Cartan_bundles}, and can therefore be exploited to obtain further results on geometric structures.

We postpone further investigations to \cite{INPROGRESS}.

\addcontentsline{toc}{section}{References}

\printbibliography

\end{document}